\documentclass[a4paper, 11pt]{article}

\usepackage[a4paper, total={6in, 8in}]{geometry}

\usepackage{mathtools, amsmath, amsfonts, amssymb}
\usepackage{stmaryrd}
\usepackage{amsthm}
\usepackage[table]{xcolor}
\usepackage[Tol]{colorblind}
\usepackage{enumitem}
\usepackage[
    backend=biber,
    style=alphabetic,
    sorting=nyt,
    doi=false,
    url=false,
    isbn=false,
    maxbibnames=10,
]{biblatex}
\usepackage{tikz}
\usepackage{tikz-cd}
\usepackage{graphicx}
\usepackage{caption}
\usepackage{subcaption}
\usepackage{hyperref}
\usepackage{cleveref}
\usepackage{supertabular}
\usepackage[Export]{adjustbox}

\theoremstyle{plain}
    \newtheorem{thm}{Theorem}
    \newtheorem{lem}[thm]{Lemma}
    \newtheorem{prop}[thm]{Proposition}

\theoremstyle{definition}
    \newtheorem{defn}[thm]{Definition}

\theoremstyle{remark}
    \newtheorem*{rmk}{Remark}

\DeclareMathOperator{\id}{id}

\DeclareMathOperator{\CKh}{CKh}

\DeclareMathOperator{\eqKh}{\widetilde{Kh}}
\DeclareMathOperator{\CBN}{CBN}
\DeclareMathOperator{\BN}{BN}
\DeclareMathOperator{\eqCBN}{\widetilde{CBN}}
\DeclareMathOperator{\eqBN}{\widetilde{BN}}

\newcommand{\bn}[1]{\llbracket#1\rrbracket}
\newcommand{\eqbn}[1]{\widetilde{\llbracket#1\rrbracket}}

\DeclareMathOperator{\Cob}{Cob}
\DeclareMathOperator{\Cobd}{Cob_\bullet}
\newcommand{\Cobdl}{\Cobd_{/l}}
\DeclareMathOperator{\Mat}{Mat}
\DeclareMathOperator{\Kom}{Kom}
\newcommand{\Komh}{\Kom_{/h}}

\DeclareMathOperator{\Kobd}{Kob_\bullet}
\newcommand{\Kobdh}{\Kobd_{/h}}
\DeclareMathOperator{\Hom}{Hom}

\DeclareMathOperator{\Cone}{Cone}

\newcommand{\bC}{\mathbf{C}}

\newcommand{\FF}{\mathbb{F}}
\newcommand{\RR}{\mathbb{R}}
\newcommand{\ZZ}{\mathbb{Z}}

\newcommand{\cC}{\mathcal{C}}

\newcommand{\ii}{\mathbf{1}}
\newcommand{\xx}{\mathbf{x}}

\newcommand{\eqC}{\widetilde{\bC}}
\newcommand{\equ}{\widetilde{u}}
\newcommand{\ord}{\mathrm{ord}}
\newcommand{\eqord}{\widetilde{\mathrm{ord}}}

\newcommand{\eqf}{\widetilde{f}}
\newcommand{\eqg}{\widetilde{g}}
\newcommand{\eqh}{\widetilde{h}}

\newcommand{\seq}{\widetilde{sg}_4}
\newcommand{\us}{\overline{s}}
\newcommand{\ls}{\underline{s}}

\newcommand{\gz}{\textcolor{gray}{0}}

\title{Equivariant Unknotting Number and\\ Involutive Khovanov Homology}
\author{KeeTaek Kim}
\date{} % last update: April 10th

\begin{document}
    \maketitle
    \begin{abstract}
    We demonstrate that the equivariant unknotting number $\equ(K)$ of a strongly invertible knot $K$ is bounded below by the $H$-torsion order $\eqord(K)$ of the involutive Bar-Natan homology $\eqBN(K)$.
    This result serves as an equivariant analogue to the bound established by Alishahi \cite{Alishahi:2019}.
    As an application, we identify five strongly invertible prime knots with crossing numbers at most $9$ for which the strict inequality $u(K) < \equ(K)$ holds.
\end{abstract}

    \section{Introduction}\label{sec:intro}

In 2000, Mikhail Khovanov introduced a bigraded link homology theory that categorifies the Jones polynomial \cite{Khovanov:2000}, now widely known as Khovanov homology.
Subsequently, Bar-Natan developed a deformed version of this theory \cite{Bar-Natan:2005}. While the Khovanov theory associates a link $K$ with a chain complex $\CKh(K)$ over a unital commutative ring $R$, the Bar-Natan theory utilizes a chain complex $\CBN(K)$ over the polynomial ring $R[H]$.

Khovanov and Bar-Natan homologies have become essential tools in low-dimensional topology and knot theory. For instance, when $R$ is a field, the Rasmussen $s$-invariant---derived from the quantum grading of the free part of Bar-Natan homology---was used to provide a combinatorial proof of the Milnor conjecture \cite{Rasmussen:2010} and to show that the Conway knot is not slice \cite{Piccirillo:2020}.

The torsion part of Bar-Natan homology also yields significant geometric information.
Any knot $K$ can be transformed into the unknot $U$ via an isotopy that allows for self-intersections at discrete times.
The \emph{unknotting number} $u(K)$ is defined as the minimum number of such self-intersections required to unknot $K$. In \cite{Alishahi:2019}, Akram Alishahi demonstrated that the maximal $H$-torsion order $\ord(K)$ of the Bar-Natan homology $\BN(K)$ with $\FF=\ZZ/2\ZZ$ coefficients provides a lower bound for the unknotting number $u(K)$.
This work also has been extented to the rational unknotting number; see \cite{Lewark-Marino-Zibrowius:2024, Iltgen-Lewark-Marino:2025}.

This paper extends Alishahi's work to the equivariant setting.
A \emph{strongly invertible knot} is a knot $K$ equipped with an orientation-preserving involution $\tau$ on $S^3$ such that $\tau(K) = r(K)$, where $r(K)$ is the reverse of $K$.
For a strongly invertible knot $(K,\tau)$, we consider the \emph{equivariant unknotting number} $\equ(K)$, which is the minimum number of self-intersections required during a $\tau$-invariant unknotting process; see \Cref{subsec:equ} for more details.
In \cite{Sano:2025}, Taketo Sano developed the involutive versions of Khovanov and Bar-Natan homology for strongly invertible knots, which we denote by $\eqKh(K)$ and $\eqBN(K)$ respectively, and refer to as \emph{involutive Khovanov homology} and \emph{involutive Bar-Natan homology}\footnote{In other literature, the notations $\eqKh(K)$ and $\eqBN(K)$ are frequently used to denote \emph{reduced} version of Khovanov and Bar-Natan homologies of $K$. In this paper, we do not treat the reduced theories; instead, equivariant invariants are denoted with a tilde.}.
By defining a maximal $H$-torsion order $\eqord(K)$ for $\eqBN(K)$, we establish the following analogous theorem:

\begin{thm}\label{thm:main}
    For a strongly invertible knot $K$, $\eqord(K) \leq \equ(K)$.
\end{thm}

We apply \Cref{thm:main} to investigate lower bounds for the equivariant unknotting numbers of strongly invertible prime knots with crossing numbers at most $9$ in Appendix.
Among these, we identify five examples where the equivariant unknotting number is strictly greater than the ordinary unknotting number.
We should note that these results can also be recovered via alternative methods from \cite{Boyle-Chen:2026}; see \Cref{prop:mainex} and the subsequent remarks for a detailed comparison.

The remainder of this paper is organized as follows.
In \Cref{sec:pre}, we provide a brief introduction to Bar-Natan homology, involutive Bar-Natan homology, and the equivariant unknotting number.
In \Cref{sec:thm}, we present the proof of \Cref{thm:main}.
Finally, in \Cref{sec:ex}, we exhibit five examples of prime strongly invertible knots with crossing numbers at most 9 for which the strict inequality $u(K) < \equ(K)$ holds.

\begin{center}\textbf{Acknowledgements}\end{center}
The author wishes to thank his advisor, JungHwan Park, for his constant support and guidance.
The author thanks Taketo Sano for helpful comments and for providing his computer program.
The author also thanks Wenzhao Chen for helpful suggestion; see the remark at the end of \Cref{subsec:typeCproof}.
The author is partially supported by the Samsung Science and Technology Foundation (SSTF-BA2102-02) and the NRF grant RS-2025-00542968.

    \section{Background}\label{sec:pre}

In this section, we briefly review the categories $\Kobd$ and $\Kobdh$, the formal Bar-Natan complex $\bn{K}$ for an involutive link $(K,\tau)$, and the definition of the involution $I_\tau$ on $\bn{K}$.

First, we establish some terminology used throughout this paper.
An \emph{involutive link} $(K,\tau)$ is a link $K$ equipped with an orientation-preserving involution $\tau$ on $S^3$ that fixes $K$ setwise\footnote{We will frequently omit the involution $\tau$ when representing an involutive link.}.
The involution $\tau$ may either preserve or reverse the orientation of each component.
In particular, if $K$ is a knot and $\tau$ reverses its orientation, $(K, \tau)$ is called a \emph{strongly invertible knot}.

The \emph{Smith conjecture}, which was proven to be true \cite{Waldhausen:1969, Morgan-Bass:1979}, states that the fixed-point set $\mathrm{Fix}(f)$ of a nontrivial orientation-preserving diffeomorphism $f \colon S^3 \to S^3$ of finite order must be an unknotted circle in $S^3$.
By virtue of this result, we may assume without loss of generality that the involution $\tau$ of an involutive link $(K,\tau)$ is given by a $180$-degree rotation about the $y$-axis (viewing $S^3$ as the one-point compactification of $\RR^3$). 
In this setting, a knot diagram projected onto the $xy$-plane is called a \emph{transvergent diagram}. 
We refer to the set of fixed points of the rotation, $\{y\text{-axis}\}=\mathrm{Fix}(\tau)$, as the \emph{axis} of the involutive link $(K,\tau)$.

Note that if the orientation of a component of an involutive link $(K,\tau)$ is reversed by $\tau$, that component must intersect the axis at exactly two points.
These points are pointwise fixed under $\tau$.
When we refer to the \emph{fixed points} of an involutive link, we mean the collection of these intersection points.

\subsection{Formal Bar-Natan complex}

In this subsection, we associate a link $K$ with a chain complex $\bn{K}$ in the category $\Kobdh$.
We begin by constructing the underlying category $\Kobdh$.

\begin{defn}
    The \emph{dotted cobordism category} $\Cobd$ is defined as follows:
    \begin{itemize}
        \item \textbf{Objects:} $1$-manifolds; i.e., disjoint unions of circles.
        \item \textbf{Morphisms:} A morphism from $O$ to $O'$ is a dotted cobordism.\footnote{We depict a morphism $\Sigma \colon O \to O'$ vertically, with the source $O$ at the top and the target $O'$ at the bottom. In equations, we follow standard functional composition; i.e., $gf$ signifies applying $f$ first, then $g$.} A \emph{dotted cobordism} is an orientable surface $\Sigma$ with boundary $\partial \Sigma = (-O) \sqcup O'$, decorated with finitely many dots. The surface may have multiple components, and dots can move freely within each component. Two dotted cobordisms are considered equivalent if they are homeomorphic relative to the boundary (preserving dot locations up to homeomorphism).
        \item \textbf{Composition:} Composition is defined by the concatenation of cobordisms.
    \end{itemize}
\end{defn}

An example of a dotted cobordism is shown in \Cref{fig:cobd_ex}.
We distinguish between individual circles within an object; thus, the two cobordisms in \Cref{fig:cobd_ex} are treated as distinct.

\begin{figure}
    \centering
    \begin{subfigure}{.48\linewidth}
        \centering
        \tikzset{every picture/.style={line width=0.75pt}} %set default line width to 0.75pt        

\begin{tikzpicture}[x=0.75pt,y=0.75pt,yscale=-1,xscale=1]
%uncomment if require: \path (0,120); %set diagram left start at 0, and has height of 120

%Curve Lines [id:da27444758309116235] 
\draw    (5,30) .. controls (5.2,20) and (35.2,20) .. (35,30) ;
%Curve Lines [id:da5367545611385904] 
\draw    (5,30) .. controls (5.2,40) and (35.2,40) .. (35,30) ;
%Curve Lines [id:da6756706808217197] 
\draw    (45,30) .. controls (45.2,20) and (75.2,20) .. (75,30) ;
%Curve Lines [id:da29314707012960795] 
\draw    (45,30) .. controls (45.2,40) and (75.2,40) .. (75,30) ;
%Curve Lines [id:da700297393905972] 
\draw    (85,30) .. controls (85.2,20) and (115.2,20) .. (115,30) ;
%Curve Lines [id:da5865444719234779] 
\draw    (85,30) .. controls (85.2,40) and (115.2,40) .. (115,30) ;
%Curve Lines [id:da42279770040571907] 
\draw  [dash pattern={on 0.84pt off 2.51pt}]  (25,90) .. controls (25.2,80) and (55.2,80) .. (55,90) ;
%Curve Lines [id:da9532386295173195] 
\draw    (25,90) .. controls (25.2,100) and (55.2,100) .. (55,90) ;
%Curve Lines [id:da19209298260082153] 
\draw  [dash pattern={on 0.84pt off 2.51pt}]  (65,90) .. controls (65.2,80) and (95.2,80) .. (95,90) ;
%Curve Lines [id:da6805476963829442] 
\draw    (65,90) .. controls (65.2,100) and (95.2,100) .. (95,90) ;
%Curve Lines [id:da23421182581932865] 
\draw    (35,30) .. controls (35.4,60.4) and (45.4,60.4) .. (45,30) ;
%Shape: Circle [id:dp6891174510630924] 
\draw  [fill={rgb, 255:red, 0; green, 0; blue, 0 }  ,fill opacity=1 ] (60,41.25) .. controls (60,40.56) and (60.56,40) .. (61.25,40) .. controls (61.94,40) and (62.5,40.56) .. (62.5,41.25) .. controls (62.5,41.94) and (61.94,42.5) .. (61.25,42.5) .. controls (60.56,42.5) and (60,41.94) .. (60,41.25) -- cycle ;
%Shape: Circle [id:dp16818256989346947] 
\draw  [fill={rgb, 255:red, 0; green, 0; blue, 0 }  ,fill opacity=1 ] (52.5,41.25) .. controls (52.5,40.56) and (53.06,40) .. (53.75,40) .. controls (54.44,40) and (55,40.56) .. (55,41.25) .. controls (55,41.94) and (54.44,42.5) .. (53.75,42.5) .. controls (53.06,42.5) and (52.5,41.94) .. (52.5,41.25) -- cycle ;
%Shape: Circle [id:dp5846053091162012] 
\draw  [fill={rgb, 255:red, 0; green, 0; blue, 0 }  ,fill opacity=1 ] (95,41.25) .. controls (95,40.56) and (95.56,40) .. (96.25,40) .. controls (96.94,40) and (97.5,40.56) .. (97.5,41.25) .. controls (97.5,41.94) and (96.94,42.5) .. (96.25,42.5) .. controls (95.56,42.5) and (95,41.94) .. (95,41.25) -- cycle ;
%Curve Lines [id:da048393825391680356] 
\draw    (5,30) .. controls (5.2,55.2) and (25.2,64.2) .. (25,90) ;
%Curve Lines [id:da7494561541250463] 
\draw    (75,30) .. controls (75.2,55.2) and (55.2,64.2) .. (55,90) ;
%Curve Lines [id:da3344776266249213] 
\draw    (85,30) .. controls (85.2,55.2) and (65.2,64.2) .. (65,90) ;
%Curve Lines [id:da7565226677710709] 
\draw    (115,30) .. controls (115.2,55.2) and (95.2,64.2) .. (95,90) ;

% Text Node
\draw (7,5.4) node [anchor=north west][inner sep=0.75pt]    {$O_{1}$};
% Text Node
\draw (47,5.4) node [anchor=north west][inner sep=0.75pt]    {$O_{2}$};
% Text Node
\draw (86,5.4) node [anchor=north west][inner sep=0.75pt]    {$O_{3}$};
% Text Node
\draw (26,100.4) node [anchor=north west][inner sep=0.75pt]    {$O_{1} '$};
% Text Node
\draw (66,100.4) node [anchor=north west][inner sep=0.75pt]    {$O_{2} '$};

\end{tikzpicture}
        \caption{A morphism $\Sigma_1$.}
        \label{fig:cobd_ex1}
    \end{subfigure}
    \hfill
    \begin{subfigure}{.48\linewidth}
        \centering
        \tikzset{every picture/.style={line width=0.75pt}} %set default line width to 0.75pt        

\begin{tikzpicture}[x=0.75pt,y=0.75pt,yscale=-1,xscale=1]
%uncomment if require: \path (0,120); %set diagram left start at 0, and has height of 120

%Curve Lines [id:da17591718005029433] 
\draw    (5,30) .. controls (5.2,20) and (35.2,20) .. (35,30) ;
%Curve Lines [id:da13536827763299397] 
\draw    (5,30) .. controls (5.2,40) and (35.2,40) .. (35,30) ;
%Curve Lines [id:da6247871594343669] 
\draw    (45,30) .. controls (45.2,20) and (75.2,20) .. (75,30) ;
%Curve Lines [id:da7014017174264878] 
\draw    (45,30) .. controls (45.2,40) and (75.2,40) .. (75,30) ;
%Curve Lines [id:da7006837870522422] 
\draw    (85,30) .. controls (85.2,20) and (115.2,20) .. (115,30) ;
%Curve Lines [id:da18234588869594248] 
\draw    (85,30) .. controls (85.2,40) and (115.2,40) .. (115,30) ;
%Curve Lines [id:da5027533906869032] 
\draw  [dash pattern={on 0.84pt off 2.51pt}]  (25,90) .. controls (25.2,80) and (55.2,80) .. (55,90) ;
%Curve Lines [id:da042540388636439275] 
\draw    (25,90) .. controls (25.2,100) and (55.2,100) .. (55,90) ;
%Curve Lines [id:da8851818334307817] 
\draw  [dash pattern={on 0.84pt off 2.51pt}]  (65,90) .. controls (65.2,80) and (95.2,80) .. (95,90) ;
%Curve Lines [id:da609359556587403] 
\draw    (65,90) .. controls (65.2,100) and (95.2,100) .. (95,90) ;
%Curve Lines [id:da5201149726738394] 
\draw    (75,30) .. controls (75.4,60.4) and (85.4,60.4) .. (85,30) ;
%Shape: Circle [id:dp8297009989029271] 
\draw  [fill={rgb, 255:red, 0; green, 0; blue, 0 }  ,fill opacity=1 ] (60,41.25) .. controls (60,40.56) and (60.56,40) .. (61.25,40) .. controls (61.94,40) and (62.5,40.56) .. (62.5,41.25) .. controls (62.5,41.94) and (61.94,42.5) .. (61.25,42.5) .. controls (60.56,42.5) and (60,41.94) .. (60,41.25) -- cycle ;
%Shape: Circle [id:dp8472805664173421] 
\draw  [fill={rgb, 255:red, 0; green, 0; blue, 0 }  ,fill opacity=1 ] (52.5,41.25) .. controls (52.5,40.56) and (53.06,40) .. (53.75,40) .. controls (54.44,40) and (55,40.56) .. (55,41.25) .. controls (55,41.94) and (54.44,42.5) .. (53.75,42.5) .. controls (53.06,42.5) and (52.5,41.94) .. (52.5,41.25) -- cycle ;
%Shape: Circle [id:dp9576033324216431] 
\draw  [fill={rgb, 255:red, 0; green, 0; blue, 0 }  ,fill opacity=1 ] (20,41.25) .. controls (20,40.56) and (20.56,40) .. (21.25,40) .. controls (21.94,40) and (22.5,40.56) .. (22.5,41.25) .. controls (22.5,41.94) and (21.94,42.5) .. (21.25,42.5) .. controls (20.56,42.5) and (20,41.94) .. (20,41.25) -- cycle ;
%Curve Lines [id:da026142933127613843] 
\draw    (5,30) .. controls (5.2,55.2) and (25.2,64.2) .. (25,90) ;
%Curve Lines [id:da10378483443869435] 
\draw    (115,30) .. controls (115.2,55.2) and (95.2,64.2) .. (95,90) ;
%Curve Lines [id:da4417723066881727] 
\draw    (35,30) .. controls (35.2,55.2) and (55.2,64.2) .. (55,90) ;
%Curve Lines [id:da02754468139418842] 
\draw    (45,30) .. controls (45.2,55.2) and (65.2,64.2) .. (65,90) ;

% Text Node
\draw (7,5.4) node [anchor=north west][inner sep=0.75pt]    {$O_{1}$};
% Text Node
\draw (47,5.4) node [anchor=north west][inner sep=0.75pt]    {$O_{2}$};
% Text Node
\draw (86,5.4) node [anchor=north west][inner sep=0.75pt]    {$O_{3}$};
% Text Node
\draw (26,100.4) node [anchor=north west][inner sep=0.75pt]    {$O_{1} '$};
% Text Node
\draw (66,100.4) node [anchor=north west][inner sep=0.75pt]    {$O_{2} '$};

\end{tikzpicture}
        \caption{A morphism $\Sigma_2$.}
        \label{fig:cobd_ex2}
    \end{subfigure}
    \caption{Examples of morphisms $\Sigma_1, \Sigma_2 \colon O_1 \sqcup O_2 \sqcup O_3 \to O_1' \sqcup O_2'$ in the category $\Cobd$.}
    \label{fig:cobd_ex}
\end{figure}
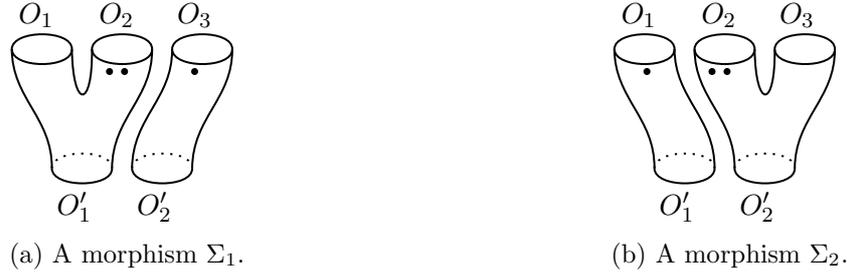

To introduce algebraic structure, we allow formal $\FF[H]$-linear sums of dotted cobordisms, such that $\Hom_{\Cobd}(O,O')$ becomes an abelian group where composition is bilinear.
This renders $\Cobd$ a preadditive category.
We continue to refer to this preadditive category as $\Cobd$.

Next, we give four local relations on the morphism set.

\begin{defn}\label{def:Cobdl}
    The category $\Cobdl$ is defined as the quotient category of $\Cobd$ by the following four local relations:
    \begin{itemize}
        \item (S) $\includegraphics[valign=c]{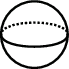}=0$
        \item (S$_\bullet$) $\includegraphics[valign=c]{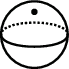}=1$
        \item (NC) $\includegraphics[valign=c,scale=.7]{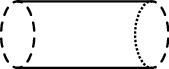}
\quad=\quad\includegraphics[valign=c,scale=.7]{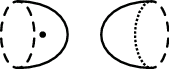}
+\includegraphics[valign=c,scale=.7]{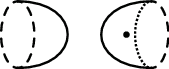}
+H\cdot \includegraphics[valign=c,scale=.7]{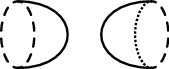}$
        \item ($H$-trading) $\includegraphics[valign=c,scale=.7]{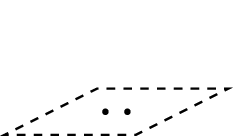}
\quad=\quad\includegraphics[valign=c,scale=.7]{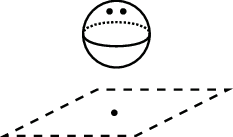}$
    \end{itemize}
\end{defn}

Again $\Cobdl$ is a preadditive category.
In $\Cobdl$, the relations (S$_\bullet$) and (NC) imply that taking the disjoint union with a two-dotted sphere is equivalent to scalar multiplication by $H$.
Furthermore, (NC) implies the following formula:
\begin{itemize}
    \item ($4$-Tu) \[\includegraphics[valign=c]{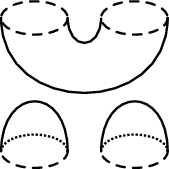}+\includegraphics[valign=c]{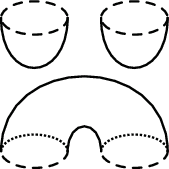}\quad=\quad\includegraphics[valign=c]{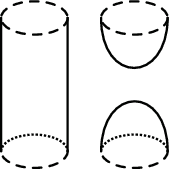}+\includegraphics[valign=c]{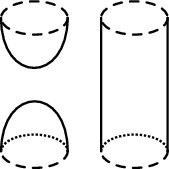}\]
\end{itemize}

\begin{rmk}
    The dashed boundaries in (NC), ($H$-trading), and ($4$-Tu) indicate the neighborhood of points on the cobordism rather than source or target objects. For example, the ($4$-Tu) relation can be interpreted as follows: let $\Sigma\in \Hom_{\Cobdl}(O,O')$ be a surface and choose four points $p_1, \dots, p_4$ on $\Sigma$. Let $\Sigma_{ij}$ denote the result of a surgery on $S^0=\{p_i, p_j\}$ for $i,j \in \{1, \dots, 4\}$. Then $\Sigma_{12} + \Sigma_{34} = \Sigma_{13} + \Sigma_{24}$ holds in $\Cobdl$.
\end{rmk}

Next, we construct the additive category $\Mat(\Cobdl)$.
For any preadditive category $\cC$, its \emph{additive closure} $\Mat(\cC)$ is defined such that objects are finite formal direct sums of objects in $\cC$, and morphisms are matrices of morphisms in $\cC$, with composition following matrix multiplication.
We then define $\Kobd = \Kom(\Mat(\Cobdl))$ as the category of chain complexes over $\Mat(\Cobdl)$, and $\Kobdh = \Komh(\Mat(\Cobdl))$ as the corresponding homotopy category. %%%%%%%%%%%%%%%%%%%%%%%%%%%%%%%%%%%%5

In the following paragraphs, we will define a chain complex $\bn{D_K}$ in $\Kobd$ for a given diagram $D_K$ of a link $K$, and it turns out to be a knot invariant, which will be denoted by $\bn{K}$ in the category $\Kobdh$.
Fix a diagram $D_K$ of a link $K$, and suppose $D_K$ has $n$ crossings.
We enumerate the crossings of $D_K$ from $1$ to $n$.
We can resolve each crossing in two ways, called \emph{0-resolution} and \emph{1-resolution}, as \Cref{fig:res}.
Then we have a total of $2^n$ ways of resolving all crossings of $D_K$.
Let $V = \{0,1\}^{\{1,\dots,n\}}$ be the set of all resolutions.
In practice, $v\in V$ will be written in length $n$ bits (cf. \Cref{fig:kh_trefoil}).
For each $v \in V$, let $D_K(v)$ be the diagram obtained by applying $v(i)$-resolution to the $i$-th crossing of $D_K$ for each $i=1,\ldots,n$.
Then we may consider $D_K(v)$ as an object in $\Cobdl$.
By gathering all objects, $\bigoplus_{v \in V} D_K(v)$ will be our complex in $\Kobd$.

\begin{figure}
    \centering
    \tikzset{every picture/.style={line width=0.75pt}} %set default line width to 0.75pt        

\begin{tikzpicture}[x=0.75pt,y=0.75pt,yscale=-1,xscale=1]
%uncomment if require: \path (0,105); %set diagram left start at 0, and has height of 105

%Shape: Ellipse [id:dp3858268166606158] 
\draw   (10,45) .. controls (10,25.67) and (25.67,10) .. (45,10) .. controls (64.33,10) and (80,25.67) .. (80,45) .. controls (80,64.33) and (64.33,80) .. (45,80) .. controls (25.67,80) and (10,64.33) .. (10,45) -- cycle ;
%Curve Lines [id:da581048838443869] 
\draw [line width=2.25]    (70,20) .. controls (50,39.8) and (40,39.8) .. (20,20) ;
%Curve Lines [id:da7387366121429697] 
\draw [line width=2.25]    (70,70) .. controls (50,49.8) and (40,49.8) .. (20,70) ;
%Shape: Ellipse [id:dp27866029533400993] 
\draw   (130,45) .. controls (130,25.67) and (145.67,10) .. (165,10) .. controls (184.33,10) and (200,25.67) .. (200,45) .. controls (200,64.33) and (184.33,80) .. (165,80) .. controls (145.67,80) and (130,64.33) .. (130,45) -- cycle ;
%Shape: Ellipse [id:dp3851465343095939] 
\draw   (250,45) .. controls (250,25.67) and (265.67,10) .. (285,10) .. controls (304.33,10) and (320,25.67) .. (320,45) .. controls (320,64.33) and (304.33,80) .. (285,80) .. controls (265.67,80) and (250,64.33) .. (250,45) -- cycle ;
%Curve Lines [id:da2062727645952278] 
\draw [line width=2.25]    (260,70) .. controls (280,50) and (280,39.8) .. (260,20) ;
%Curve Lines [id:da015001694442060454] 
\draw [line width=2.25]    (310,70) .. controls (290,49.8) and (290,40) .. (310,20) ;
%Straight Lines [id:da5875542529191169] 
\draw [line width=2.25]    (140,20) -- (190,70) ;
%Straight Lines [id:da11566020675286581] 
\draw [line width=2.25]    (140,70) -- (160,50) ;
%Straight Lines [id:da17847299134625627] 
\draw [line width=2.25]    (170,40) -- (190,20) ;
%Straight Lines [id:da558703879968662] 
\draw    (120,50) -- (92,50) ;
\draw [shift={(90,50)}, rotate = 360] [color={rgb, 255:red, 0; green, 0; blue, 0 }  ][line width=0.75]    (10.93,-3.29) .. controls (6.95,-1.4) and (3.31,-0.3) .. (0,0) .. controls (3.31,0.3) and (6.95,1.4) .. (10.93,3.29)   ;
%Straight Lines [id:da8418309477381755] 
\draw    (210,50) -- (238,50) ;
\draw [shift={(240,50)}, rotate = 180] [color={rgb, 255:red, 0; green, 0; blue, 0 }  ][line width=0.75]    (10.93,-3.29) .. controls (6.95,-1.4) and (3.31,-0.3) .. (0,0) .. controls (3.31,0.3) and (6.95,1.4) .. (10.93,3.29)   ;

% Text Node
\draw (101,32.4) node [anchor=north west][inner sep=0.75pt]    {$0$};
% Text Node
\draw (217,32.4) node [anchor=north west][inner sep=0.75pt]    {$1$};

\end{tikzpicture}
    \caption{Resolution of a crossing.}
    \label{fig:res}
\end{figure}
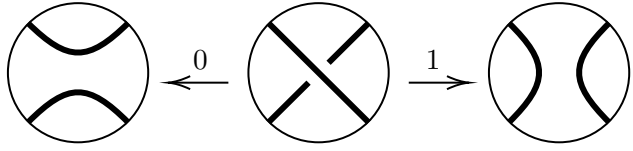

The differential on $\bigoplus_{v \in V} D_K(v)$ is defined as following.
Suppose $v, w \in V$ differ at exactly one crossing, say that $v(i) = 0$, $w(i)=1$, and $v(j)=w(j)$ for all $j \neq i$.
Then there is a cobordism $S_{v,w}$ from $D_K(v)$ to $D_K(w)$, which is a saddle near the $i$-th crossing.
We collect all such cobordisms $S_{v,w}$ and define a differential $d$ on $\bigoplus_{v \in V} D_K(v)$.
One can check that $d^2=0$ because each square
$
\begin{tikzcd}
    &D_K(v) \arrow[rd] & \\
    D_K(u) \arrow[ru]\arrow[rd] && D_K(w) \\
    &D_K(v') \arrow[ru] &
\end{tikzcd}
$
commutes, and we are working over $\FF$=$\ZZ/2\ZZ$.

\begin{figure}
    \centering
    \includegraphics[scale=.7]{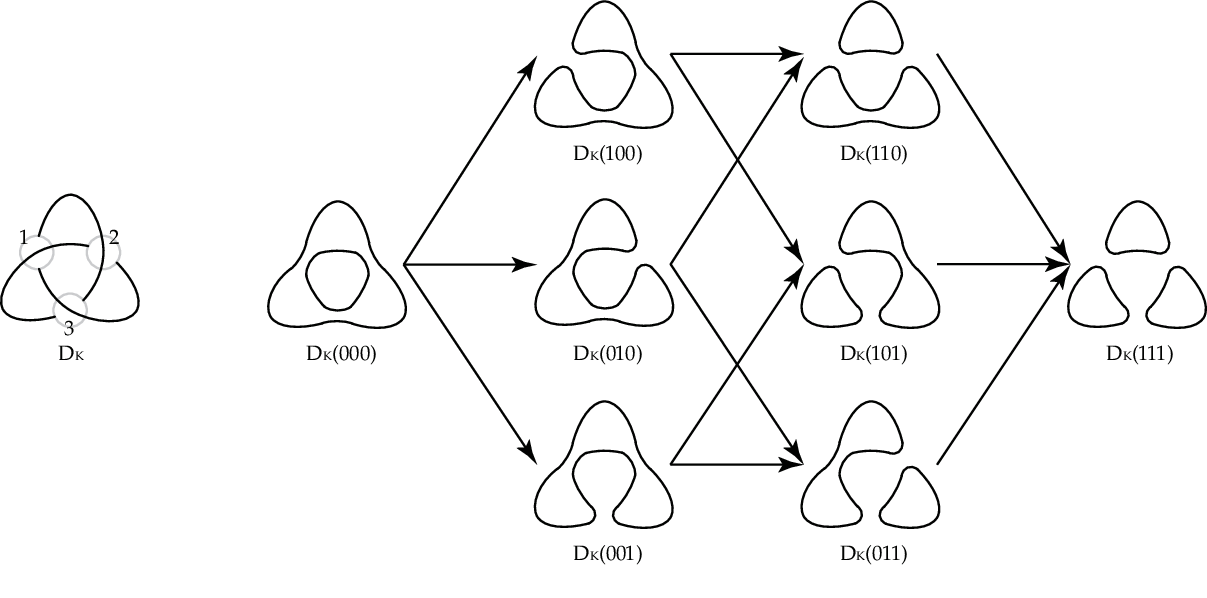}
    \caption{A formal Khovanov complex for the left-hand trefoil.}
    \label{fig:kh_trefoil}
\end{figure}

Now we define $\bn{D_K}$ to be the complex $\left(\bigoplus_{v \in V} D_K(v), d\right)$ in $\Kobd$.
The following theorem says that $\bn{D_K}$ actually can be seen as an invariant of the link $K$, when we consider it up to chain homotopy.

\begin{thm}[\cite{Khovanov:2000}, \cite{Bar-Natan:2005}]
    If two diagrams $D_1$ and $D_2$ are related by one Reidemeister move, then the complexes $\bn{D_1}$ and $\bn{D_2}$ are chain homotopy equivalent. In particular, the chain homotopy type of $\bn{D_K}$ is an invariant of the link $K$.
\end{thm}

Now we define $\bn{K}$ to be the chain homotopy type of $\bn{D_K}$, which lies in the category $\Kobdh$.
$\bn{K}$ is called the \emph{formal Bar-Natan complex} of $K$.

\begin{rmk}
    We have introduced only the essential components of Khovanov homology theory required for our purposes.
    Specifically:
    \begin{itemize}
        \item There exists a more general version, called \emph{$U(2)$-equivariant Khovanov homology}, which utilizes a coefficient ring $R[h,t]$ with two variables.
        \item The category $\Kobdh$ and the invariant $\bn{K}$ can be defined for tangles.
        \item $\bn{K}$ admits a bigraded structure.
    \end{itemize}
    We do not address these aspects in this paper.
    Indeed, we will later define a chain map that does not necessarily preserve the homological grading.
    For a more general theory, we refer the reader to the literature, e.g., \cite{Khovanov:2000, Bar-Natan:2005, Kotelskiy-Watson-Zibrowius:2019, Kim-Sano:2025}.
\end{rmk}

\begin{rmk}
    In many related papers, the categories $\Cob$ and $\Cobd$ require embeddings of surfaces into $3$-dimensional space; objects are assumed to be crossingless diagrams on the disk $D^2$, and morphisms are dotted cobordisms embedded in $D^2 \times I$.
    In contrast, we use abstract cobordisms, which are sufficient to recover the Bar-Natan homology of a link.
\end{rmk}

\subsection{Recovering \texorpdfstring{$\BN(K)$}{BN(K)} from \texorpdfstring{$\bn{K}$}{[[K]]}}

In this subsection, we explain how to convert the formal complex $\bn{K}$ into a computable form—specifically, a chain complex of modules—by applying a 2d-TQFT.

A \emph{2 dimensional topological quantum field theory (2d-TQFT)} is a monoidal functor from $\Cob$ to $(R-\mathrm{Mod})$, the category of modules over a ring $R$.
It is well-known that there is an equivalence of categories between 2d-TQFTs and commutative Frobenius algebras; a set $A$ with four operations $\iota, m, \Delta, \epsilon$ satisfying the following:
\begin{itemize}
    \item $A$ is a commutative algebra with unit $\iota \colon R \to A$ and multiplication $m\colon A \otimes A \to A$;
    \item $A$ is a cocommutative coalgebra with comultiplication $\Delta \colon A \to A \otimes A$ and counit $\epsilon \colon A \to R$; and
    \item $m$ and $\Delta$ satisfy the \emph{Frobenius laws}:
    \[ (m \otimes \id) \circ (\id \otimes \Delta)
    = \Delta \circ m
    = (\id \otimes m) \circ (\Delta \otimes \id).\]
\end{itemize}

This correspondence arises from the presentation of the category $\Cob$.
Every morphism in $\Cob$ is generated by basic cobordisms, subject to several relations that describe isotopies between cobordisms (see \cite{Kock:2004}).
For a given commutative Frobenius algebra $A$, the corresponding 2d-TQFT sends a circle to $A$ and each basic cobordism to an operation of $A$.
\Cref{fig:TQFT_gen1} shows the basic cobordisms and their corresponding operations.

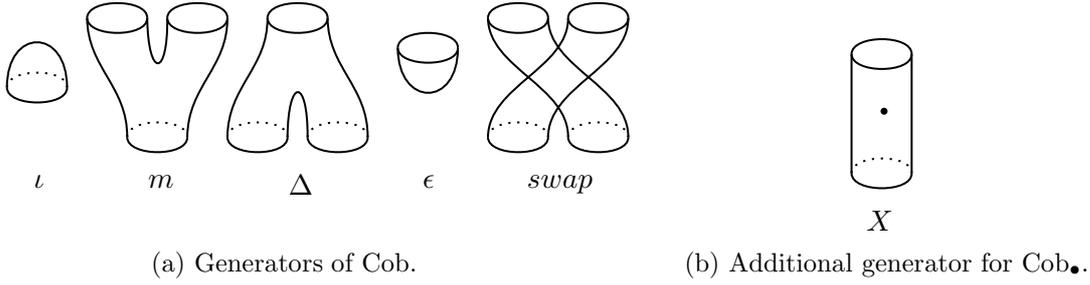
\begin{figure}
    \centering
    \begin{subfigure}{.48\linewidth}
        \centering
        \tikzset{every picture/.style={line width=0.75pt}} %set default line width to 0.75pt        

\begin{tikzpicture}[x=0.75pt,y=0.75pt,yscale=-1,xscale=1]
%uncomment if require: \path (0,119); %set diagram left start at 0, and has height of 119

%Curve Lines [id:da39451332159217845] 
\draw    (10,50) .. controls (10.2,60) and (40.2,60) .. (40,50) ;
%Curve Lines [id:da09287971750986512] 
\draw    (10,50) .. controls (10,19.6) and (40,19.6) .. (40,50) ;
%Curve Lines [id:da19140936109626105] 
\draw    (50,15) .. controls (50.2,5) and (80.2,5) .. (80,15) ;
%Curve Lines [id:da30823525929774165] 
\draw    (50,15) .. controls (50.2,25) and (80.2,25) .. (80,15) ;
%Curve Lines [id:da41529811440057407] 
\draw    (90,15) .. controls (90.2,5) and (120.2,5) .. (120,15) ;
%Curve Lines [id:da5272031586184814] 
\draw    (90,15) .. controls (90.2,25) and (120.2,25) .. (120,15) ;
%Curve Lines [id:da20268700982898213] 
\draw  [dash pattern={on 0.84pt off 2.51pt}]  (70,75) .. controls (70.2,65) and (100.2,65) .. (100,75) ;
%Curve Lines [id:da7139320391317516] 
\draw    (70,75) .. controls (70.2,85) and (100.2,85) .. (100,75) ;
%Curve Lines [id:da5526250639807065] 
\draw    (50,15) .. controls (50.2,40.2) and (70.2,49.2) .. (70,75) ;
%Curve Lines [id:da34907417133290986] 
\draw    (120,15) .. controls (120.2,40.2) and (100.2,49.2) .. (100,75) ;
%Curve Lines [id:da30001838175361595] 
\draw    (80,15) .. controls (80.4,45.4) and (90.4,45.4) .. (90,15) ;
%Curve Lines [id:da005652349812134805] 
\draw  [dash pattern={on 0.84pt off 2.51pt}]  (10,50) .. controls (10.2,40) and (40.2,40) .. (40,50) ;
%Curve Lines [id:da7370778336697331] 
\draw    (120,75) .. controls (120.2,85) and (150.2,85) .. (150,75) ;
%Curve Lines [id:da08797281000951074] 
\draw    (160,75) .. controls (160.2,85) and (190.2,85) .. (190,75) ;
%Curve Lines [id:da46227173550543954] 
\draw    (140,15) .. controls (140.2,5) and (170.2,5) .. (170,15) ;
%Curve Lines [id:da757995699130749] 
\draw    (140,15) .. controls (140.2,25) and (170.2,25) .. (170,15) ;
%Curve Lines [id:da9514881665847054] 
\draw    (150,75) .. controls (150.2,44.7) and (160.2,44.7) .. (160,75) ;
%Curve Lines [id:da39740542521769817] 
\draw  [dash pattern={on 0.84pt off 2.51pt}]  (120,75) .. controls (120.2,65) and (150.2,65) .. (150,75) ;
%Curve Lines [id:da6444558363113394] 
\draw  [dash pattern={on 0.84pt off 2.51pt}]  (160,75) .. controls (160.2,65) and (190.2,65) .. (190,75) ;
%Curve Lines [id:da7506208325078927] 
\draw    (205,30) .. controls (205.2,20) and (235.2,20) .. (235,30) ;
%Curve Lines [id:da3513696658062391] 
\draw    (205,30) .. controls (205.2,40) and (235.2,40) .. (235,30) ;
%Curve Lines [id:da46712932617300984] 
\draw    (205,30) .. controls (205.2,60.2) and (235.2,60.2) .. (235,30) ;
%Curve Lines [id:da7000046748923849] 
\draw    (250,15) .. controls (250.2,5) and (280.2,5) .. (280,15) ;
%Curve Lines [id:da5641257480667944] 
\draw    (250,15) .. controls (250.2,25) and (280.2,25) .. (280,15) ;
%Curve Lines [id:da3249201428051388] 
\draw    (290,15) .. controls (290.2,5) and (320.2,5) .. (320,15) ;
%Curve Lines [id:da795735103331056] 
\draw    (290,15) .. controls (290.2,25) and (320.2,25) .. (320,15) ;
%Curve Lines [id:da9797300111447778] 
\draw    (250,75) .. controls (250.2,85) and (280.2,85) .. (280,75) ;
%Curve Lines [id:da8589522805483782] 
\draw  [dash pattern={on 0.84pt off 2.51pt}]  (250,75) .. controls (250.2,65) and (280.2,65) .. (280,75) ;
%Curve Lines [id:da14225491806515433] 
\draw    (290,75) .. controls (290.2,85) and (320.2,85) .. (320,75) ;
%Curve Lines [id:da05014596741356636] 
\draw  [dash pattern={on 0.84pt off 2.51pt}]  (290,75) .. controls (290.2,65) and (320.2,65) .. (320,75) ;
%Curve Lines [id:da22860824729381146] 
\draw    (140,15) .. controls (140.2,40.2) and (120.2,49.2) .. (120,75) ;
%Curve Lines [id:da27846360261321623] 
\draw    (170,15) .. controls (170.2,40.2) and (190.2,49.2) .. (190,75) ;
%Curve Lines [id:da008993739872615425] 
\draw    (280,15) .. controls (280.2,40.2) and (320.2,49.2) .. (320,75) ;
%Curve Lines [id:da5848042565712377] 
\draw    (250,15) .. controls (250.2,40.2) and (290.2,49.2) .. (290,75) ;
%Curve Lines [id:da6039380834228587] 
\draw    (290,15) .. controls (290.2,40.2) and (250.2,49.2) .. (250,75) ;
%Curve Lines [id:da3453893686873105] 
\draw    (320,15) .. controls (320.2,40.2) and (280.2,49.2) .. (280,75) ;

% Text Node
\draw (22,92.4) node [anchor=north west][inner sep=0.75pt]    {$\iota $};
% Text Node
\draw (79,92.4) node [anchor=north west][inner sep=0.75pt]    {$m$};
% Text Node
\draw (149,92.4) node [anchor=north west][inner sep=0.75pt]    {$\Delta $};
% Text Node
\draw (216,92.4) node [anchor=north west][inner sep=0.75pt]    {$\epsilon $};
% Text Node
\draw (268,92.4) node [anchor=north west][inner sep=0.75pt]    {$swap$};

\end{tikzpicture}
        \caption{Generators of $\Cob$.}
        \label{fig:TQFT_gen1}
    \end{subfigure}
    \hfill
    \begin{subfigure}{.48\linewidth}
        \centering
        \tikzset{every picture/.style={line width=0.75pt}} %set default line width to 0.75pt        

\begin{tikzpicture}[x=0.75pt,y=0.75pt,yscale=-1,xscale=1]
%uncomment if require: \path (0,119); %set diagram left start at 0, and has height of 119

%Curve Lines [id:da3160888833526453] 
\draw    (10,15) .. controls (10.2,5) and (40.2,5) .. (40,15) ;
%Curve Lines [id:da20081579591125676] 
\draw    (10,15) .. controls (10.2,25) and (40.2,25) .. (40,15) ;
%Curve Lines [id:da6711805296387527] 
\draw  [dash pattern={on 0.84pt off 2.51pt}]  (10,75) .. controls (10.2,65) and (40.2,65) .. (40,75) ;
%Curve Lines [id:da8298130362972905] 
\draw    (10,75) .. controls (10.2,85) and (40.2,85) .. (40,75) ;
%Straight Lines [id:da8434902104552884] 
\draw    (10,15) -- (10,75) ;
%Straight Lines [id:da20141099095595194] 
\draw    (40,15) -- (40,75) ;
%Shape: Circle [id:dp06334516311385785] 
\draw  [fill={rgb, 255:red, 0; green, 0; blue, 0 }  ,fill opacity=1 ] (25,43.75) .. controls (25,43.06) and (25.56,42.5) .. (26.25,42.5) .. controls (26.94,42.5) and (27.5,43.06) .. (27.5,43.75) .. controls (27.5,44.44) and (26.94,45) .. (26.25,45) .. controls (25.56,45) and (25,44.44) .. (25,43.75) -- cycle ;

% Text Node
\draw (16,92.4) node [anchor=north west][inner sep=0.75pt]    {$X$};

\end{tikzpicture}
        \caption{Additional generator for $\Cobd$.}
        \label{fig:TQFT_gen2}
    \end{subfigure}
    \caption{Building blocks for $2$-dimensional (dotted) cobordisms. Each cobordism is labeled with a corresponding module homomorphism (here, $swap \colon A \otimes A \to A \otimes A$ maps $a\otimes b$ to $b\otimes a$). In the context of $\Cobd$ and Bar-Natan homology, the map $X \colon V \to V$ is defined by $a \mapsto m(a \otimes \xx)$.}
    \label{fig:TQFT_gen}
\end{figure}

Our commutative Frobenius algebra $V$ is:
\begin{itemize}
    \item $V$ is a free $\FF[H]$-module with two generators $\ii, \xx$,
    \item A unit $\iota \colon \FF[H] \to V$ and a multiplication $m \colon V \otimes V \to V$ is defined by
    \[ \iota(1) = \ii, \qquad m(\xx \otimes \xx) = H\cdot\xx, \]
    (i.e., $\ii$ is the multiplicative identity.)
    \item A comultiplication $\Delta \colon V \to V \otimes V$ and a counit $\epsilon \colon V \to \FF[H]$ is defined by
    \[ \begin{cases}
    \Delta(\ii) &= \ii\otimes\xx + \xx\otimes\ii + H\cdot\ii\otimes\ii \\
    \Delta(\xx) &= \xx\otimes\xx
    \end{cases}, \qquad
    \begin{cases}
        \epsilon(\ii)&=0\\
        \epsilon(\xx)&=1
    \end{cases}.\]
\end{itemize}

In our setting, we additionally account for dots.
We introduce one further generator: a dotted cylinder, as shown in \Cref{fig:TQFT_gen2}.
While dots introduce additional relations, it suffices to consider the `sliding' of a dot across each generator.
These relations are listed in \Cref{fig:TQFTd_rel}; one can easily verify that our algebra $V$ and the map $X \colon a \mapsto m(a \otimes \xx)$ (which corresponds to the dotted cylinder) satisfy these relations.

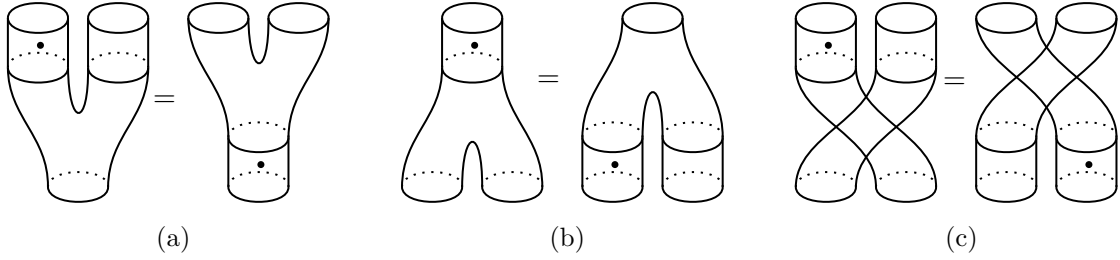
\begin{figure}
    \centering
    \begin{subfigure}{.32\linewidth}
        \centering
        \tikzset{every picture/.style={line width=0.75pt}} %set default line width to 0.75pt        

\begin{tikzpicture}[x=0.75pt,y=0.75pt,yscale=-1,xscale=1]
%uncomment if require: \path (0,119); %set diagram left start at 0, and has height of 119

%Curve Lines [id:da8510181941978792] 
\draw    (5,40) .. controls (5.2,50) and (35.2,50) .. (35,40) ;
%Curve Lines [id:da8402177837775909] 
\draw    (45,40) .. controls (45.2,50) and (75.2,50) .. (75,40) ;
%Curve Lines [id:da2727109209661429] 
\draw  [dash pattern={on 0.84pt off 2.51pt}]  (25,100) .. controls (25.2,90) and (55.2,90) .. (55,100) ;
%Curve Lines [id:da5059512943529285] 
\draw    (25,100) .. controls (25.2,110) and (55.2,110) .. (55,100) ;
%Curve Lines [id:da28290182921460305] 
\draw    (5,40) .. controls (5.2,65.2) and (25.2,74.2) .. (25,100) ;
%Curve Lines [id:da6427300485735756] 
\draw    (75,40) .. controls (75.2,65.2) and (55.2,74.2) .. (55,100) ;
%Curve Lines [id:da35870047994473675] 
\draw    (35,40) .. controls (35.4,70.4) and (45.4,70.4) .. (45,40) ;
%Curve Lines [id:da5475883997801523] 
\draw    (5,15) .. controls (5.2,5) and (35.2,5) .. (35,15) ;
%Curve Lines [id:da3953371606100192] 
\draw    (5,15) .. controls (5.2,25) and (35.2,25) .. (35,15) ;
%Curve Lines [id:da43950868149228495] 
\draw  [dash pattern={on 0.84pt off 2.51pt}]  (5,40) .. controls (5.2,30) and (35.2,30) .. (35,40) ;
%Curve Lines [id:da4578629852088597] 
\draw  [dash pattern={on 0.84pt off 2.51pt}]  (45,40) .. controls (45.2,30) and (75.2,30) .. (75,40) ;
%Straight Lines [id:da7905569471206292] 
\draw    (75,15) -- (75,40) ;
%Curve Lines [id:da14833628630074747] 
\draw    (45,15) .. controls (45.2,5) and (75.2,5) .. (75,15) ;
%Curve Lines [id:da12432878300253669] 
\draw    (45,15) .. controls (45.2,25) and (75.2,25) .. (75,15) ;
%Shape: Circle [id:dp9940561968578953] 
\draw  [fill={rgb, 255:red, 0; green, 0; blue, 0 }  ,fill opacity=1 ] (20,28.75) .. controls (20,28.06) and (20.56,27.5) .. (21.25,27.5) .. controls (21.94,27.5) and (22.5,28.06) .. (22.5,28.75) .. controls (22.5,29.44) and (21.94,30) .. (21.25,30) .. controls (20.56,30) and (20,29.44) .. (20,28.75) -- cycle ;
%Curve Lines [id:da935616096253106] 
\draw    (95,15) .. controls (95.2,5) and (125.2,5) .. (125,15) ;
%Curve Lines [id:da6973629864912507] 
\draw    (95,15) .. controls (95.2,25) and (125.2,25) .. (125,15) ;
%Curve Lines [id:da2159444588667465] 
\draw    (135,15) .. controls (135.2,5) and (165.2,5) .. (165,15) ;
%Curve Lines [id:da9393338513149248] 
\draw    (135,15) .. controls (135.2,25) and (165.2,25) .. (165,15) ;
%Curve Lines [id:da3900265269488217] 
\draw  [dash pattern={on 0.84pt off 2.51pt}]  (115,75) .. controls (115.2,65) and (145.2,65) .. (145,75) ;
%Curve Lines [id:da15993723055184084] 
\draw    (115,75) .. controls (115.2,85) and (145.2,85) .. (145,75) ;
%Curve Lines [id:da08772985514259157] 
\draw    (95,15) .. controls (95.2,40.2) and (115.2,49.2) .. (115,75) ;
%Curve Lines [id:da43370369892358096] 
\draw    (165,15) .. controls (165.2,40.2) and (145.2,49.2) .. (145,75) ;
%Curve Lines [id:da6215209841784702] 
\draw    (125,15) .. controls (125.4,45.4) and (135.4,45.4) .. (135,15) ;
%Straight Lines [id:da5758612525826792] 
\draw    (45,15) -- (45,40) ;
%Straight Lines [id:da027334585775759312] 
\draw    (5,15) -- (5,40) ;
%Straight Lines [id:da6453363730880582] 
\draw    (35,15) -- (35,40) ;
%Straight Lines [id:da9077626698082337] 
\draw    (145,75) -- (145,100) ;
%Straight Lines [id:da8666203371426177] 
\draw    (115,75) -- (115,100) ;
%Curve Lines [id:da22536815818505385] 
\draw    (115,100) .. controls (115.2,110) and (145.2,110) .. (145,100) ;
%Curve Lines [id:da037891263130173325] 
\draw  [dash pattern={on 0.84pt off 2.51pt}]  (115,100) .. controls (115.2,90) and (145.2,90) .. (145,100) ;
%Shape: Circle [id:dp04582227953626228] 
\draw  [fill={rgb, 255:red, 0; green, 0; blue, 0 }  ,fill opacity=1 ] (130,88.75) .. controls (130,88.06) and (130.56,87.5) .. (131.25,87.5) .. controls (131.94,87.5) and (132.5,88.06) .. (132.5,88.75) .. controls (132.5,89.44) and (131.94,90) .. (131.25,90) .. controls (130.56,90) and (130,89.44) .. (130,88.75) -- cycle ;

% Text Node
\draw (76,52.4) node [anchor=north west][inner sep=0.75pt]    {$=$};

\end{tikzpicture}
        \caption{}
        \label{fig:TQFTd_rel1}
    \end{subfigure}
    \hfill
    \begin{subfigure}{.32\linewidth}
        \centering
        \tikzset{every picture/.style={line width=0.75pt}} %set default line width to 0.75pt        

\begin{tikzpicture}[x=0.75pt,y=0.75pt,yscale=-1,xscale=1]
%uncomment if require: \path (0,119); %set diagram left start at 0, and has height of 119

%Curve Lines [id:da39388728009467966] 
\draw    (170,100) .. controls (170.2,110) and (200.2,110) .. (200,100) ;
%Curve Lines [id:da8808446628617085] 
\draw    (210,100) .. controls (210.2,110) and (240.2,110) .. (240,100) ;
%Curve Lines [id:da7437368909869113] 
\draw    (200,100) .. controls (200.2,69.7) and (210.2,69.7) .. (210,100) ;
%Curve Lines [id:da18946986627216322] 
\draw  [dash pattern={on 0.84pt off 2.51pt}]  (170,100) .. controls (170.2,90) and (200.2,90) .. (200,100) ;
%Curve Lines [id:da0049932671898373915] 
\draw  [dash pattern={on 0.84pt off 2.51pt}]  (210,100) .. controls (210.2,90) and (240.2,90) .. (240,100) ;
%Curve Lines [id:da6327854292782024] 
\draw    (190,40) .. controls (190.2,65.2) and (170.2,74.2) .. (170,100) ;
%Curve Lines [id:da03571995562315333] 
\draw    (220,40) .. controls (220.2,65.2) and (240.2,74.2) .. (240,100) ;
%Curve Lines [id:da6253220381881438] 
\draw    (190,40) .. controls (190.2,50) and (220.2,50) .. (220,40) ;
%Curve Lines [id:da7273399070462002] 
\draw  [dash pattern={on 0.84pt off 2.51pt}]  (190,40) .. controls (190.2,30) and (220.2,30) .. (220,40) ;
%Straight Lines [id:da9353789041398992] 
\draw    (220,15) -- (220,40) ;
%Curve Lines [id:da655715459308014] 
\draw    (190,15) .. controls (190.2,5) and (220.2,5) .. (220,15) ;
%Curve Lines [id:da07148210466499993] 
\draw    (190,15) .. controls (190.2,25) and (220.2,25) .. (220,15) ;
%Straight Lines [id:da9108258917620528] 
\draw    (190,15) -- (190,40) ;
%Curve Lines [id:da30481752261355355] 
\draw    (260,75) .. controls (260.2,85) and (290.2,85) .. (290,75) ;
%Curve Lines [id:da0394271226629217] 
\draw    (300,75) .. controls (300.2,85) and (330.2,85) .. (330,75) ;
%Curve Lines [id:da949196887894662] 
\draw    (290,75) .. controls (290.2,44.7) and (300.2,44.7) .. (300,75) ;
%Curve Lines [id:da6192771368429466] 
\draw  [dash pattern={on 0.84pt off 2.51pt}]  (260,75) .. controls (260.2,65) and (290.2,65) .. (290,75) ;
%Curve Lines [id:da2893483817550274] 
\draw  [dash pattern={on 0.84pt off 2.51pt}]  (300,75) .. controls (300.2,65) and (330.2,65) .. (330,75) ;
%Curve Lines [id:da4567164929785228] 
\draw    (280,15) .. controls (280.2,40.2) and (260.2,49.2) .. (260,75) ;
%Curve Lines [id:da357339787070278] 
\draw    (310,15) .. controls (310.2,40.2) and (330.2,49.2) .. (330,75) ;
%Curve Lines [id:da5533376495411564] 
\draw    (280,15) .. controls (280.2,5) and (310.2,5) .. (310,15) ;
%Curve Lines [id:da6867714491913559] 
\draw  [dash pattern={on 0.84pt off 2.51pt}]  (260,100) .. controls (260.2,90) and (290.2,90) .. (290,100) ;
%Straight Lines [id:da1376472717584153] 
\draw    (290,75) -- (290,100) ;
%Straight Lines [id:da8920456658557211] 
\draw    (260,75) -- (260,100) ;
%Curve Lines [id:da9876910726178577] 
\draw  [dash pattern={on 0.84pt off 2.51pt}]  (300,100) .. controls (300.2,90) and (330.2,90) .. (330,100) ;
%Straight Lines [id:da5458478421076903] 
\draw    (330,75) -- (330,100) ;
%Straight Lines [id:da6721454688604315] 
\draw    (300,75) -- (300,100) ;
%Curve Lines [id:da9257200711126143] 
\draw    (260,100) .. controls (260.2,110) and (290.2,110) .. (290,100) ;
%Curve Lines [id:da6317105150622216] 
\draw    (300,100) .. controls (300.2,110) and (330.2,110) .. (330,100) ;
%Shape: Circle [id:dp5426039044269577] 
\draw  [fill={rgb, 255:red, 0; green, 0; blue, 0 }  ,fill opacity=1 ] (205,28.75) .. controls (205,28.06) and (205.56,27.5) .. (206.25,27.5) .. controls (206.94,27.5) and (207.5,28.06) .. (207.5,28.75) .. controls (207.5,29.44) and (206.94,30) .. (206.25,30) .. controls (205.56,30) and (205,29.44) .. (205,28.75) -- cycle ;
%Shape: Circle [id:dp05392333372702274] 
\draw  [fill={rgb, 255:red, 0; green, 0; blue, 0 }  ,fill opacity=1 ] (275,88.75) .. controls (275,88.06) and (275.56,87.5) .. (276.25,87.5) .. controls (276.94,87.5) and (277.5,88.06) .. (277.5,88.75) .. controls (277.5,89.44) and (276.94,90) .. (276.25,90) .. controls (275.56,90) and (275,89.44) .. (275,88.75) -- cycle ;
%Curve Lines [id:da8258419554621267] 
\draw    (280,15) .. controls (280.2,25) and (310.2,25) .. (310,15) ;

% Text Node
\draw (236,42.4) node [anchor=north west][inner sep=0.75pt]    {$=$};

\end{tikzpicture}
        \caption{}
        \label{fig:TQFTd_rel2}
    \end{subfigure}
    \hfill
    \begin{subfigure}{.32\linewidth}
        \centering
        \tikzset{every picture/.style={line width=0.75pt}} %set default line width to 0.75pt        

\begin{tikzpicture}[x=0.75pt,y=0.75pt,yscale=-1,xscale=1]
%uncomment if require: \path (0,119); %set diagram left start at 0, and has height of 119

%Curve Lines [id:da9041941426627691] 
\draw    (335,40) .. controls (335.2,50) and (365.2,50) .. (365,40) ;
%Curve Lines [id:da4646190029557782] 
\draw    (375,40) .. controls (375.2,50) and (405.2,50) .. (405,40) ;
%Curve Lines [id:da990810802289269] 
\draw    (375,100) .. controls (375.2,110) and (405.2,110) .. (405,100) ;
%Curve Lines [id:da4914015487752009] 
\draw  [dash pattern={on 0.84pt off 2.51pt}]  (375,100) .. controls (375.2,90) and (405.2,90) .. (405,100) ;
%Curve Lines [id:da3617149948815196] 
\draw    (335,100) .. controls (335.2,110) and (365.2,110) .. (365,100) ;
%Curve Lines [id:da007762814156650699] 
\draw  [dash pattern={on 0.84pt off 2.51pt}]  (335,100) .. controls (335.2,90) and (365.2,90) .. (365,100) ;
%Curve Lines [id:da8353008247108591] 
\draw    (365,40) .. controls (365.2,65.2) and (405.2,74.2) .. (405,100) ;
%Curve Lines [id:da7585099631220561] 
\draw    (335,40) .. controls (335.2,65.2) and (375.2,74.2) .. (375,100) ;
%Curve Lines [id:da8474747906728015] 
\draw    (405,40) .. controls (405.2,65.2) and (365.2,74.2) .. (365,100) ;
%Curve Lines [id:da6404952673162596] 
\draw    (375,40) .. controls (375.2,65.2) and (335.2,74.2) .. (335,100) ;
%Curve Lines [id:da5939682170068697] 
\draw  [dash pattern={on 0.84pt off 2.51pt}]  (335,40) .. controls (335.2,30) and (365.2,30) .. (365,40) ;
%Straight Lines [id:da6336559984267796] 
\draw    (365,15) -- (365,40) ;
%Curve Lines [id:da48578658997863855] 
\draw    (335,15) .. controls (335.2,5) and (365.2,5) .. (365,15) ;
%Curve Lines [id:da14407002671657565] 
\draw    (335,15) .. controls (335.2,25) and (365.2,25) .. (365,15) ;
%Straight Lines [id:da14885157283840877] 
\draw    (335,15) -- (335,40) ;
%Curve Lines [id:da5072361841464319] 
\draw  [dash pattern={on 0.84pt off 2.51pt}]  (375,40) .. controls (375.2,30) and (405.2,30) .. (405,40) ;
%Straight Lines [id:da14605850104519813] 
\draw    (405,15) -- (405,40) ;
%Curve Lines [id:da4444223444367593] 
\draw    (375,15) .. controls (375.2,5) and (405.2,5) .. (405,15) ;
%Curve Lines [id:da7491962813851899] 
\draw    (375,15) .. controls (375.2,25) and (405.2,25) .. (405,15) ;
%Straight Lines [id:da7519288880219005] 
\draw    (375,15) -- (375,40) ;
%Curve Lines [id:da5287722207200898] 
\draw    (425,75) .. controls (425.2,85) and (455.2,85) .. (455,75) ;
%Curve Lines [id:da5263076730185409] 
\draw  [dash pattern={on 0.84pt off 2.51pt}]  (425,75) .. controls (425.2,65) and (455.2,65) .. (455,75) ;
%Curve Lines [id:da8928677912108423] 
\draw    (465,75) .. controls (465.2,85) and (495.2,85) .. (495,75) ;
%Curve Lines [id:da3960058723141018] 
\draw  [dash pattern={on 0.84pt off 2.51pt}]  (465,75) .. controls (465.2,65) and (495.2,65) .. (495,75) ;
%Curve Lines [id:da8120924598030521] 
\draw    (455,15) .. controls (455.2,40.2) and (495.2,49.2) .. (495,75) ;
%Curve Lines [id:da9095006297879095] 
\draw    (425,15) .. controls (425.2,40.2) and (465.2,49.2) .. (465,75) ;
%Curve Lines [id:da3120646964874354] 
\draw    (465,15) .. controls (465.2,40.2) and (425.2,49.2) .. (425,75) ;
%Curve Lines [id:da4964924865271486] 
\draw    (495,15) .. controls (495.2,40.2) and (455.2,49.2) .. (455,75) ;
%Curve Lines [id:da4803799735444697] 
\draw    (425,15) .. controls (425.2,5) and (455.2,5) .. (455,15) ;
%Curve Lines [id:da6584753605745459] 
\draw    (465,15) .. controls (465.2,5) and (495.2,5) .. (495,15) ;
%Curve Lines [id:da42630264129863116] 
\draw  [dash pattern={on 0.84pt off 2.51pt}]  (425,100) .. controls (425.2,90) and (455.2,90) .. (455,100) ;
%Straight Lines [id:da2501489546745995] 
\draw    (455,75) -- (455,100) ;
%Straight Lines [id:da2803590795636107] 
\draw    (425,75) -- (425,100) ;
%Curve Lines [id:da08168567118450232] 
\draw  [dash pattern={on 0.84pt off 2.51pt}]  (465,100) .. controls (465.2,90) and (495.2,90) .. (495,100) ;
%Straight Lines [id:da30953551645751065] 
\draw    (495,75) -- (495,100) ;
%Straight Lines [id:da47002126440861236] 
\draw    (465,75) -- (465,100) ;
%Curve Lines [id:da422621054280337] 
\draw    (425,100) .. controls (425.2,110) and (455.2,110) .. (455,100) ;
%Curve Lines [id:da2447679242079811] 
\draw    (465,100) .. controls (465.2,110) and (495.2,110) .. (495,100) ;
%Shape: Circle [id:dp45047119257693535] 
\draw  [fill={rgb, 255:red, 0; green, 0; blue, 0 }  ,fill opacity=1 ] (350,28.75) .. controls (350,28.06) and (350.56,27.5) .. (351.25,27.5) .. controls (351.94,27.5) and (352.5,28.06) .. (352.5,28.75) .. controls (352.5,29.44) and (351.94,30) .. (351.25,30) .. controls (350.56,30) and (350,29.44) .. (350,28.75) -- cycle ;
%Shape: Circle [id:dp6524191077305939] 
\draw  [fill={rgb, 255:red, 0; green, 0; blue, 0 }  ,fill opacity=1 ] (480,88.75) .. controls (480,88.06) and (480.56,87.5) .. (481.25,87.5) .. controls (481.94,87.5) and (482.5,88.06) .. (482.5,88.75) .. controls (482.5,89.44) and (481.94,90) .. (481.25,90) .. controls (480.56,90) and (480,89.44) .. (480,88.75) -- cycle ;
%Curve Lines [id:da329759895913719] 
\draw    (425,15) .. controls (425.2,25) and (455.2,25) .. (455,15) ;
%Curve Lines [id:da7405945430368096] 
\draw    (465,15) .. controls (465.2,25) and (495.2,25) .. (495,15) ;

% Text Node
\draw (407,43.4) node [anchor=north west][inner sep=0.75pt]    {$=$};

\end{tikzpicture}
        \caption{}
        \label{fig:TQFTd_rel3}
    \end{subfigure}
    \caption{Additional relations for $\Cobd$.}
    \label{fig:TQFTd_rel}
\end{figure}

Consequently, we obtain a functor from $\Cobd$ to $(\FF[H]-\mathrm{Mod})$.
One can verify that this functor satisfies the four relations provided in \Cref{def:Cobdl}.
It therefore descends to a functor on $\Cobdl$, and subsequently to the categories of complexes $\Kobd$ and $\Kobdh$.
By applying this TQFT to $\bn{K}$, we obtain a chain complex $\CBN(K)$ of $\FF[H]$-modules whose chain homotopy type is an invariant of the link $K$.
The homology of this complex is the \emph{Bar-Natan homology}, denoted by $\BN(K)$.

\subsection{Involution on the formal Khovanov complex}

For a given involutive link $(K,\tau)$, fix a transvergent diagram $D_K$.
Each crossing $i \in \{1, \dots, n\}$ of $D_K$ has a corresponding crossing $\tau(i) \in \{1, \dots, n\}$ under the involution $\tau$.
Furthermore, for each resolution $v \in V = \{0,1\}^{\{1,\dots,n\}}$ and its associated diagram $D_K(v)$, there is a corresponding diagram $\tau D_K(v) = D_K(v\circ\tau)$.
We define a morphism $I_\tau$ consisting of cylinders, where each cylinder connects a circle of $D_K(v)$ to the corresponding circle of $\tau D_K(v)$ under $\tau$.
This $I_\tau$ is indeed a chain map, and $I_\tau^2$ is the identity map on $\bn{K}$.

Now we define the formal involutive Bar-Natan complex.
\begin{defn}
    The \emph{formal involutive Bar-Natan complex} $\eqbn{D}$ for a transvergent diagram $D$ is defined by
    \[ \eqbn{D} \coloneqq \Cone \left(\bn{D} \xrightarrow{1+I_\tau} \bn{D} \right).\]
    More precisely, the underlying object is $\bn{D} \oplus \bn{D}$ and the differential is given by $\begin{bmatrix} d & 0 \\ 1+I_\tau & d \end{bmatrix}$.
\end{defn}

In \cite{Lobb-Watson:2021}, the authors defined \emph{involutive Reidemeister moves} and proved that two transvergent diagrams represent equivalent strongly invertible links if and only if they are related by a sequence of involutive Reidemeister moves.
The following theorem guarantees that the formal involutive Bar-Natan complex is an invariant of an invertible link.

\begin{thm}[\cite{Sano:2025}]
    If two transvergent diagrams $D_1$ and $D_2$ are related by an involutive Reidemeister move, then the complexes $\eqbn{D_1}$ and $\eqbn{D_2}$ are chain homotopy equivalent.
    In particular, the chain homotopy type of $\eqbn{D_K}$ is an invariant of the strongly invertible link $K$.
\end{thm}

We denote the chain homotopy type of $\eqbn{D_K}$ by $\eqbn{K}$, which lies in the category $\Kobdh$.
As before, we apply the TQFT to obtain a chain complex $\eqCBN(K)$ of $\FF[H]$-modules.
Its homology is called the \emph{involutive Bar-Natan homology}, denoted by $\eqBN(K)$.

Before moving to the next section, we introduce some algebraic notions.

\begin{defn}
    Let $C$ and $C'$ be chain complexes over an additive category with coefficients in $\FF$.
    We write $f \simeq_h g$ if the chain maps $f$ and $g$ are homotopic via a homotopy $h$.
    \begin{itemize}
        \item A pair $\bC = (C,\tau)$ is an \emph{equivariant chain complex} if the chain map $\tau \colon C \to C$ is an involution.
        For an equivariant chain complex $\bC$, we define $\eqC = \Cone(C \xrightarrow{1+\tau} C)$.
        \item A chain map $f \colon \bC \to \bC'$ between equivariant chain complexes $\bC = (C,\tau)$, $\bC' = (C',\tau')$ is called a \emph{homotopy equivariant chain map} if there exists a homotopy $h_f$ such that $f \tau \simeq_{h_f} \tau' f$.
        \item Two homotopy equivariant chain maps $f, g$ (with homotopies $h_f, h_g$) are \emph{coherently homotopic with homotopy $h$} if $f \simeq_h g$ and there exists a homotopy $k$ such that $h \tau + \tau' h \simeq_k h_f + h_g$.
    \end{itemize}
\end{defn}

The following results from \cite{Sano:2025} are straightforward to verify.

\begin{prop}\label{prop:eq_chain}
    \begin{itemize}
        \item The composition $gf$ of two homotopy equivariant chain maps $f \colon \bC \to \bC'$ and $g \colon \bC' \to \bC''$ is also a homotopy equivariant chain map with $h_{gf} \coloneqq gh_f + h_g f$.
        \item A homotopy equivariant chain map $f \colon \bC \to \bC'$ induces a chain map $\eqf \colon \eqC \to \eqC'$ given by $\eqf = \begin{bmatrix} f & 0 \\ h_f & f \end{bmatrix}$.
        The identity map $\id_{\bC} \colon \bC \to \bC$ induces $\widetilde{\id_{\bC}} = \id_{\eqC}$. Furthermore, for any composable homotopy equivariant chain maps $f$ and $g$, we have $\widetilde{gf} = \widetilde{g}\widetilde{f}$.
        \item For two coherently homotopic homotopy equivariant chain maps $f, g \colon \bC \to \bC'$ with homotopy $h$ (with a homotopy $k$ such that $h\tau+\tau'h\simeq_k h_f+h_g$), the induced maps $\eqf, \eqg$ are chain homotopic via the homotopy $\eqh = \begin{bmatrix} h & 0 \\ k & h \end{bmatrix}$.
    \end{itemize}
\end{prop}

\subsection{Equivariant unknotting}\label{subsec:equ}

In \cite{Boyle-Chen:2026}, the authors studied the equivariant unknotting number by considering the following three types of equivariant crossing changes.

\begin{defn}[\cite{Boyle-Chen:2026}]
    We define a \emph{Type X equivariant crossing change} for $X \in \{A, B, C\}$ as follows: 
    \begin{itemize}
        \item \textbf{Type A}: Two crossing changes occur simultaneously off the axis;
        \item \textbf{Type B}: Two strands intersect on the axis at the moment of the crossing change; and
        \item \textbf{Type C}: Two fixed points move along the axis and collide at the moment of the crossing change.
    \end{itemize}
    See \Cref{fig:ecc} for a diagrammatic explanation.
\end{defn}

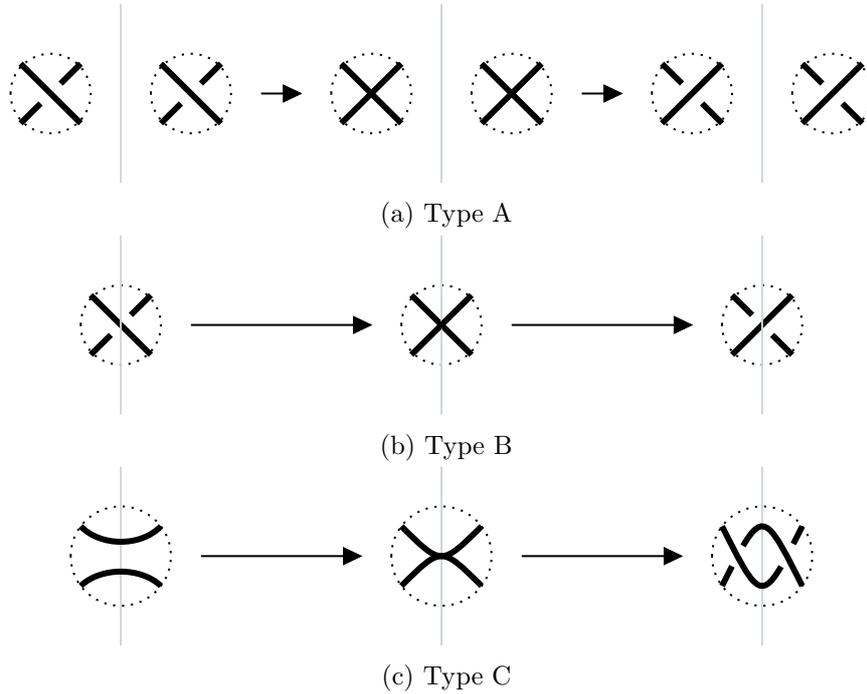
\begin{figure}
    \centering
    \begin{subfigure}{\linewidth}
        \centering
        \tikzset{every picture/.style={line width=0.75pt}} %set default line width to 0.75pt        

\begin{tikzpicture}[x=0.75pt,y=0.75pt,yscale=-1,xscale=1]
%uncomment if require: \path (0,91); %set diagram left start at 0, and has height of 91

%Shape: Ellipse [id:dp7129725853015582] 
\draw  [dash pattern={on 0.84pt off 2.51pt}] (15,45) .. controls (15,33.95) and (23.95,25) .. (35,25) .. controls (46.05,25) and (55,33.95) .. (55,45) .. controls (55,56.05) and (46.05,65) .. (35,65) .. controls (23.95,65) and (15,56.05) .. (15,45) -- cycle ;
%Straight Lines [id:da32525937368977365] 
\draw [line width=2.25]    (20,30) -- (50,60) ;
%Straight Lines [id:da5214176338829674] 
\draw [line width=2.25]    (50,30) -- (40,40) ;
%Straight Lines [id:da7655324719191758] 
\draw [line width=2.25]    (30,50) -- (20,60) ;
%Straight Lines [id:da8157449109350116] 
\draw [line width=2.25]    (90,30) -- (120,60) ;
%Straight Lines [id:da5577492902540905] 
\draw [line width=2.25]    (120,30) -- (110,40) ;
%Straight Lines [id:da07461193810362554] 
\draw [line width=2.25]    (100,50) -- (90,60) ;
%Straight Lines [id:da5638706749745757] 
\draw [color={rgb, 255:red, 155; green, 155; blue, 155 }  ,draw opacity=1 ]   (70,0) -- (70,90) ;
%Shape: Ellipse [id:dp4997282523222496] 
\draw  [dash pattern={on 0.84pt off 2.51pt}] (85,45) .. controls (85,33.95) and (93.95,25) .. (105,25) .. controls (116.05,25) and (125,33.95) .. (125,45) .. controls (125,56.05) and (116.05,65) .. (105,65) .. controls (93.95,65) and (85,56.05) .. (85,45) -- cycle ;
%Straight Lines [id:da7701654329367813] 
\draw    (140,45) -- (157,45) ;
\draw [shift={(160,45)}, rotate = 180] [fill={rgb, 255:red, 0; green, 0; blue, 0 }  ][line width=0.08]  [draw opacity=0] (8.93,-4.29) -- (0,0) -- (8.93,4.29) -- cycle    ;
%Straight Lines [id:da14796024210841086] 
\draw [line width=2.25]    (180,30) -- (210,60) ;
%Shape: Ellipse [id:dp4800348050963503] 
\draw  [dash pattern={on 0.84pt off 2.51pt}] (175,45) .. controls (175,33.95) and (183.95,25) .. (195,25) .. controls (206.05,25) and (215,33.95) .. (215,45) .. controls (215,56.05) and (206.05,65) .. (195,65) .. controls (183.95,65) and (175,56.05) .. (175,45) -- cycle ;
%Straight Lines [id:da2848396192305246] 
\draw [color={rgb, 255:red, 155; green, 155; blue, 155 }  ,draw opacity=1 ]   (230,0) -- (230,90) ;
%Straight Lines [id:da5026070303845216] 
\draw [line width=2.25]    (210,30) -- (180,60) ;
%Straight Lines [id:da6434420655886517] 
\draw [line width=2.25]    (250,30) -- (280,60) ;
%Straight Lines [id:da053946273654312926] 
\draw [line width=2.25]    (280,30) -- (250,60) ;
%Shape: Ellipse [id:dp8045242693851076] 
\draw  [dash pattern={on 0.84pt off 2.51pt}] (245,45) .. controls (245,33.95) and (253.95,25) .. (265,25) .. controls (276.05,25) and (285,33.95) .. (285,45) .. controls (285,56.05) and (276.05,65) .. (265,65) .. controls (253.95,65) and (245,56.05) .. (245,45) -- cycle ;
%Straight Lines [id:da5261504690560529] 
\draw [line width=2.25]    (370,30) -- (340,60) ;
%Straight Lines [id:da504948697702246] 
\draw    (300,45) -- (317,45) ;
\draw [shift={(320,45)}, rotate = 180] [fill={rgb, 255:red, 0; green, 0; blue, 0 }  ][line width=0.08]  [draw opacity=0] (8.93,-4.29) -- (0,0) -- (8.93,4.29) -- cycle    ;
%Straight Lines [id:da5053310320443148] 
\draw [line width=2.25]    (340,30) -- (350,40) ;
%Shape: Ellipse [id:dp9866697790075494] 
\draw  [dash pattern={on 0.84pt off 2.51pt}] (335,45) .. controls (335,33.95) and (343.95,25) .. (355,25) .. controls (366.05,25) and (375,33.95) .. (375,45) .. controls (375,56.05) and (366.05,65) .. (355,65) .. controls (343.95,65) and (335,56.05) .. (335,45) -- cycle ;
%Straight Lines [id:da23954218512330194] 
\draw [color={rgb, 255:red, 155; green, 155; blue, 155 }  ,draw opacity=1 ]   (390,0) -- (390,90) ;
%Shape: Ellipse [id:dp1390445890390265] 
\draw  [dash pattern={on 0.84pt off 2.51pt}] (405,45) .. controls (405,33.95) and (413.95,25) .. (425,25) .. controls (436.05,25) and (445,33.95) .. (445,45) .. controls (445,56.05) and (436.05,65) .. (425,65) .. controls (413.95,65) and (405,56.05) .. (405,45) -- cycle ;
%Straight Lines [id:da5295022494117545] 
\draw [line width=2.25]    (360,50) -- (370,60) ;
%Straight Lines [id:da9723206292294044] 
\draw [line width=2.25]    (440,30) -- (410,60) ;
%Straight Lines [id:da6962609039770016] 
\draw [line width=2.25]    (410,30) -- (420,40) ;
%Straight Lines [id:da8229255601358788] 
\draw [line width=2.25]    (430,50) -- (440,60) ;

\end{tikzpicture}
        \caption{Type A}
        \label{fig:Aecc}
    \end{subfigure}
    \hfill
    \begin{subfigure}{\linewidth}
        \centering
        \tikzset{every picture/.style={line width=0.75pt}} %set default line width to 0.75pt        

\begin{tikzpicture}[x=0.75pt,y=0.75pt,yscale=-1,xscale=1]
%uncomment if require: \path (0,91); %set diagram left start at 0, and has height of 91

%Shape: Ellipse [id:dp8134494502770792] 
\draw  [dash pattern={on 0.84pt off 2.51pt}] (50,45) .. controls (50,33.95) and (58.95,25) .. (70,25) .. controls (81.05,25) and (90,33.95) .. (90,45) .. controls (90,56.05) and (81.05,65) .. (70,65) .. controls (58.95,65) and (50,56.05) .. (50,45) -- cycle ;
%Straight Lines [id:da17457294473432028] 
\draw [line width=2.25]    (55,30) -- (85,60) ;
%Straight Lines [id:da09450875443723306] 
\draw [line width=2.25]    (85,30) -- (75,40) ;
%Straight Lines [id:da08218313825005763] 
\draw [line width=2.25]    (65,50) -- (55,60) ;
%Straight Lines [id:da520640915488156] 
\draw [color={rgb, 255:red, 155; green, 155; blue, 155 }  ,draw opacity=1 ]   (70,0) -- (70,90) ;
%Straight Lines [id:da2777970370562418] 
\draw    (105,45) -- (192,45) ;
\draw [shift={(195,45)}, rotate = 180] [fill={rgb, 255:red, 0; green, 0; blue, 0 }  ][line width=0.08]  [draw opacity=0] (8.93,-4.29) -- (0,0) -- (8.93,4.29) -- cycle    ;
%Straight Lines [id:da5055573539255788] 
\draw [line width=2.25]    (215,30) -- (245,60) ;
%Shape: Ellipse [id:dp42893298517343414] 
\draw  [dash pattern={on 0.84pt off 2.51pt}] (210,45) .. controls (210,33.95) and (218.95,25) .. (230,25) .. controls (241.05,25) and (250,33.95) .. (250,45) .. controls (250,56.05) and (241.05,65) .. (230,65) .. controls (218.95,65) and (210,56.05) .. (210,45) -- cycle ;
%Straight Lines [id:da18573615103133856] 
\draw [color={rgb, 255:red, 155; green, 155; blue, 155 }  ,draw opacity=1 ]   (230,0) -- (230,90) ;
%Straight Lines [id:da6704711811056836] 
\draw [line width=2.25]    (245,30) -- (215,60) ;
%Straight Lines [id:da08019408623490909] 
\draw [line width=2.25]    (405,30) -- (375,60) ;
%Straight Lines [id:da3836087906660792] 
\draw    (265,45) -- (352,45) ;
\draw [shift={(355,45)}, rotate = 180] [fill={rgb, 255:red, 0; green, 0; blue, 0 }  ][line width=0.08]  [draw opacity=0] (8.93,-4.29) -- (0,0) -- (8.93,4.29) -- cycle    ;
%Straight Lines [id:da7596274264023757] 
\draw [line width=2.25]    (375,30) -- (385,40) ;
%Shape: Ellipse [id:dp9823480077205194] 
\draw  [dash pattern={on 0.84pt off 2.51pt}] (370,45) .. controls (370,33.95) and (378.95,25) .. (390,25) .. controls (401.05,25) and (410,33.95) .. (410,45) .. controls (410,56.05) and (401.05,65) .. (390,65) .. controls (378.95,65) and (370,56.05) .. (370,45) -- cycle ;
%Straight Lines [id:da7914373816864106] 
\draw [color={rgb, 255:red, 155; green, 155; blue, 155 }  ,draw opacity=1 ]   (390,0) -- (390,90) ;
%Straight Lines [id:da08487785477542986] 
\draw [line width=2.25]    (395,50) -- (405,60) ;

\end{tikzpicture}
        \caption{Type B}
        \label{fig:Becc}
    \end{subfigure}
    \hfill
    \begin{subfigure}{\linewidth}
        \centering
        \tikzset{every picture/.style={line width=0.75pt}} %set default line width to 0.75pt        

\begin{tikzpicture}[x=0.75pt,y=0.75pt,yscale=-1,xscale=1]
%uncomment if require: \path (0,91); %set diagram left start at 0, and has height of 91

%Shape: Ellipse [id:dp7978326166525193] 
\draw  [dash pattern={on 0.84pt off 2.51pt}] (45,45) .. controls (45,31.19) and (56.19,20) .. (70,20) .. controls (83.81,20) and (95,31.19) .. (95,45) .. controls (95,58.81) and (83.81,70) .. (70,70) .. controls (56.19,70) and (45,58.81) .. (45,45) -- cycle ;
%Straight Lines [id:da9191563895491562] 
\draw [color={rgb, 255:red, 155; green, 155; blue, 155 }  ,draw opacity=1 ]   (70,0) -- (70,90) ;
%Straight Lines [id:da40878929899292227] 
\draw    (110,45) -- (187,45) ;
\draw [shift={(190,45)}, rotate = 180] [fill={rgb, 255:red, 0; green, 0; blue, 0 }  ][line width=0.08]  [draw opacity=0] (8.93,-4.29) -- (0,0) -- (8.93,4.29) -- cycle    ;
%Straight Lines [id:da06686098573731725] 
\draw [color={rgb, 255:red, 155; green, 155; blue, 155 }  ,draw opacity=1 ]   (230,0) -- (230,90) ;
%Straight Lines [id:da46064454285764966] 
\draw    (270,45) -- (347,45) ;
\draw [shift={(350,45)}, rotate = 180] [fill={rgb, 255:red, 0; green, 0; blue, 0 }  ][line width=0.08]  [draw opacity=0] (8.93,-4.29) -- (0,0) -- (8.93,4.29) -- cycle    ;
%Straight Lines [id:da07364474930761011] 
\draw [color={rgb, 255:red, 155; green, 155; blue, 155 }  ,draw opacity=1 ]   (390,0) -- (390,90) ;
%Curve Lines [id:da0619029263664016] 
\draw [line width=2.25]    (50,30) .. controls (60,40.4) and (80,40.4) .. (90,30) ;
%Curve Lines [id:da6981346949130683] 
\draw [line width=2.25]    (50,60) .. controls (60,50.4) and (80,50.4) .. (90,60) ;
%Shape: Ellipse [id:dp9077536106038533] 
\draw  [dash pattern={on 0.84pt off 2.51pt}] (205,45) .. controls (205,31.19) and (216.19,20) .. (230,20) .. controls (243.81,20) and (255,31.19) .. (255,45) .. controls (255,58.81) and (243.81,70) .. (230,70) .. controls (216.19,70) and (205,58.81) .. (205,45) -- cycle ;
%Curve Lines [id:da543272593282894] 
\draw [line width=2.25]    (210,30) .. controls (230,50) and (230,50) .. (250,30) ;
%Curve Lines [id:da7458913165681471] 
\draw [line width=2.25]    (210,60) .. controls (230,40) and (230,40) .. (250,60) ;
%Shape: Ellipse [id:dp9327547135065456] 
\draw  [dash pattern={on 0.84pt off 2.51pt}] (365,45) .. controls (365,31.19) and (376.19,20) .. (390,20) .. controls (403.81,20) and (415,31.19) .. (415,45) .. controls (415,58.81) and (403.81,70) .. (390,70) .. controls (376.19,70) and (365,58.81) .. (365,45) -- cycle ;
%Curve Lines [id:da262561145395221] 
\draw [line width=2.25]    (370,30) .. controls (385.69,61.38) and (389.07,68.14) .. (399.46,50.28) ;
%Curve Lines [id:da8349738680967133] 
\draw [line width=2.25]    (380.46,39.86) .. controls (390.91,21.82) and (394.27,28.53) .. (410,60) ;
%Curve Lines [id:da3608627434437217] 
\draw [line width=2.25]    (370,60) .. controls (371.81,56.37) and (373.46,53.07) .. (374.98,50.1) ;
%Curve Lines [id:da8765355868406559] 
\draw [line width=2.25]    (405.54,38.88) .. controls (406.91,36.18) and (408.39,33.22) .. (410,30) ;

\end{tikzpicture}
        \caption{Type C}
        \label{fig:Cecc}
    \end{subfigure}
    \caption{Three types of equivariant crossing change.}
    \label{fig:ecc}
\end{figure}

\begin{defn}
    The \emph{equivariant unknotting number} $\equ(K)$ of a strongly invertible knot $K$ is defined by
    \[ \equ(K) \coloneqq \min\left\{ 2u_A + u_B + u_C \:\middle|\: \substack{  K \text{ can be unknotted by } u_A \text{ Type A, } u_B \text{ Type B, and } \\ u_C \text{ Type C equivariant crossing changes} } \right\}. \]
\end{defn}

We briefly describe two lower bounds for the equivariant unknotting number.
Recall that the unknotting number $u(K)$ is the minimal number of self-intersections occurring during a homotopy that transforms $K$ into the unknot.
By definition, this yields the trivial lower bound $u(K) \leq \equ(K)$.

Furthermore, recall that the \emph{smooth $4$-genus} $g_4(K)$ of a knot $K$ is the minimal genus of a surface in $D^4$ that bounds a given knot $K \subset \partial D^4$.
It is a well-known fact that $g_4(K)$ provides a lower bound for $u(K)$.
This is proven by resolving the singularities of an immersed surface $\Sigma \subset S^3 \times I$, where the slices $\Sigma_t = \Sigma \cap (S^3 \times \{t\})$ represent the unknotting process from $K$ to the unknot $U$.
Analogously, we can bound the equivariant unknotting number using an equivariant version of the smooth $4$-genus.
While several definitions exist (\cite{Miller-Powell:2023,Dai-Mallick-Stoffregen:2023,Boyle-Issa:2022,Borodzik-Dai-Mallick-Stoffregen:2025}), we introduce the \emph{simple equivariant (smooth) $4$-genus} $\seq(K)$, defined as the minimal genus of a smooth cobordism $\Sigma \subset S^3 \times I$ between $(K,\tau)$ and the unknot $U$ that is invariant under the involution $\tau \times \id$ on $S^3\times I$.
By the same reasoning as in the ordinary case, the inequality $\seq(K) \leq \equ(K)$ holds.
In \cite{Sano:2025}, the author defined two equivariant Rasmussen invariants, $\us(K)$ and $\ls(K)$, and showed that both $|\us(K)|/2$ and $|\ls(K)|/2$ provide lower bounds for $\seq(K)$.
Consequently, one can obtain a lower bound for $\equ(K)$ by computing $\us(K)$ and $\ls(K)$ from $\eqBN(K)$.

    \section{Main Result}\label{sec:thm}

In this section, we prove the main result of this paper.
The following proposition is central to the proof of \Cref{thm:main}.

\begin{prop}\label{prop:eq_cc}
    Suppose a strongly invertible knot $J$ is obtained from a strongly invertible knot $K$ by a single Type $X$ equivariant crossing change, where $X \in \{A, B, C\}$.
    Then there exist chain maps $\eqCBN(K) \xrightleftharpoons[g]{f} \eqCBN(J)$ such that $g \circ f \simeq H^t \cdot \id_{\eqCBN(K)}$ and $f \circ g \simeq H^t \cdot \id_{\eqCBN(J)}$, where $t=2$ for $X=A$ and $t=1$ for $X = B$ or $X = C$.
\end{prop}

\begin{defn}\label{def:ord}
    For an $\FF[H]$-module $M$, we say an element $x \in M$ is \emph{$H$-torsion} if there exists $n \geq 0$ such that $H^n \cdot x = 0$.
    The minimal such $n$ is called the \emph{$H$-torsion order} of $x$ and is denoted by $\ord(x)$.
    We define the \emph{maximal $H$-torsion order} of $M$ as
    \[ \ord(M) \coloneqq \max \{ \ord(x) \mid x \in M \text{ is } H\text{-torsion} \}. \]
\end{defn}

Note that $\ord(M) = 0$ if $M$ contains no non-zero $H$-torsion elements.
For brevity, we write $\ord(K)$ for $\ord(\BN(K))$ (resp. $\eqord(K)$ for $\ord(\eqBN(K))$) for a knot (resp. strongly invertible knot) $K$.

Before proving \Cref{prop:eq_cc}, we demonstrate how \Cref{thm:main} follows from it, utilizing the following purely algebraic lemma.

\begin{lem}[{\cite[Lemma~3.1]{Alishahi:2019}}]\label{lem:ord}
    Let $C$ and $C'$ be two chain complexes over $\FF[H]$.
    Suppose there exist chain maps $C \xrightleftharpoons[g]{f} C'$ such that $g \circ f \simeq H^t \cdot \id_{C}$ and $f \circ g \simeq H^t \cdot \id_{C'}$ for some $t \geq 0$.
    Then $|\ord(H_*(C)) - \ord(H_*(C'))| \leq t$.
\end{lem}
\begin{proof}
    For an $H$-torsion element $x \in H_*(C)$, its image $f_*(x) \in H_*(C')$ is also an $H$-torsion element.
    Since a module homomorphism cannot increase the $H$-torsion order, we have
    \[ \ord(f_*(x)) \geq \ord(g_*(f_*(x))) = \ord(H^t \cdot x) \geq \ord(x) - t, \]
    which implies
    \[ t \geq \ord(x) - \ord(f_*(x)) \geq \ord(x) - \ord(H_*(C')). \]
    Taking the maximum over all $x \in H_*(C)$, we obtain $t \geq \ord(H_*(C)) - \ord(H_*(C'))$.
    By symmetry, it follows that $|\ord(H_*(C)) - \ord(H_*(C'))| \leq t$.
\end{proof}

\begin{proof}[Proof of \Cref{thm:main}]
    The theorem follows from \Cref{prop:eq_cc}, \Cref{lem:ord}, and the fact that $\eqord(U) = 0$.
\end{proof}

\begin{rmk}
    In general, the same proof establishes a lower bound for the \emph{equivariant Gordian distance} $\tilde{d}(K,J)$ between two given strongly invertible knots $K$ and $J$, defined as the minimal number of self-intersections occurring during an involution-invariant homotopy that transforms $K$ into $J$.
    Indeed, the inequality $|\eqord(K)-\eqord(J)| \leq \tilde{d}(K,J)$ holds. 
\end{rmk}

Now we prove \Cref{prop:eq_cc} by considering each case separately.
Let $K'$ be a transvergent diagram obtained from $K$ by a Type $X$ equivariant crossing change.\footnote{Henceforth, we use $K$ and $K'$ to denote both the knots and their diagrams. The letter $D$ will refer to a specific morphism.}
We will show that:
\begin{enumerate}
    \item There exist two maps $\bn{K} \xrightleftharpoons[g]{f} \bn{K'}$ that are strictly $I_\tau$-equivariant (i.e., $I_\tau f = fI_\tau$ and $I_\tau g = g I_\tau$); so that $f,g$ are homotopy equivariant chain maps via $h_f=0$ and $h_g=0$.
    \item There exist homotopies $h, h'$ such that $gf \simeq_h H^t \cdot \id_{\bn{K}}$ and $fg \simeq_{h'} H^t\cdot \id_{\bn{K'}}$ (where $t=2$ if $X=A$, otherwise $t=1$).
    \item There exist homotopies $k, k'$ such that $I_\tau h + h I_\tau \simeq_k 0$ and $I_\tau h' + h' I_\tau \simeq_{k'} 0$.
\end{enumerate}
By \Cref{prop:eq_chain}, this implies \Cref{prop:eq_cc}.

The maps $f$ and $g$ are based on Alishahi's construction \cite{Alishahi:2019}.
Suppose $K'$ is obtained from $K$ by a crossing change at $c$. Let $K_i$ be the diagram where $c$ is $i$-resolved for $i \in \{0, 1\}$.
We can represent $\bn{K}$ and $\bn{K'}$ as mapping cones:
\[ \bn{K} = \Cone\left(\bn{K_0} \xrightarrow{S} \bn{K_1}\right), \qquad \bn{K'} = \Cone\left(\bn{K_1} \xrightarrow{S} \bn{K_0}\right), \]
where $S$ denotes the local saddle map between the two resolutions.\footnote{While there are two saddle maps $S_0 \colon \bn{K_0}\to\bn{K_1}$ and $S_1\colon\bn{K_1}\to\bn{K_0}$, we omit subscripts when the direction is clear.}

\begin{figure}
    \centering
    \tikzset{every picture/.style={line width=0.75pt}} %set default line width to 0.75pt        

\begin{tikzpicture}[x=0.75pt,y=0.75pt,yscale=-1,xscale=1]
%uncomment if require: \path (0,120); %set diagram left start at 0, and has height of 120

%Shape: Ellipse [id:dp28522696791120605] 
\draw   (170,45) .. controls (170,25.67) and (185.67,10) .. (205,10) .. controls (224.33,10) and (240,25.67) .. (240,45) .. controls (240,64.33) and (224.33,80) .. (205,80) .. controls (185.67,80) and (170,64.33) .. (170,45) -- cycle ;
%Curve Lines [id:da11337562565736858] 
\draw [line width=2.25]    (230,20) .. controls (210,39.8) and (200,39.8) .. (180,20) ;
%Curve Lines [id:da5865125368127954] 
\draw [line width=2.25]    (230,70) .. controls (210,49.8) and (200,49.8) .. (180,70) ;
%Shape: Ellipse [id:dp8923712295103718] 
\draw   (10,45) .. controls (10,25.67) and (25.67,10) .. (45,10) .. controls (64.33,10) and (80,25.67) .. (80,45) .. controls (80,64.33) and (64.33,80) .. (45,80) .. controls (25.67,80) and (10,64.33) .. (10,45) -- cycle ;
%Shape: Ellipse [id:dp8282360612950087] 
\draw   (250,45) .. controls (250,25.67) and (265.67,10) .. (285,10) .. controls (304.33,10) and (320,25.67) .. (320,45) .. controls (320,64.33) and (304.33,80) .. (285,80) .. controls (265.67,80) and (250,64.33) .. (250,45) -- cycle ;
%Curve Lines [id:da3898804577658781] 
\draw [line width=2.25]    (260,70) .. controls (280,50) and (280,39.8) .. (260,20) ;
%Curve Lines [id:da05984028829261512] 
\draw [line width=2.25]    (310,70) .. controls (290,49.8) and (290,40) .. (310,20) ;
%Straight Lines [id:da7570286604206071] 
\draw [line width=2.25]    (20,20) -- (70,70) ;
%Straight Lines [id:da03659043077012458] 
\draw [line width=2.25]    (20,70) -- (40,50) ;
%Straight Lines [id:da6836886888322646] 
\draw [line width=2.25]    (50,40) -- (70,20) ;
%Shape: Ellipse [id:dp45023777641201346] 
\draw   (90,45) .. controls (90,25.67) and (105.67,10) .. (125,10) .. controls (144.33,10) and (160,25.67) .. (160,45) .. controls (160,64.33) and (144.33,80) .. (125,80) .. controls (105.67,80) and (90,64.33) .. (90,45) -- cycle ;
%Straight Lines [id:da3954401938879486] 
\draw [line width=2.25]    (100,20) -- (120,40) ;
%Straight Lines [id:da31493321382393646] 
\draw [line width=2.25]    (100,70) -- (150,20) ;
%Straight Lines [id:da5375616234839437] 
\draw [line width=2.25]    (130,50) -- (150,70) ;

% Text Node
\draw (41,92.4) node [anchor=north west][inner sep=0.75pt]    {$K$};
% Text Node
\draw (119,92.4) node [anchor=north west][inner sep=0.75pt]    {$K'$};
% Text Node
\draw (199,92.4) node [anchor=north west][inner sep=0.75pt]    {$K_{0}$};
% Text Node
\draw (277,92.4) node [anchor=north west][inner sep=0.75pt]    {$K_{1}$};

\end{tikzpicture}
    \caption{Local pictures of diagrams $K, K', K_0$, and $K_1$.}
    \label{fig:ali_diag}
\end{figure}
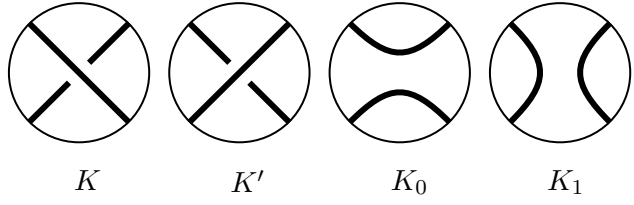

Define maps $\bn{K} \xrightleftharpoons[g_0]{f_0} \bn{K'}$ using the following diagram:
% https://tikzcd.yichuanshen.de/#N4Igdg9gJgpgziAXAbVABwnAlgFyxMJZAJgBpiBdUkANwEMAbAVxiRAB12AjMYAaQD6ARgC+IEaXSZc+QigAs5KrUYs2nHvwEAGMRKnY8BIou3L6zVog7deg3eMkgMh2UQBspM9QtrrGu2E9JxcZYxQhJR9VKxAAXkcDMLlkSO8VSzYE-WdpIxTtL3MY9Vt+YKT8okLKaMz-Mr4Acj1lGCgAc3giUAAzACcIAFskMhAcCCQAZjq-EABlRJAB4aRC8cnESJAGOi4YBgAFPLdrfqwOgAscEFnYxZyVkcQxiaRt-bAoafXfWM4sN9HoNnts3i9qJ9vogpr8StYACJLJ7Tajg9ZQn53NhI4GrRDrcEzECYmFw+o2QHIkFITwbJAAdkhMC+WIyc161PxTPpiDppNh2OsHXEFBEQA
\[
\begin{tikzcd}
\bn{K} \arrow[dd, "f_0", shift left, red]  & = &                                                            &  & \bn{K_0} \arrow[rr, "S"] \arrow[dd, "\id", shift left, red, near start] &  & \bn{K_1} \arrow[lllldd, "D", shift left, red, near start] \\
                                   &   &                                                            &  &                                                       &  &                                         \\
\bn{K'} \arrow[uu, "g_0", shift left, blue] & = & \bn{K_1} \arrow[rr, "S"'] \arrow[rrrruu, "\id", shift left, blue, near start] &  & \bn{K_0} \arrow[uu, "D", shift left, blue, near start]                   &  &                                        
\end{tikzcd}
\]
Here, $D$ is a cobordism map consisting of the sum of two cylinders, each with a dot placed on either side of the crossing $c$ (see \Cref{fig:D}).\footnote{Again, we omit subscripts for $D$.}

Regarding the maps $S$ and $D$, we note the following: if the two arcs are connected outside the local picture, $D$ vanishes because the dot can be moved from one side to the other.
Similarly, $SD=0$ and $DS=0$.
Finally, the relation (NC) is expressed as $S^2 = D + H \cdot \id$.

The composition $g_0f_0 + H\cdot\id\colon \bn{K} \to \bn{K}$ can be visualized as:
% https://tikzcd.yichuanshen.de/#N4Igdg9gJgpgziAXAbVABwnAlgFyxMJZABgBpiBdUkANwEMAbAVxiRAB12AjMYAaQC+IAaXSZc+QijIAmKrUYs2nHvyEix2PASIBGcvPrNWiEAF5hokBi2S9pOdSNLTFjdfHapyGQaeKTDm5ePgB9YnUrGwkdFAAWPwVjZWD+UN1IzRjvX0cklyDVMIjLLK8iBLznQJUQ9PV5GCgAc3giUAAzACcIAFskMhAcCCR9EC4YMCgkAGZB6rZmjtKQbr6B6mHR6gmppABaOf9k0wAJTgBjKAgcAAJOLGnqBjoJhgAFTztTLqxmgAscCs1v1EAkhiNEABWY4FADKwJ6oIAbJtIQB2Z6vGAfL6xEC-AFA2GBBHuEFIcFbRCo8aTaaII75QIAEUR6zBaKQtN2DMOgxeb0+tnxhMBIBJbHO7CuN3u7Ee7NBMIhSExdL2iH5WKFeKkBL+4slZ3lsruD2m5KRSBV1PVvNm8wCbDZVo5tshtIWpgRzywYECUDocH+TWEFAEQA
\[\begin{tikzcd}
\bn{K} \arrow[dd, "g_0f_0+H\cdot \id", red]& = & \bn{K_0} \arrow[rr, "S"] \arrow[dd, "D+H\cdot \id",red]&  & \bn{K_1}\arrow[dd, "D+H\cdot \id", red] \arrow[lldd, "S" description, dashed, green] \\
                                                                         &   &                                                                                           &  &                                                                                                      \\
\bn{K}                                                                   & = & \bn{K_0} \arrow[rr, "S"']                                                                 &  & \bn{K_1}                                                                                            
\end{tikzcd}\]

Using a homotopy $h_0$ defined by the saddle map $S \colon \bn{K_1} \to \bn{K_0}$, one can verify $g_0f_0 \simeq_{h_0} H\cdot \id$ via the relation $S^2=D+H\cdot\id$.
The proof for $f_0g_0 \simeq_{h_0'} H\cdot \id$ is analogous.

\begin{figure}
    \centering
    $D=\includegraphics[valign=c,scale=.2]{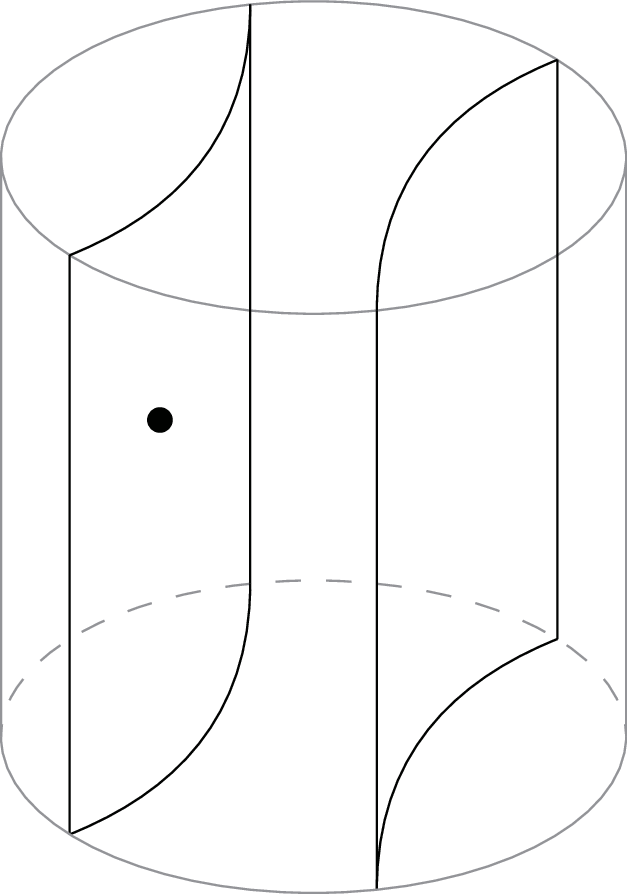}+\includegraphics[valign=c,scale=.2]{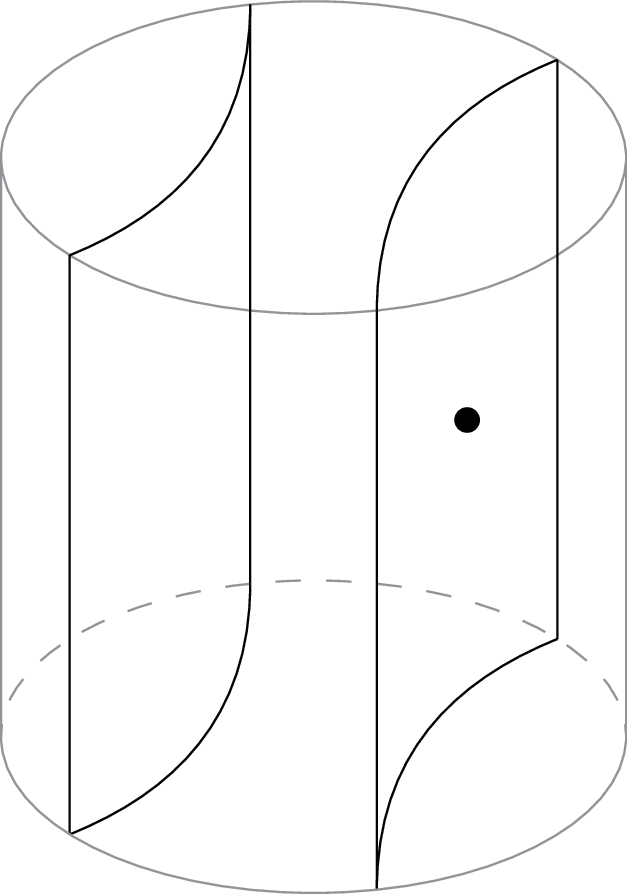}$
    \caption{The chain map $D$.}
    \label{fig:D}
\end{figure}

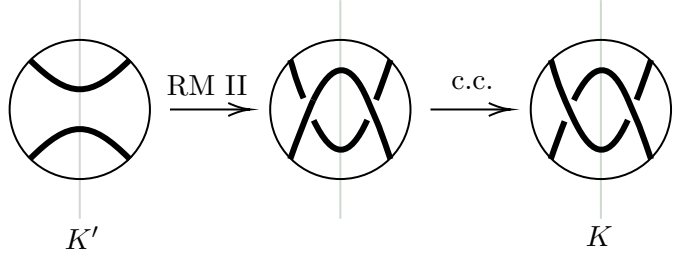
\begin{figure}
    \centering
    \tikzset{every picture/.style={line width=0.75pt}} %set default line width to 0.75pt        

\begin{tikzpicture}[x=0.75pt,y=0.75pt,yscale=-1,xscale=1]
%uncomment if require: \path (0,164); %set diagram left start at 0, and has height of 164

%Straight Lines [id:da004976924823152484] 
\draw [color={rgb, 255:red, 155; green, 155; blue, 155 }  ,draw opacity=1 ]   (45,10) -- (45,120) ;
%Straight Lines [id:da653119681161629] 
\draw [color={rgb, 255:red, 155; green, 155; blue, 155 }  ,draw opacity=1 ]   (305,10) -- (305,120) ;
%Shape: Ellipse [id:dp6467440145870536] 
\draw   (10,65) .. controls (10,45.67) and (25.67,30) .. (45,30) .. controls (64.33,30) and (80,45.67) .. (80,65) .. controls (80,84.33) and (64.33,100) .. (45,100) .. controls (25.67,100) and (10,84.33) .. (10,65) -- cycle ;
%Curve Lines [id:da8553148323375497] 
\draw [line width=2.25]    (70,40) .. controls (50,59.8) and (40,59.8) .. (20,40) ;
%Curve Lines [id:da8568788339484532] 
\draw [line width=2.25]    (70,90) .. controls (50,69.8) and (40,69.8) .. (20,90) ;
%Shape: Ellipse [id:dp8765567400542256] 
\draw   (270,65) .. controls (270,45.67) and (285.67,30) .. (305,30) .. controls (324.33,30) and (340,45.67) .. (340,65) .. controls (340,84.33) and (324.33,100) .. (305,100) .. controls (285.67,100) and (270,84.33) .. (270,65) -- cycle ;
%Curve Lines [id:da47221401513062955] 
\draw [line width=2.25]    (330,40) .. controls (327.45,47.69) and (325.06,54.4) .. (322.8,60.13) ;
%Curve Lines [id:da4483392137391803] 
\draw [line width=2.25]    (330,90) .. controls (328.22,84.65) and (326.52,79.78) .. (324.89,75.39) .. controls (323.25,71) and (327.19,80.47) .. (318.65,60.36) .. controls (310.1,40.25) and (301.26,40.18) .. (291.51,60.34) ;
%Curve Lines [id:da5559523722472429] 
\draw [line width=2.25]    (318.43,70.29) .. controls (309.6,90) and (301.26,90.4) .. (291.68,70.46) .. controls (282.1,50.52) and (281.54,43.9) .. (280,40) ;
%Curve Lines [id:da5375304767209766] 
\draw [line width=2.25]    (286.9,70.97) .. controls (285.3,74.98) and (282.44,82.78) .. (280,90) ;
%Straight Lines [id:da38890407563518503] 
\draw [color={rgb, 255:red, 155; green, 155; blue, 155 }  ,draw opacity=1 ]   (175,10) -- (175,120) ;
%Shape: Ellipse [id:dp5973047821452704] 
\draw   (140,65) .. controls (140,45.67) and (155.67,30) .. (175,30) .. controls (194.33,30) and (210,45.67) .. (210,65) .. controls (210,84.33) and (194.33,100) .. (175,100) .. controls (155.67,100) and (140,84.33) .. (140,65) -- cycle ;
%Curve Lines [id:da5051977011373735] 
\draw [line width=2.25]    (200,40) .. controls (197.45,47.69) and (195.06,54.4) .. (192.8,60.13) ;
%Curve Lines [id:da9181662735948685] 
\draw [line width=2.25]    (200,90) .. controls (198.22,84.65) and (196.52,79.78) .. (194.89,75.39) .. controls (193.25,71) and (197.19,80.47) .. (188.65,60.36) .. controls (180.1,40.25) and (171.26,40.18) .. (161.51,60.34) ;
%Curve Lines [id:da09937756655683128] 
\draw [line width=2.25]    (156.65,59.35) .. controls (151.82,47.8) and (151.21,43.06) .. (150,40) ;
%Curve Lines [id:da2644053333430587] 
\draw [line width=2.25]    (161.51,60.34) .. controls (159,65.75) and (152.44,82.78) .. (150,90) ;
%Straight Lines [id:da8474350781393205] 
\draw    (90,65) -- (128,65) ;
\draw [shift={(130,65)}, rotate = 180] [color={rgb, 255:red, 0; green, 0; blue, 0 }  ][line width=0.75]    (10.93,-3.29) .. controls (6.95,-1.4) and (3.31,-0.3) .. (0,0) .. controls (3.31,0.3) and (6.95,1.4) .. (10.93,3.29)   ;
%Curve Lines [id:da3817132490967271] 
\draw [line width=2.25]    (188.43,70.29) .. controls (179.6,90) and (171.26,90.4) .. (161.68,70.46) ;
%Straight Lines [id:da4566362378965857] 
\draw    (220,65) -- (258,65) ;
\draw [shift={(260,65)}, rotate = 180] [color={rgb, 255:red, 0; green, 0; blue, 0 }  ][line width=0.75]    (10.93,-3.29) .. controls (6.95,-1.4) and (3.31,-0.3) .. (0,0) .. controls (3.31,0.3) and (6.95,1.4) .. (10.93,3.29)   ;

% Text Node
\draw (296,122.4) node [anchor=north west][inner sep=0.75pt]    {$K$};
% Text Node
\draw (36,122.4) node [anchor=north west][inner sep=0.75pt]    {$K'$};
% Text Node
\draw (87,46) node [anchor=north west][inner sep=0.75pt]   [align=left] {RM II};
% Text Node
\draw (229,47) node [anchor=north west][inner sep=0.75pt]   [align=left] {c.c.};

\end{tikzpicture}
    \caption{A diagrammatic description of the construction of two maps $f$ and $g$ for the Type C equivariant crossing change.}
    \label{fig:typeC_composition}
\end{figure}

In the forthcoming proof, we construct the maps $f$ and $g$ as follows:
\begin{itemize}
    \item \textbf{(Type A)} The map $f$ is obtained by applying $f_0$ to the left crossing first, followed by the right crossing. The map $g$ is constructed by the reversed process.
    \item \textbf{(Type B)} The maps $f$ and $g$ are are essentially identical to $f_0$ and $g_0$.
    \item \textbf{(Type C)} The map $f$ is obtained by applying the chain homotopy equivalence for the Reidemeister move II and then applying $f_0$ to one of the crossings. The map $g$ is constructed by the reversed process. See \Cref{fig:typeC_composition} for a diagrammatic description.
\end{itemize}
We also provide chain complex diagrams to illustrate the chain maps and homotopies.
The chain maps are indicated by solid red (or blue) arrows, while the homotopies are represented by green dashed arrows.

\subsection{Type A equivariant crossing change}

\begin{figure}
    \centering
    \tikzset{every picture/.style={line width=0.75pt}} %set default line width to 0.75pt        

\begin{tikzpicture}[x=0.75pt,y=0.75pt,yscale=-1,xscale=1]
%uncomment if require: \path (0,460); %set diagram left start at 0, and has height of 460

%Shape: Ellipse [id:dp35122842157551837] 
\draw   (10,75) .. controls (10,55.67) and (25.67,40) .. (45,40) .. controls (64.33,40) and (80,55.67) .. (80,75) .. controls (80,94.33) and (64.33,110) .. (45,110) .. controls (25.67,110) and (10,94.33) .. (10,75) -- cycle ;
%Straight Lines [id:da2805105673093593] 
\draw [line width=2.25]    (20,50) -- (70,100) ;
%Straight Lines [id:da5844198334365965] 
\draw [line width=2.25]    (20,100) -- (40,80) ;
%Straight Lines [id:da02225492516496541] 
\draw [line width=2.25]    (50,70) -- (70,50) ;
%Straight Lines [id:da8533467612745235] 
\draw [color={rgb, 255:red, 155; green, 155; blue, 155 }  ,draw opacity=1 ][line width=0.75]    (90,10) -- (90,140) ;
%Shape: Ellipse [id:dp3324727092737022] 
\draw   (100,75) .. controls (100,55.67) and (115.67,40) .. (135,40) .. controls (154.33,40) and (170,55.67) .. (170,75) .. controls (170,94.33) and (154.33,110) .. (135,110) .. controls (115.67,110) and (100,94.33) .. (100,75) -- cycle ;
%Straight Lines [id:da844580070494403] 
\draw [line width=2.25]    (110,50) -- (160,100) ;
%Straight Lines [id:da7986151284290206] 
\draw [line width=2.25]    (110,100) -- (130,80) ;
%Straight Lines [id:da04612983741990395] 
\draw [line width=2.25]    (140,70) -- (160,50) ;
%Straight Lines [id:da38348951654641183] 
\draw [color={rgb, 255:red, 155; green, 155; blue, 155 }  ,draw opacity=1 ]   (90,190) -- (90,320) ;
%Shape: Ellipse [id:dp538314850920491] 
\draw   (10,255) .. controls (10,235.67) and (25.67,220) .. (45,220) .. controls (64.33,220) and (80,235.67) .. (80,255) .. controls (80,274.33) and (64.33,290) .. (45,290) .. controls (25.67,290) and (10,274.33) .. (10,255) -- cycle ;
%Straight Lines [id:da8185540935979925] 
\draw [line width=2.25]    (20,230) -- (40,250) ;
%Straight Lines [id:da28933692831573676] 
\draw [line width=2.25]    (20,280) -- (70,230) ;
%Straight Lines [id:da9073184076032119] 
\draw [line width=2.25]    (50,260) -- (70,280) ;
%Shape: Ellipse [id:dp8780831907005935] 
\draw   (100,255) .. controls (100,235.67) and (115.67,220) .. (135,220) .. controls (154.33,220) and (170,235.67) .. (170,255) .. controls (170,274.33) and (154.33,290) .. (135,290) .. controls (115.67,290) and (100,274.33) .. (100,255) -- cycle ;
%Straight Lines [id:da18463085572258675] 
\draw [line width=2.25]    (110,230) -- (130,250) ;
%Straight Lines [id:da6366774353461291] 
\draw [line width=2.25]    (110,280) -- (160,230) ;
%Straight Lines [id:da7126155178175672] 
\draw [line width=2.25]    (140,260) -- (160,280) ;
%Shape: Ellipse [id:dp08993937730839285] 
\draw   (200,75) .. controls (200,55.67) and (215.67,40) .. (235,40) .. controls (254.33,40) and (270,55.67) .. (270,75) .. controls (270,94.33) and (254.33,110) .. (235,110) .. controls (215.67,110) and (200,94.33) .. (200,75) -- cycle ;
%Straight Lines [id:da8035497506896851] 
\draw [color={rgb, 255:red, 155; green, 155; blue, 155 }  ,draw opacity=1 ][line width=0.75]    (280,10) -- (280,140) ;
%Shape: Ellipse [id:dp6250024505783105] 
\draw   (290,75) .. controls (290,55.67) and (305.67,40) .. (325,40) .. controls (344.33,40) and (360,55.67) .. (360,75) .. controls (360,94.33) and (344.33,110) .. (325,110) .. controls (305.67,110) and (290,94.33) .. (290,75) -- cycle ;
%Straight Lines [id:da06529974383339798] 
\draw [color={rgb, 255:red, 155; green, 155; blue, 155 }  ,draw opacity=1 ]   (280,190) -- (280,320) ;
%Shape: Ellipse [id:dp551042497165022] 
\draw   (200,255) .. controls (200,235.67) and (215.67,220) .. (235,220) .. controls (254.33,220) and (270,235.67) .. (270,255) .. controls (270,274.33) and (254.33,290) .. (235,290) .. controls (215.67,290) and (200,274.33) .. (200,255) -- cycle ;
%Shape: Ellipse [id:dp26984792358513243] 
\draw   (290,255) .. controls (290,235.67) and (305.67,220) .. (325,220) .. controls (344.33,220) and (360,235.67) .. (360,255) .. controls (360,274.33) and (344.33,290) .. (325,290) .. controls (305.67,290) and (290,274.33) .. (290,255) -- cycle ;
%Shape: Ellipse [id:dp6041747856283227] 
\draw   (390,75) .. controls (390,55.67) and (405.67,40) .. (425,40) .. controls (444.33,40) and (460,55.67) .. (460,75) .. controls (460,94.33) and (444.33,110) .. (425,110) .. controls (405.67,110) and (390,94.33) .. (390,75) -- cycle ;
%Straight Lines [id:da08122325521140905] 
\draw [color={rgb, 255:red, 155; green, 155; blue, 155 }  ,draw opacity=1 ][line width=0.75]    (470,10) -- (470,140) ;
%Shape: Ellipse [id:dp1930383843367256] 
\draw   (480,75) .. controls (480,55.67) and (495.67,40) .. (515,40) .. controls (534.33,40) and (550,55.67) .. (550,75) .. controls (550,94.33) and (534.33,110) .. (515,110) .. controls (495.67,110) and (480,94.33) .. (480,75) -- cycle ;
%Straight Lines [id:da8676110934525099] 
\draw [color={rgb, 255:red, 155; green, 155; blue, 155 }  ,draw opacity=1 ]   (470,190) -- (470,320) ;
%Shape: Ellipse [id:dp8363443597083274] 
\draw   (390,255) .. controls (390,235.67) and (405.67,220) .. (425,220) .. controls (444.33,220) and (460,235.67) .. (460,255) .. controls (460,274.33) and (444.33,290) .. (425,290) .. controls (405.67,290) and (390,274.33) .. (390,255) -- cycle ;
%Shape: Ellipse [id:dp39420519935571985] 
\draw   (480,255) .. controls (480,235.67) and (495.67,220) .. (515,220) .. controls (534.33,220) and (550,235.67) .. (550,255) .. controls (550,274.33) and (534.33,290) .. (515,290) .. controls (495.67,290) and (480,274.33) .. (480,255) -- cycle ;
%Curve Lines [id:da15024055626554955] 
\draw [line width=2.25]    (260,50) .. controls (240,69.8) and (230,69.8) .. (210,50) ;
%Curve Lines [id:da5354510975319194] 
\draw [line width=2.25]    (260,100) .. controls (240,79.8) and (230,79.8) .. (210,100) ;
%Curve Lines [id:da9078932749991652] 
\draw [line width=2.25]    (350,50) .. controls (330,69.8) and (320,69.8) .. (300,50) ;
%Curve Lines [id:da28886054153454366] 
\draw [line width=2.25]    (350,100) .. controls (330,79.8) and (320,79.8) .. (300,100) ;
%Curve Lines [id:da6052982265279728] 
\draw [line width=2.25]    (540,50) .. controls (520,69.8) and (510,69.8) .. (490,50) ;
%Curve Lines [id:da4731285775088462] 
\draw [line width=2.25]    (540,100) .. controls (520,79.8) and (510,79.8) .. (490,100) ;
%Curve Lines [id:da5144275951503915] 
\draw [line width=2.25]    (260,230) .. controls (240,249.8) and (230,249.8) .. (210,230) ;
%Curve Lines [id:da9303956675058312] 
\draw [line width=2.25]    (260,280) .. controls (240,259.8) and (230,259.8) .. (210,280) ;
%Curve Lines [id:da21349353912307634] 
\draw [line width=2.25]    (400,100) .. controls (420,80) and (420,69.8) .. (400,50) ;
%Curve Lines [id:da22231038665583036] 
\draw [line width=2.25]    (450,100) .. controls (430,79.8) and (430,70) .. (450,50) ;
%Curve Lines [id:da6512969368219608] 
\draw [line width=2.25]    (300,280) .. controls (320,260) and (320,249.8) .. (300,230) ;
%Curve Lines [id:da7490388771233268] 
\draw [line width=2.25]    (350,280) .. controls (330,259.8) and (330,250) .. (350,230) ;
%Curve Lines [id:da7709344255916651] 
\draw [line width=2.25]    (400,280) .. controls (420,260) and (420,249.8) .. (400,230) ;
%Curve Lines [id:da698699468310522] 
\draw [line width=2.25]    (450,280) .. controls (430,259.8) and (430,250) .. (450,230) ;
%Curve Lines [id:da5132722277014677] 
\draw [line width=2.25]    (490,280) .. controls (510,260) and (510,249.8) .. (490,230) ;
%Curve Lines [id:da9299698552649537] 
\draw [line width=2.25]    (540,280) .. controls (520,259.8) and (520,250) .. (540,230) ;

% Text Node
\draw (81,142.4) node [anchor=north west][inner sep=0.75pt]    {$K$};
% Text Node
\draw (81,322.4) node [anchor=north west][inner sep=0.75pt]    {$K'$};
% Text Node
\draw (271,142.4) node [anchor=north west][inner sep=0.75pt]    {$K_{00}$};
% Text Node
\draw (271,322.4) node [anchor=north west][inner sep=0.75pt]    {$K_{01}$};
% Text Node
\draw (461,142.4) node [anchor=north west][inner sep=0.75pt]    {$K_{10}$};
% Text Node
\draw (461,322.4) node [anchor=north west][inner sep=0.75pt]    {$K_{11}$};

\end{tikzpicture}
    \caption{Local pictures of diagrams $K$, $K'$, and their resolutions. Each vertical line denotes the axis.}
    \label{fig:typeA_diag}
\end{figure}
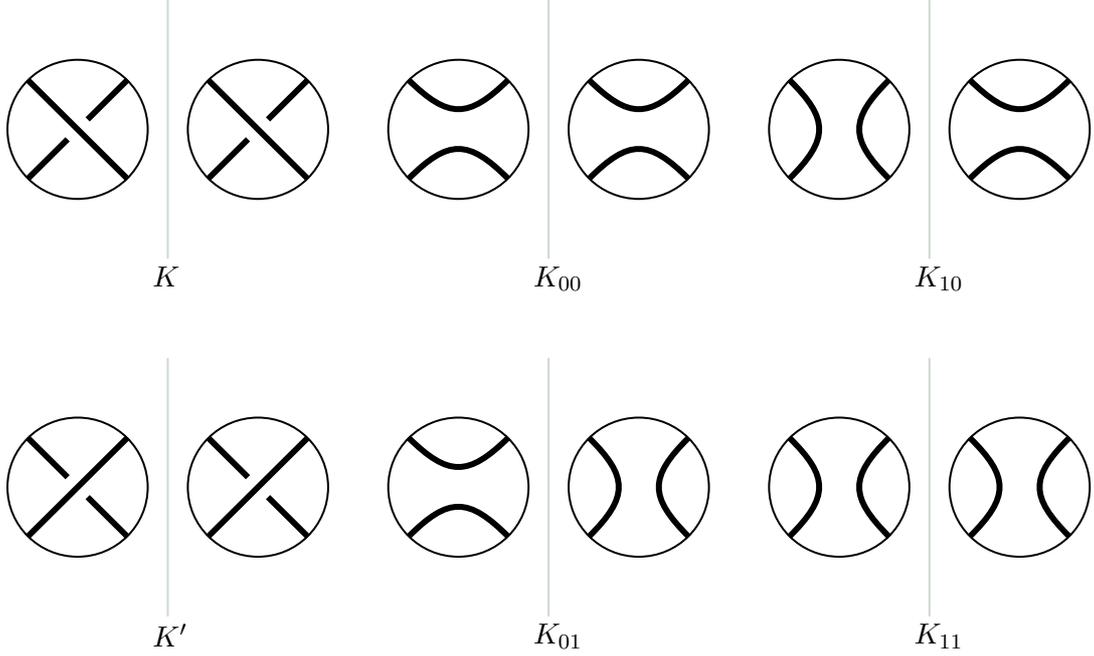

Suppose we have a transvergent diagram $K$, and let $c$ be an off-axis crossing.
Let $K'$ be the diagram obtained from $K$ by changing the crossings at both $c$ and $\tau c$.
Since there are two crossings, there are four possible resolution choices, as shown in \Cref{fig:typeA_diag}.
The formal Khovanov complexes $\bn{K}$ and $\bn{K'}$ can be expressed as follows:
\[
\begin{tikzcd}
          &  &                                                   & \bn{K_{10}}\arrow[dr,"S_r"]   &               \\
    \bn{K}&= & \bn{K_{00}} \arrow[ur, "S_l"]\arrow[dr, "S_r"]    &                               & \bn{K_{11}}   \\
          &  &                                                   & \bn{K_{01}}\arrow[ur,"S_l"]   &
\end{tikzcd},\]
and
\[
\begin{tikzcd}
          &  &                                                   & \bn{K_{01}}\arrow[dr,"S_r"]   &               \\
    \bn{K'}&= & \bn{K_{11}} \arrow[ur, "S_l"]\arrow[dr, "S_r"]    &                               & \bn{K_{00}}   \\
          &  &                                                   & \bn{K_{10}}\arrow[ur,"S_l"]   &
\end{tikzcd}.\]
Note that the involution $I_\tau$ maps $\bn{K_{ij}} \to \bn{K_{ji}}$ for $i,j \in \{0,1\}$.
We define two chain maps $\bn{K} \xrightleftharpoons[g]{f} \bn{K'}$ by
\[
f=
\begin{bmatrix}
    \gz&\gz&\gz&D_lD_r\\
    \gz&\gz&D_r&\gz\\
    \gz&D_l&\gz&\gz\\
    \id&\gz&\gz&\gz
\end{bmatrix}, \qquad
g=
\begin{bmatrix}
    \gz&\gz&\gz&\id\\
    \gz&\gz&D_l&\gz\\
    \gz&D_r&\gz&\gz\\
    D_lD_r&\gz&\gz&\gz
\end{bmatrix},
\]
where we decompose $\bn{K}$ as $\bn{K_{00}}\oplus\bn{K_{10}}\oplus\bn{K_{01}}\oplus\bn{K_{11}}$ and $\bn{K'}$ as $\bn{K_{11}}\oplus\bn{K_{01}}\oplus\bn{K_{10}}\oplus\bn{K_{00}}$.
The subscript $l$ (resp. $r$) indicates the application of the map $D$ to the left (resp. right) region.
It is straightforward to verify that $f$ and $g$ commute with $I_\tau$.
These maps are illustrated in \Cref{fig:typeA_map}.

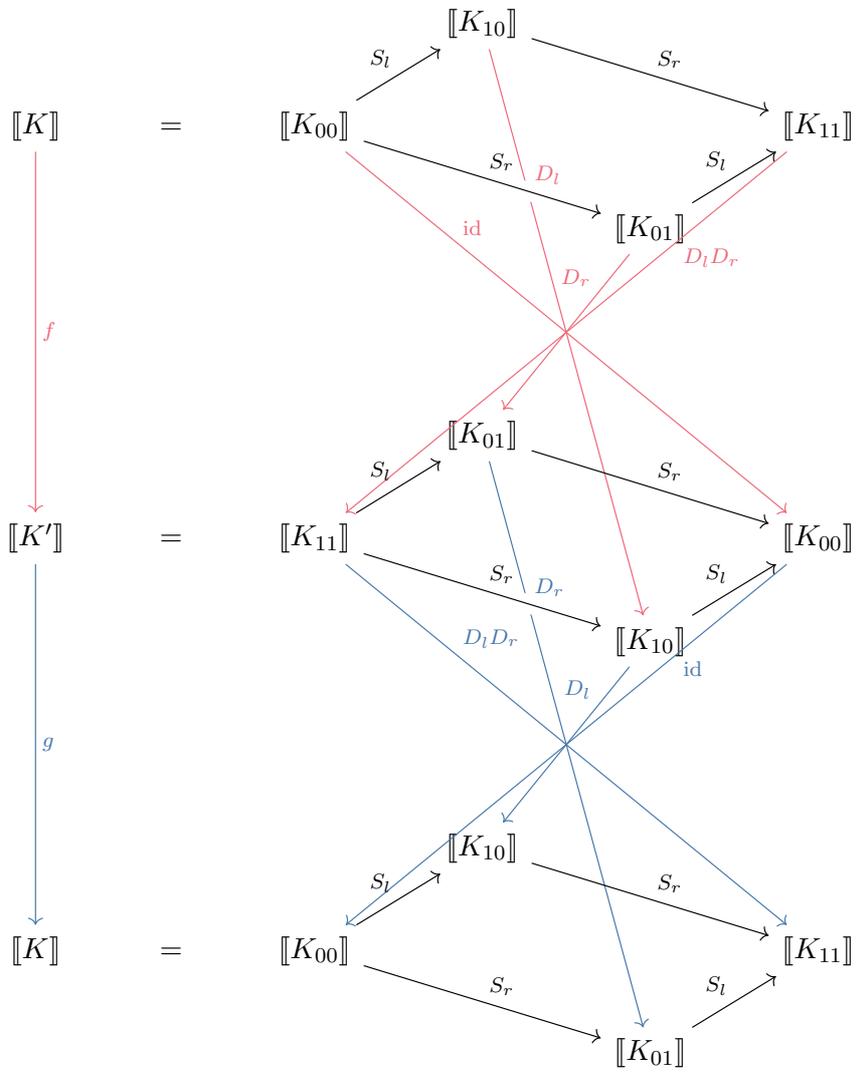
\begin{figure}
    \centering
    \begin{tikzcd}
    &           &&\bn{K_{10}}\arrow[rrd,"S_r"]\arrow[rddddddd,"D_l",red,near start]&&\\
    \bn{K}\arrow[ddddd,"f",red]&=    &\bn{K_{00}}\arrow[rrrddddd,"\id",red,near start]\arrow[ru,"S_l"]\arrow[rrd,"S_r",crossing over]&&&\bn{K_{11}}\arrow[lllddddd,"D_lD_r",red,near start]\\
    &           &&&\bn{K_{01}}\arrow[lddd,"D_r"',red,near start]\arrow[ru,"S_l"]&\\
    &&&&&\\
    &&&&&\\
    &           &&\bn{K_{01}}\arrow[rrd,"S_r"]\arrow[rddddddd,"D_r",blue,near start]&&\\
    \bn{K'}\arrow[ddddd,"g",blue]&=    &\bn{K_{11}}\arrow[rrrddddd,"D_lD_r",blue,near start]\arrow[ru,"S_l"]\arrow[rrd,"S_r",crossing over]&&&\bn{K_{00}}\arrow[lllddddd,"\id",blue,near start]\\
    &           &&&\bn{K_{10}}\arrow[lddd,"D_l"',blue,near start]\arrow[ru,"S_l"]&\\
    &&&&&\\
    &&&&&\\
    &           &&\bn{K_{10}}\arrow[rrd,"S_r"]&&\\
    \bn{K}&=    &\bn{K_{00}}\arrow[ru,"S_l"]\arrow[rrd,"S_r"]&&&\bn{K_{11}}\\
    &           &&&\bn{K_{01}}\arrow[ru,"S_l"]&
\end{tikzcd}
    \caption{(Type A) Two chain maps $\bn{K} \xrightleftharpoons[g]{f} \bn{K'}$.}
    \label{fig:typeA_map}
\end{figure}

The composition $gf$ is the map $D_l D_r$.
We define a homotopy $h$ as:
\[
h=
\begin{bmatrix}
    \gz&H\cdot S_l&S_rD_l&\gz\\
    \gz&\gz&\gz&S_rD_l\\
    \gz&\gz&\gz&H\cdot S_l\\
    \gz&\gz&\gz&\gz
\end{bmatrix}.
\]
Refer \Cref{fig:typeA_h} to see a map $gf+H\cdot\id$ and a homotopy $h$ between $gf+H\cdot\id$ and the zero map.
Using the relations $S_l^2 = D_l + H\cdot \id$ and $S_r^2 = D_r + H\cdot \id$, one can directly verify that $H^2\cdot \id + gf = dh + hd$.

\begin{figure}
    \centering
    \resizebox{.95\textwidth}{!}{\begin{tikzcd}
    &           &&&\bn{K_{10}}&&\\
    \bn{K}&=    &\bn{K_{00}}&&&&&&\bn{K_{11}}\\
    &           &&&&&\bn{K_{01}}&\\
    &&&&&\\
    &&&&&\\
    &           &&&\bn{K_{10}}&&\\
    \bn{K}&=    &\bn{K_{00}}&&&&&&\bn{K_{11}}\\
    &           &&&&&\bn{K_{01}}&
    \arrow[from=2-1,to=7-1,"gf+H^2\cdot\id"',red]
    \arrow[from=2-3,to=7-3,"D_lD_r+H^2\cdot\id"',red]\arrow[from=1-5,to=6-5,"D_lD_r+H^2\cdot\id",red]\arrow[from=2-9,to=7-9,"D_lD_r+H^2\cdot\id",red]\arrow[from=1-5,to=7-3,"H\cdot S_l" description,dashed,green]\arrow[from=2-9,to=6-5,"S_rD_l" description,dashed,green,near end]
    \arrow[from=7-3,to=6-5,"S_l"']\arrow[from=6-5,to=7-9,"S_r"]
    \arrow[from=2-3,to=1-5,"S_l"]\arrow[from=1-5,to=2-9,"S_r"]
    \arrow[from=2-9,to=8-7,"H\cdot S_l" description,dashed,green,crossing over]
    \arrow[from=2-3,to=3-7,"S_r",crossing over]\arrow[from=3-7,to=2-9,"S_l",crossing over]\arrow[from=7-3,to=8-7,"S_r"]\arrow[from=8-7,to=7-9,"S_l"]
    \arrow[from=3-7,to=8-7,"D_lD_r+H^2\cdot\id",red,crossing over]
    \arrow[from=3-7,to=7-3,"S_rD_l" description,dashed,green,crossing over]
\end{tikzcd}}
    \caption{(Type A) Red arrow depicts the map $gf+H^2\cdot\id$, and green dashed arrow depicts the homotopy $h$.}
    \label{fig:typeA_h}
\end{figure}
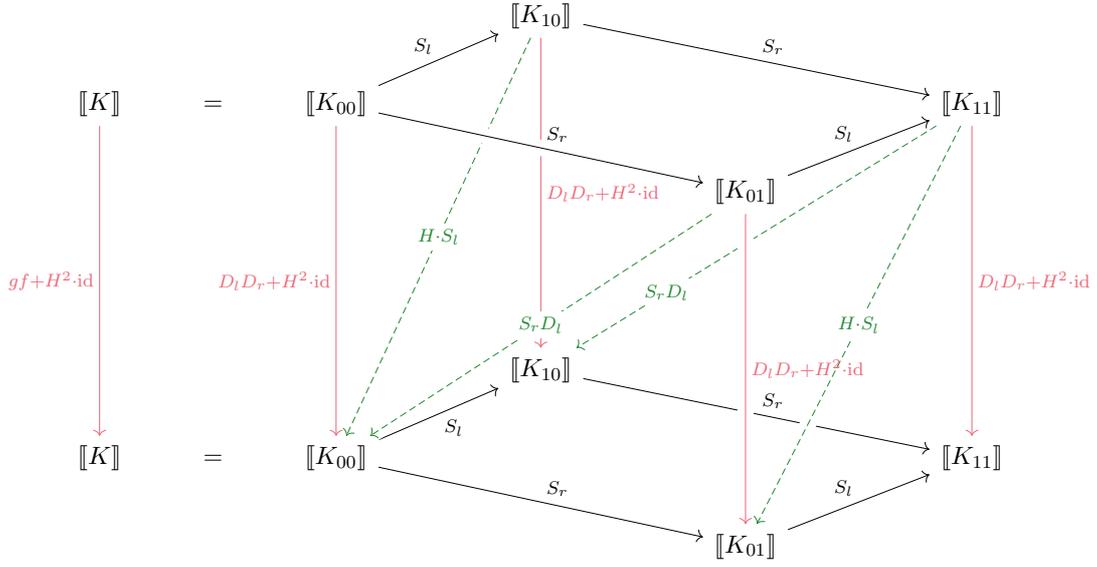
\begin{figure}
    \centering
    \resizebox{.95\textwidth}{!}{\begin{tikzcd}
    &       &&&&\bn{K_{10}}&&\\
    \bn{K}&=&\bn{K_{00}}&&&&&&\bn{K_{11}}\\
    &       &&&&\bn{K_{01}}&&\\
    &&&&&\\
    &&&&&\\
    &       &&&&\bn{K_{10}}&&\\
    \bn{K}&=&\bn{K_{00}}&&&&&&\bn{K_{11}}\\
    &       &&&&\bn{K_{01}}&&
    \arrow[from=7-3,to=6-6,"S_l"']\arrow[from=6-6,to=7-9,"S_r"]\arrow[from=7-3,to=8-6,"S_r"]\arrow[from=8-6,to=7-9,"S_l"]
    \arrow[from=2-9,to=7-3,"S_lS_rI_\tau" description,green,dashed]
    \arrow[from=1-6,to=7-3,"I_\tau (H\cdot S_l) + S_r D_l I_\tau"',red]
    \arrow[from=2-3,to=1-6,"S_l"]\arrow[from=1-6,to=2-9,"S_r"]
    \arrow[from=2-3,to=3-6,"S_r",crossing over]\arrow[from=3-6,to=2-9,"S_l",crossing over]
    \arrow[from=3-6,to=7-3,"I_\tau S_rD_l+(H\cdot S_l)I_\tau"',red,crossing over]\arrow[from=2-9,to=8-6,"I_\tau (H\cdot S_l) + S_r D_l I_\tau",red,crossing over]
    \arrow[from=2-9,to=6-6,"I_\tau S_rD_l+(H\cdot S_l)I_\tau",red]
    \arrow[from=2-1,to=7-1,"I_\tau h+hI_\tau",red]
\end{tikzcd}}
    \caption{(Type A) Red arrow depicts the map $I_\tau h+hI_\tau$, and green dashed arrow depicts the homotopy $k$.}
    \label{fig:typeA_k}
\end{figure}
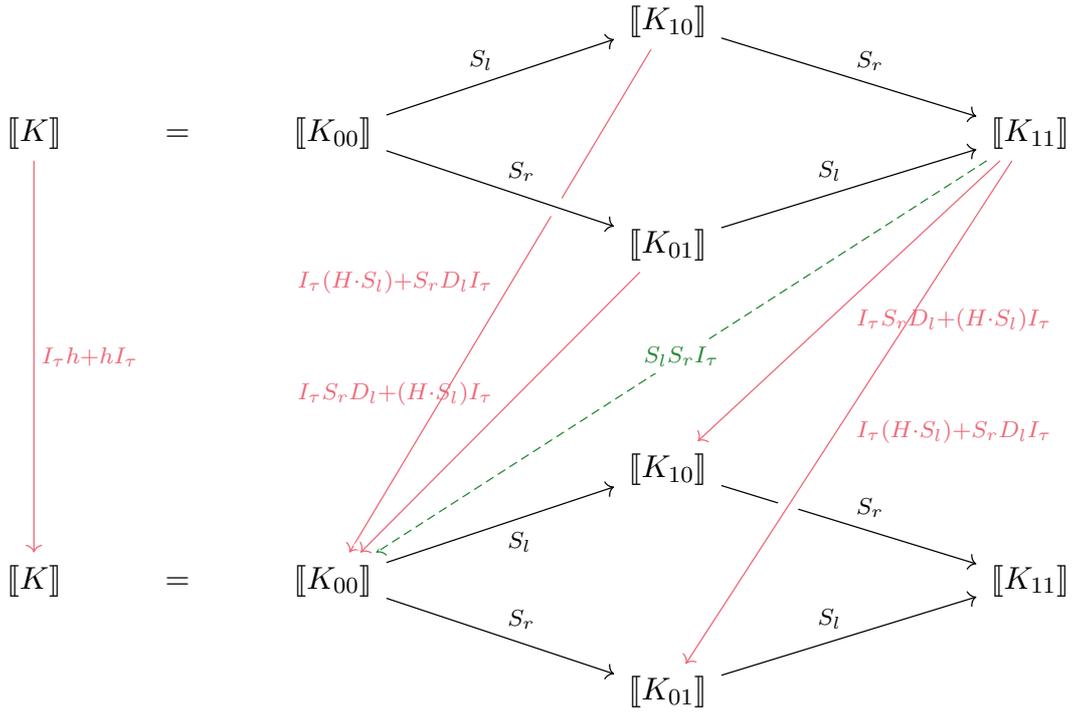

Finally, we must find a homotopy $k$ between $I_\tau h + h I_\tau$ and $0$; i.e., $I_\tau h + h I_\tau = dk + kd$.
We define $k$ by:
\[
k=
\begin{bmatrix}
    \gz&\gz&\gz&S_lS_rI_\tau\\
    \gz&\gz&\gz&\gz\\
    \gz&\gz&\gz&\gz\\
    \gz&\gz&\gz&\gz
\end{bmatrix}
\]
represented by the green dashed arrows in \Cref{fig:typeA_k}.
The equation $I_\tau h + h I_\tau = dk + kd$ can be expressed in matrix form as:
\begin{align*}
I_\tau h+h I_\tau &=
\begin{bmatrix}
    \gz&I_\tau (H\cdot S_l) + S_rD_lI_\tau&I_\tau S_rD_l+(H\cdot S_l)I_\tau&\gz\\
    \gz&\gz&\gz&I_\tau S_rD_l+(H\cdot S_l)I_\tau\\
    \gz&\gz&\gz&I_\tau (H\cdot S_l) + S_rD_lI_\tau\\
    \gz&\gz&\gz&\gz
\end{bmatrix}\\
&=
\begin{bmatrix}
    \gz&S_lS_rI_\tau S_r&S_lS_rI_\tau S_l&\gz\\
    \gz&\gz&\gz&S_lS_lS_rI_\tau\\
    \gz&\gz&\gz&S_rS_lS_rI_\tau\\
    \gz&\gz&\gz&\gz
\end{bmatrix} = dk+kd.
\end{align*}
The above equalities hold because $I_\tau$ commutes with $S$ and $D$ while swapping the subscripts $l \leftrightarrow r$.
For example, the map $\bn{K_{10}} \xrightarrow{I_\tau(H\cdot S_l) + S_r D_l I_\tau} \bn{K_{00}}$ is verified by:
\[ (S_l S_r I_\tau) S_r = S_l S_r^2 I_\tau = S_l (D_r + H\cdot \id) I_\tau = S_l D_r I_\tau + H S_l I_\tau = S_r D_l I_\tau + I_\tau (H S_l). \]

The verification for the other direction ($fg \simeq_{h'} H^2\cdot \id$ and $h' I_\tau + I_\tau h' \simeq_{k'} 0$) follows from an analogous argument.

\subsection{Type B equivariant crossing change}

Let $K$ be a transvergent diagram, and let $c$ be an on-axis crossing.
Let $K'$ be the transvergent diagram obtained by changing the crossing at $c$. As before, we can write:
\[ \bn{K} = \Cone\left(\bn{K_0} \xrightarrow{S} \bn{K_1}\right), \qquad \bn{K'} = \Cone\left(\bn{K_1} \xrightarrow{S} \bn{K_0}\right) \]
where $K_0$ and $K_1$ are the transvergent diagrams obtained by resolving $K$ at $c$; see \Cref{fig:typeB_diag}.
\begin{figure}
    \centering
    \tikzset{every picture/.style={line width=0.75pt}} %set default line width to 0.75pt        

\begin{tikzpicture}[x=0.75pt,y=0.75pt,yscale=-1,xscale=1]
%uncomment if require: \path (0,198); %set diagram left start at 0, and has height of 198

%Straight Lines [id:da47843963845601367] 
\draw [color={rgb, 255:red, 155; green, 155; blue, 155 }  ,draw opacity=1 ]   (45,10) -- (45,120) ;
%Straight Lines [id:da13701671102320845] 
\draw [color={rgb, 255:red, 155; green, 155; blue, 155 }  ,draw opacity=1 ]   (125,10) -- (125,120) ;
%Straight Lines [id:da01701986170799097] 
\draw [color={rgb, 255:red, 155; green, 155; blue, 155 }  ,draw opacity=1 ]   (205,10) -- (205,120) ;
%Straight Lines [id:da09987146495210453] 
\draw [color={rgb, 255:red, 155; green, 155; blue, 155 }  ,draw opacity=1 ]   (285,10) -- (285,120) ;
%Shape: Ellipse [id:dp9793553636127664] 
\draw   (170,65) .. controls (170,45.67) and (185.67,30) .. (205,30) .. controls (224.33,30) and (240,45.67) .. (240,65) .. controls (240,84.33) and (224.33,100) .. (205,100) .. controls (185.67,100) and (170,84.33) .. (170,65) -- cycle ;
%Curve Lines [id:da9743727486002948] 
\draw [line width=2.25]    (230,40) .. controls (210,59.8) and (200,59.8) .. (180,40) ;
%Curve Lines [id:da22900560191086872] 
\draw [line width=2.25]    (230,90) .. controls (210,69.8) and (200,69.8) .. (180,90) ;
%Shape: Ellipse [id:dp9600982917927745] 
\draw   (10,65) .. controls (10,45.67) and (25.67,30) .. (45,30) .. controls (64.33,30) and (80,45.67) .. (80,65) .. controls (80,84.33) and (64.33,100) .. (45,100) .. controls (25.67,100) and (10,84.33) .. (10,65) -- cycle ;
%Shape: Ellipse [id:dp94955578283303] 
\draw   (250,65) .. controls (250,45.67) and (265.67,30) .. (285,30) .. controls (304.33,30) and (320,45.67) .. (320,65) .. controls (320,84.33) and (304.33,100) .. (285,100) .. controls (265.67,100) and (250,84.33) .. (250,65) -- cycle ;
%Curve Lines [id:da193590666779053] 
\draw [line width=2.25]    (260,90) .. controls (280,70) and (280,59.8) .. (260,40) ;
%Curve Lines [id:da36751170295679303] 
\draw [line width=2.25]    (310,90) .. controls (290,69.8) and (290,60) .. (310,40) ;
%Straight Lines [id:da8613277319343091] 
\draw [line width=2.25]    (20,40) -- (70,90) ;
%Straight Lines [id:da9102148317785004] 
\draw [line width=2.25]    (20,90) -- (40,70) ;
%Straight Lines [id:da3535544201708938] 
\draw [line width=2.25]    (50,60) -- (70,40) ;
%Shape: Ellipse [id:dp1259124966941909] 
\draw   (90,65) .. controls (90,45.67) and (105.67,30) .. (125,30) .. controls (144.33,30) and (160,45.67) .. (160,65) .. controls (160,84.33) and (144.33,100) .. (125,100) .. controls (105.67,100) and (90,84.33) .. (90,65) -- cycle ;
%Straight Lines [id:da9408244538947608] 
\draw [line width=2.25]    (100,40) -- (120,60) ;
%Straight Lines [id:da4111778408415109] 
\draw [line width=2.25]    (100,90) -- (150,40) ;
%Straight Lines [id:da6251677542779728] 
\draw [line width=2.25]    (130,70) -- (150,90) ;

% Text Node
\draw (41,125.4) node [anchor=north west][inner sep=0.75pt]    {$K$};
% Text Node
\draw (119,125.4) node [anchor=north west][inner sep=0.75pt]    {$K'$};
% Text Node
\draw (199,125.4) node [anchor=north west][inner sep=0.75pt]    {$K_{0}$};
% Text Node
\draw (277,125.4) node [anchor=north west][inner sep=0.75pt]    {$K_{1}$};

\end{tikzpicture}
    \caption{Local pictures of diagrams $K$, $K'$ and its resolution. Each vertical line denotes the axis.}
    \label{fig:typeB_diag}
\end{figure}
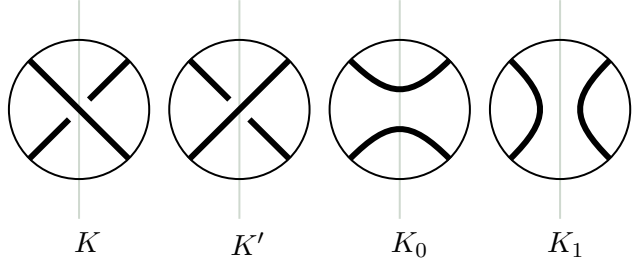
As mentioned earlier, the two maps $\bn{K} \xrightleftharpoons[g]{f} \bn{K'}$ and homotopies $h, h'$ are identical to Alishahi's maps $\bn{K} \xrightleftharpoons[g_0]{f_0} \bn{K'}$ and homotopies $h_0, h_0'$.
It is straightforward to verify that $f$ and $g$ are strictly $I_\tau$-equivariant.
The remaining task is to find homotopies $k$ and $k'$ such that $I_\tau h+ hI_\tau \simeq_k 0$ and $I_\tau h'+ h'I_\tau \simeq_{k'} 0$.

Since the homotopy $h$ consists solely of the map $S \colon \bn{K_1} \to \bn{K_0}$, and this map $S$ commutes with $I_\tau$, the term $I_\tau h + h I_\tau$ is simply the zero map (recall that we are working over $\FF$).
Consequently, we can take $k=0$, and similarly $k'=0$.

\subsection{Type C equivariant crossing change}\label{subsec:typeCproof}

Let $K'$ be a transvergent diagram with two nearby fixed points, as shown in \Cref{fig:typeC_diag}.
Applying a Type C equivariant crossing change yields the diagram $K$.
For $K$, there are four possible resolutions: $K_{00}, K_{10}, K_{01}$, and $K_{11}$.
Notably, $K_{00}$ is isotopic to $K'$.
We can express the complexes $\bn{K}$ and $\bn{K'}$ as:
\[
\begin{tikzcd}
    &&&\bn{K_{10}}\arrow[rd,"S_r"]&\\
    \bn{K}&=&\bn{K_{00}}\arrow[ru,"S_l"]\arrow[rd,"S_r"]&&\bn{K_{11}}\\
    &&&\bn{K_{01}}\arrow[ru,"S_l"]&
\end{tikzcd}, \qquad
\begin{tikzcd}
    \bn{K'}&=&\bn{K_{00}},
\end{tikzcd}
\]
where the involution $I_\tau$ maps $\bn{K_{ij}} \to \bn{K_{ji}}$ for $i,j \in \{0,1\}$.
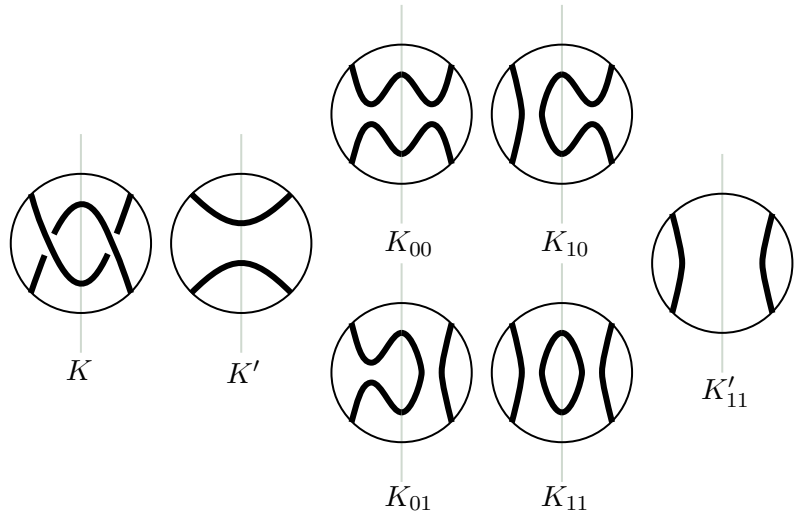
\begin{figure}
    \centering
    \tikzset{every picture/.style={line width=0.75pt}} %set default line width to 0.75pt        

\begin{tikzpicture}[x=0.75pt,y=0.75pt,yscale=-1,xscale=1]
%uncomment if require: \path (0,303); %set diagram left start at 0, and has height of 303

%Straight Lines [id:da8687359067952001] 
\draw [color={rgb, 255:red, 155; green, 155; blue, 155 }  ,draw opacity=1 ]   (125,85) -- (125,195) ;
%Straight Lines [id:da5981753852726095] 
\draw [color={rgb, 255:red, 155; green, 155; blue, 155 }  ,draw opacity=1 ]   (45,85) -- (45,195) ;
%Shape: Ellipse [id:dp5054947160602796] 
\draw   (90,140) .. controls (90,120.67) and (105.67,105) .. (125,105) .. controls (144.33,105) and (160,120.67) .. (160,140) .. controls (160,159.33) and (144.33,175) .. (125,175) .. controls (105.67,175) and (90,159.33) .. (90,140) -- cycle ;
%Curve Lines [id:da13062803769675124] 
\draw [line width=2.25]    (150,115) .. controls (130,134.8) and (120,134.8) .. (100,115) ;
%Curve Lines [id:da06494312139707403] 
\draw [line width=2.25]    (150,165) .. controls (130,144.8) and (120,144.8) .. (100,165) ;
%Shape: Ellipse [id:dp5485025936734615] 
\draw   (10,140) .. controls (10,120.67) and (25.67,105) .. (45,105) .. controls (64.33,105) and (80,120.67) .. (80,140) .. controls (80,159.33) and (64.33,175) .. (45,175) .. controls (25.67,175) and (10,159.33) .. (10,140) -- cycle ;
%Curve Lines [id:da6982460516443214] 
\draw [line width=2.25]    (70,115) .. controls (67.45,122.69) and (65.06,129.4) .. (62.8,135.13) ;
%Curve Lines [id:da885954244249777] 
\draw [line width=2.25]    (70,165) .. controls (68.22,159.65) and (66.52,154.78) .. (64.89,150.39) .. controls (63.25,146) and (67.19,155.47) .. (58.65,135.36) .. controls (50.1,115.25) and (41.26,115.18) .. (31.51,135.34) ;
%Curve Lines [id:da9503666922804973] 
\draw [line width=2.25]    (58.43,145.29) .. controls (49.6,165) and (41.26,165.4) .. (31.68,145.46) .. controls (22.1,125.52) and (21.54,118.9) .. (20,115) ;
%Curve Lines [id:da6409697186031835] 
\draw [line width=2.25]    (26.9,145.97) .. controls (25.3,149.98) and (22.44,157.78) .. (20,165) ;
%Straight Lines [id:da8211273897372715] 
\draw [color={rgb, 255:red, 155; green, 155; blue, 155 }  ,draw opacity=1 ]   (205,20) -- (205,130) ;
%Shape: Ellipse [id:dp2383406852731259] 
\draw   (170,75) .. controls (170,55.67) and (185.67,40) .. (205,40) .. controls (224.33,40) and (240,55.67) .. (240,75) .. controls (240,94.33) and (224.33,110) .. (205,110) .. controls (185.67,110) and (170,94.33) .. (170,75) -- cycle ;
%Curve Lines [id:da5533684951360266] 
\draw [line width=2.25]    (205,55) .. controls (199.8,55.05) and (195.2,69.95) .. (190,70) .. controls (184.8,70.05) and (182.7,60.4) .. (180,50) ;
%Curve Lines [id:da6464908221576002] 
\draw [line width=2.25]    (230,100) .. controls (227.65,90.4) and (224.95,79.95) .. (220,80) .. controls (215.05,80.05) and (210.2,94.95) .. (205,95) ;
%Curve Lines [id:da8313893698591005] 
\draw [line width=2.25]    (230,50) .. controls (227.7,60.45) and (224.95,69.95) .. (220,70) .. controls (215.05,70.05) and (210.2,54.95) .. (205,55) ;
%Curve Lines [id:da8067718055989407] 
\draw [line width=2.25]    (205,95) .. controls (199.8,95.05) and (195.2,79.95) .. (190,80) .. controls (184.8,80.05) and (182.65,90.4) .. (180,100) ;
%Straight Lines [id:da47988772789013145] 
\draw [color={rgb, 255:red, 155; green, 155; blue, 155 }  ,draw opacity=1 ]   (285,20) -- (285,130) ;
%Shape: Ellipse [id:dp6856566251658408] 
\draw   (250,75) .. controls (250,55.67) and (265.67,40) .. (285,40) .. controls (304.33,40) and (320,55.67) .. (320,75) .. controls (320,94.33) and (304.33,110) .. (285,110) .. controls (265.67,110) and (250,94.33) .. (250,75) -- cycle ;
%Curve Lines [id:da8521860686281323] 
\draw [line width=2.25]    (310,100) .. controls (307.65,90.4) and (304.95,79.95) .. (300,80) .. controls (295.05,80.05) and (290.2,94.95) .. (285,95) ;
%Curve Lines [id:da5499824267499418] 
\draw [line width=2.25]    (310,50) .. controls (307.7,60.45) and (304.95,69.95) .. (300,70) .. controls (295.05,70.05) and (290.2,54.95) .. (285,55) ;
%Curve Lines [id:da8331702094955906] 
\draw [line width=2.25]    (285,55) .. controls (279.8,55.05) and (275.05,70.35) .. (275,75) .. controls (274.95,79.65) and (280.05,95.1) .. (285,95) ;
%Curve Lines [id:da3316216337702944] 
\draw [line width=2.25]    (260,100) .. controls (262.55,90.35) and (264.95,79.65) .. (265,75) .. controls (265.05,70.35) and (262.7,60.4) .. (260,50) ;
%Straight Lines [id:da446886705986996] 
\draw [color={rgb, 255:red, 155; green, 155; blue, 155 }  ,draw opacity=1 ]   (205,150) -- (205,260) ;
%Shape: Ellipse [id:dp32247919626156496] 
\draw   (170,205) .. controls (170,185.67) and (185.67,170) .. (205,170) .. controls (224.33,170) and (240,185.67) .. (240,205) .. controls (240,224.33) and (224.33,240) .. (205,240) .. controls (185.67,240) and (170,224.33) .. (170,205) -- cycle ;
%Curve Lines [id:da38066449835968563] 
\draw [line width=2.25]    (205,185) .. controls (199.8,185.05) and (195.2,199.95) .. (190,200) .. controls (184.8,200.05) and (182.7,190.4) .. (180,180) ;
%Curve Lines [id:da8950721650180102] 
\draw [line width=2.25]    (205,225) .. controls (199.8,225.05) and (195.2,209.95) .. (190,210) .. controls (184.8,210.05) and (182.65,220.4) .. (180,230) ;
%Straight Lines [id:da4658080017255961] 
\draw [color={rgb, 255:red, 155; green, 155; blue, 155 }  ,draw opacity=1 ]   (285,150) -- (285,260) ;
%Shape: Ellipse [id:dp08326703649975109] 
\draw   (250,205) .. controls (250,185.67) and (265.67,170) .. (285,170) .. controls (304.33,170) and (320,185.67) .. (320,205) .. controls (320,224.33) and (304.33,240) .. (285,240) .. controls (265.67,240) and (250,224.33) .. (250,205) -- cycle ;
%Curve Lines [id:da4726730788979401] 
\draw [line width=2.25]    (285,185) .. controls (279.8,185.05) and (275.05,200.35) .. (275,205) .. controls (274.95,209.65) and (280.05,225.1) .. (285,225) ;
%Curve Lines [id:da24122135044258886] 
\draw [line width=2.25]    (260,230) .. controls (262.55,220.35) and (264.95,209.65) .. (265,205) .. controls (265.05,200.35) and (262.7,190.4) .. (260,180) ;
%Curve Lines [id:da5060359915413397] 
\draw [line width=2.25]    (230,230) .. controls (227.65,220.4) and (225.05,210.1) .. (225,205) .. controls (224.95,199.9) and (227.3,190.6) .. (230,180) ;
%Curve Lines [id:da9063706172771963] 
\draw [line width=2.25]    (205,225) .. controls (210.05,225.1) and (214.95,209.4) .. (215,205) .. controls (215.05,200.6) and (210.2,184.95) .. (205,185) ;
%Curve Lines [id:da9868530170240808] 
\draw [line width=2.25]    (310,230) .. controls (307.65,220.4) and (305.05,210.1) .. (305,205) .. controls (304.95,199.9) and (307.3,190.6) .. (310,180) ;
%Curve Lines [id:da42113380828302116] 
\draw [line width=2.25]    (285,225) .. controls (290.05,225.1) and (294.95,209.4) .. (295,205) .. controls (295.05,200.6) and (290.2,184.95) .. (285,185) ;
%Straight Lines [id:da8513929661859924] 
\draw [color={rgb, 255:red, 155; green, 155; blue, 155 }  ,draw opacity=1 ]   (365,95) -- (365,205) ;
%Shape: Ellipse [id:dp2629716342232966] 
\draw   (330,150) .. controls (330,130.67) and (345.67,115) .. (365,115) .. controls (384.33,115) and (400,130.67) .. (400,150) .. controls (400,169.33) and (384.33,185) .. (365,185) .. controls (345.67,185) and (330,169.33) .. (330,150) -- cycle ;
%Curve Lines [id:da08731501185486401] 
\draw [line width=2.25]    (340,175) .. controls (342.55,165.35) and (344.95,154.65) .. (345,150) .. controls (345.05,145.35) and (342.7,135.4) .. (340,125) ;
%Curve Lines [id:da2939728946604908] 
\draw [line width=2.25]    (390,175) .. controls (387.65,165.4) and (385.05,155.1) .. (385,150) .. controls (384.95,144.9) and (387.3,135.6) .. (390,125) ;

% Text Node
\draw (195,132.4) node [anchor=north west][inner sep=0.75pt]    {$K_{00}$};
% Text Node
\draw (273,132.4) node [anchor=north west][inner sep=0.75pt]    {$K_{10}$};
% Text Node
\draw (36,197.4) node [anchor=north west][inner sep=0.75pt]    {$K$};
% Text Node
\draw (116,197.4) node [anchor=north west][inner sep=0.75pt]    {$K'$};
% Text Node
\draw (195,260.4) node [anchor=north west][inner sep=0.75pt]    {$K_{01}$};
% Text Node
\draw (273,260.4) node [anchor=north west][inner sep=0.75pt]    {$K_{11}$};
% Text Node
\draw (353,205.4) node [anchor=north west][inner sep=0.75pt]    {$K_{11} '$};

\end{tikzpicture}
    \caption{Local pictures of diagrams $K$, $K'$ and its resolution $K_{00}, K_{10}, K_{01}$, and $K_{11}$. Each vertical line denotes the axis. $K_{11}'$ is a diagram obtained from $K_{11}$ by removing the circle.}
    \label{fig:typeC_diag}
\end{figure}
We connect the equivalent diagrams $K_{10}$, $K_{01}$, and $K_{11}'$ via isotopy, denoting the isomorphisms between the complexes by $J_l$ and $J_r$ as follows:
\[
\begin{tikzcd}
    \bn{K_{10}}\arrow[r,"J_r",shift left]&\bn{K_{11}'}\arrow[l,"J_r^{-1}",shift left]\arrow[r,"J_l^{-1}",shift left]&\bn{K_{01}}\arrow[l,"J_l",shift left]
\end{tikzcd}.
\]
Define the two maps $\bn{K} \xrightleftharpoons[g]{f} \bn{K'}$ by $f = \begin{bmatrix} \id & 0 & 0 & \alpha \end{bmatrix}$ and $g = \begin{bmatrix} D & 0 & 0 & \iota S \end{bmatrix}^T$, where we decompose $\bn{K}$ as $\bn{K_{00}}\oplus\bn{K_{10}}\oplus\bn{K_{01}}\oplus\bn{K_{11}}$.
These are illustrated in \Cref{fig:typeC_map}.
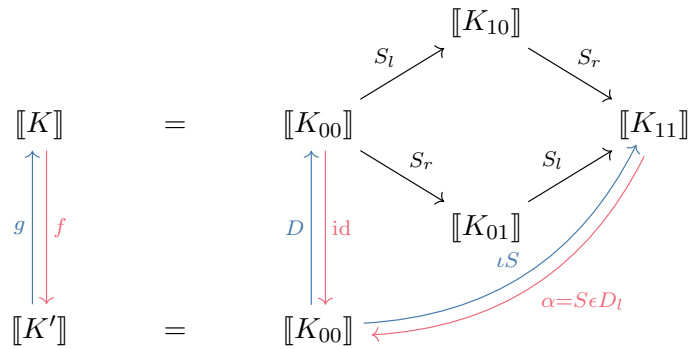
\begin{figure}
    \centering
    \begin{tikzcd}
                                            &           &                   &\bn{K_{10}}\arrow[rd,"S_r"]   & \\
    \bn{K}\arrow[dd,"f",shift left, red]   &=    & \bn{K_{00}}\arrow[ru,"S_l"]\arrow[rd,"S_r"]\arrow[dd,"\id",shift left, red]&              &\bn{K_{11}}\arrow[ddll,"\alpha=S\epsilon D_l",bend left,shift left,red]     \\
                                            &           &                   &\bn{K_{01}}\arrow[ru,"S_l"]   &  \\
    \bn{K'}\arrow[uu,"g",shift left, blue]   &=   &\bn{K_{00}}\arrow[uu,"D",shift left,blue]\arrow[uurr, "\iota S",shift left, bend right,blue]&              &
\end{tikzcd}
    \caption{(Type C) Two chain maps $\bn{K} \xrightleftharpoons[g]{f} \bn{K'}$.}
    \label{fig:typeC_map}
\end{figure}
The subscripts $l$ and $r$ on $S$ and $D$ indicate whether the local map is applied to the left or right of the axis.
The maps $\iota \colon \bn{K_{11}'} \to \bn{K_{11}}$ and $\epsilon \colon \bn{K_{11}} \to \bn{K_{11}'}$ represent the birth and death of the circle, respectively (see \Cref{fig:TQFT_gen1}).
The map $\alpha \coloneqq S \epsilon D_l$ is shown in \Cref{fig:alpha}.
Since $\alpha$ is symmetric, $f$ and $g$ commute with $I_\tau$.
\begin{figure}
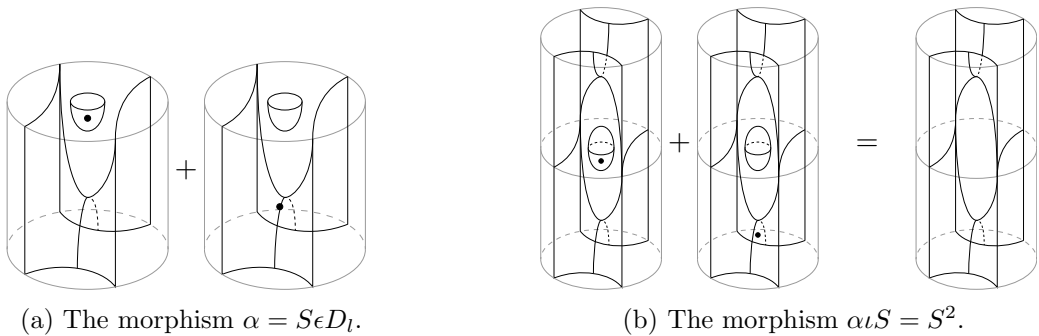

    \centering
    \begin{subfigure}{.48\linewidth}
        \centering
        \input{eq/alpha.tex}
        \caption{The morphism $\alpha=S\epsilon D_l$.}
        \label{fig:alpha}
    \end{subfigure}
    \hfill
    \begin{subfigure}{.48\linewidth}
        \centering
        \input{eq/alphaiS.tex}
        \caption{The morphism $\alpha\iota S = S^2$.}
        \label{fig:alphaiS}
    \end{subfigure}
    \caption{Morphisms $\alpha$ and $\alpha\iota S$.}
\end{figure}

Unlike the previous cases, the compositions $gf$ and $fg$ are not symmetric.
We first check the composition $fg \colon \bn{K'} \to \bn{K'}$, computed as $(\id \circ D) + (\alpha \circ \iota S)$.
As shown in \Cref{fig:alphaiS}, the relations (S) and (S$_\bullet$) imply that $\alpha \iota S = S^2$.
Thus, $fg = D + S^2$, which is exactly $H\cdot \id$.
In this case, we may simply set the homotopies $h'$ and $k'$ to zero.

Now consider $gf \colon \bn{K} \to \bn{K}$:
\[gf=
\begin{bmatrix}
    D&\gz&\gz&D \alpha\\
    \gz&\gz&\gz&\gz\\
    \gz&\gz&\gz&\gz\\
    \iota S&\gz&\gz&\iota S \alpha
\end{bmatrix},
\]
as in \Cref{fig:typeC_gf}.
\begin{figure}
    \centering
    \begin{tikzcd}
    &                                               &&\bn{K_{10}}&&\\
    \bn{K}&=&\bn{K_{00}}&&&\bn{K_{11}}\\
    &                                               &&&\bn{K_{01}}&\\
    &                                               &&&&\\
    &                                               &&&&\\
    &                                               &&\bn{K_{10}}&&\\
    \bn{K}&=                                      &\bn{K_{00}}&&&\bn{K_{11}}\\
    &                                               &&&\bn{K_{01}}&
    \arrow[from=2-1,to=7-1,"gf",red]\arrow[from=6-4,to=7-6,"S_r"]
    \arrow[from=2-3,to=7-6,"\iota S",red,near start]\arrow[from=2-6,to=7-3,"D\alpha=0",red,dotted,near start]
    \arrow[from=2-3,to=1-4,"S_l"]\arrow[from=1-4,to=2-6,"S_r"]\arrow[from=2-3,to=3-5,"S_r",crossing over]\arrow[from=3-5,to=2-6,"S_l"]
    \arrow[from=2-3,to=7-3,"D",red]\arrow[from=2-6,to=7-6,"\iota S\alpha",red]
    \arrow[from=7-3,to=6-4,"S_l"]\arrow[from=7-3,to=8-5,"S_r"]\arrow[from=8-5,to=7-6,"S_l"]
\end{tikzcd}
    \caption{(Type C) $gf$ is drawn by red arrows.}
    \label{fig:typeC_gf}
\end{figure}

Among nontrivial entries, the map $D\alpha \colon \bn{K_{11}}\to\bn{K_{00}}$ is $D\alpha = (DS)\epsilon D_l = 0$.
The map
\[
gf + H\cdot \id=
\begin{bmatrix}
    D+H\cdot\id&\gz&\gz&\gz\\
    \gz&H\cdot\id&\gz&\gz\\
    \gz&\gz&H\cdot\id&\gz\\
    \iota S&\gz&\gz&\iota S \alpha+H\cdot\id
\end{bmatrix}
\]
is represented by red arrows in \Cref{fig:typeC_h}.
Take a homotopy $h$ between $gf$ and $H\cdot\id$ as green dashed arrows in \Cref{fig:typeC_h}.
In a matrix form,
\[h=
\begin{bmatrix}
    \gz&S_l&\gz&\gz\\
    \gz&\gz&\gz&\epsilon D_l\\
    \gz&\gz&\gz&S_l\\
    \gz&\gz&\iota&\gz
\end{bmatrix}.
\]
\begin{figure}
    \centering
    \resizebox{.95\textwidth}{!}{\begin{tikzcd}
    &                                               &&&\bn{K_{10}}&&\\
    \bn{K}&=                                        &\bn{K_{00}}&&&&&&\bn{K_{11}}\\
    &                                               &&&&&\bn{K_{01}}&\\
    &                                               &&&&\\
    &                                               &&&&\\
    &                                               &&&\bn{K_{10}}&&\\
    \bn{K}&=                                      &\bn{K_{00}}&&&&&&\bn{K_{11}}\\
    &                                               &&&&&\bn{K_{01}}&
    \arrow[from=2-1,to=7-1,"gf+H\cdot\id",red]\arrow[from=1-5,to=6-5,"H\cdot\id",red]\arrow[from=2-9,to=6-5,"J_r^{-1}\epsilon D_l" description,dashed,green,near start]\arrow[from=6-5,to=7-9,"S_r"]\arrow[from=1-5,to=7-3,"S_l" description,green,dashed]
    \arrow[from=2-3,to=7-9,"\iota S",red,crossing over]
    \arrow[from=2-3,to=1-5,"S_l"]\arrow[from=1-5,to=2-9,"S_r"]\arrow[from=2-3,to=3-7,"S_r",crossing over]\arrow[from=3-7,to=2-9,"S_l"]
    \arrow[from=2-3,to=7-3,"D+H\cdot\id"',red]\arrow[from=3-7,to=8-7,"H\cdot\id",red,crossing over]\arrow[from=2-9,to=7-9,"\iota S\alpha + H\cdot\id",red]
    \arrow[from=3-7,to=7-9,"\iota J_l" description,dashed,green,near start,crossing over]\arrow[from=2-9,to=8-7,"S_l" description, dashed,green,near start,crossing over]
    \arrow[from=7-3,to=6-5,"S_l"]\arrow[from=7-3,to=8-7,"S_r"]\arrow[from=8-7,to=7-9,"S_l"]
\end{tikzcd}}
    \caption{(Type C) Red arrow depicts the map $gf+H\cdot\id$, and green dashed arrow depicts the homotopy $h$.}
    \label{fig:typeC_h}
\end{figure}
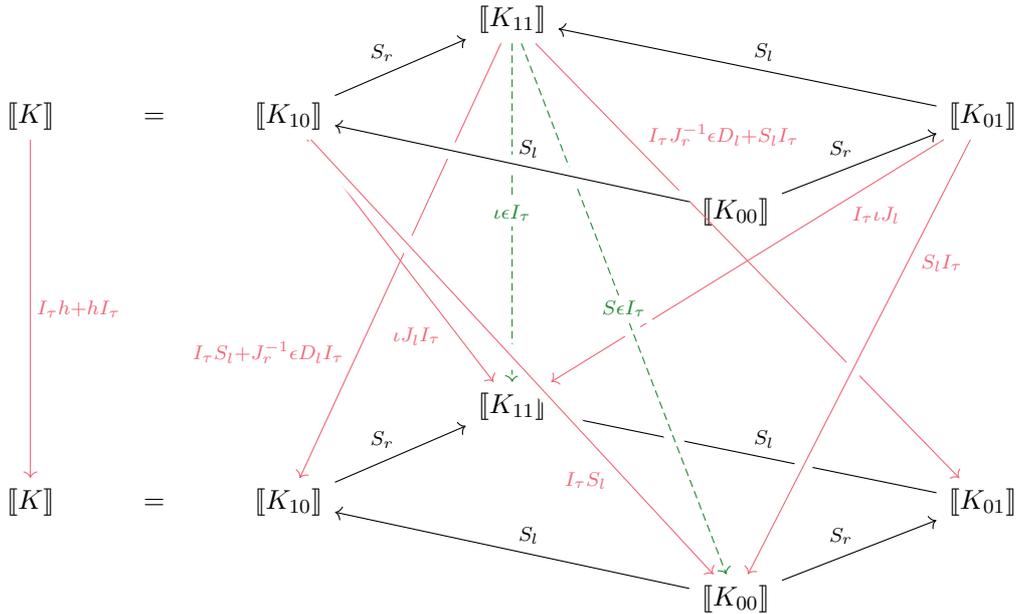
\begin{figure}
    \centering
    \resizebox{.9\textwidth}{!}{\begin{tikzcd}
    &                                               &&&\bn{K_{11}}&&\\
    \bn{K}&=                                        &\bn{K_{10}}&&&&&&\bn{K_{01}}\\
    &                                               &&&&&\bn{K_{00}}&\\
    &                                               &&&&\\
    &                                               &&&&\\
    &                                               &&&\bn{K_{11}}&&\\
    \bn{K}&=                                        &\bn{K_{10}}&&&&&&\bn{K_{01}}\\
    &                                               &&&&&\bn{K_{00}}&
    \arrow[from=8-7,to=7-3,"S_l"']\arrow[from=8-7,to=7-9,"S_r"]\arrow[from=7-3,to=6-5,"S_r"]\arrow[from=7-9,to=6-5,"S_l"']
    %%%%
    \arrow[from=1-5,to=7-3,"I_\tau S_l+J_r^{-1}\epsilon D_lI_\tau"',red,near end]\arrow[from=1-5,to=7-9,"I_\tau J_r^{-1}\epsilon D_l + S_lI_\tau",red,near start]
    \arrow[from=2-3,to=6-5,"\iota J_lI_\tau"',red,near end]\arrow[from=2-9,to=6-5,"I_\tau\iota J_l",red,near start]
    %%%%
    \arrow[from=1-5,to=6-5,"\iota\epsilon I_\tau" description,dashed,green,crossing over]
    \arrow[from=1-5,to=8-7,"S\epsilon I_\tau" description,dashed,green,crossing over]
    %%%%
    \arrow[from=2-3,to=8-7,"I_\tau S_l"',red,crossing over,near end]
    \arrow[from=2-9,to=8-7,"S_l I_\tau",red,crossing over,near start]
    %%%%
    \arrow[from=3-7,to=2-3,"S_l"',crossing over]\arrow[from=3-7,to=2-9,"S_r",crossing over]\arrow[from=2-3,to=1-5,"S_r",crossing over]\arrow[from=2-9,to=1-5,"S_l"',crossing over]
    \arrow[from=2-1,to=7-1,"I_\tau h+hI_\tau",red]
\end{tikzcd}}
    \caption{(Type C) Red arrow depicts the map $I_\tau h+hI_\tau$, and green dashed arrow depicts the homotopy $h$. Caution: the complex $\bn{K}$ is represented in a different angle than previous diagrams.}
    \label{fig:typeC_k}
\end{figure}
While the homotopy $h$ is formally obtained by composing the homotopies from a Reidemeister move II with Alishahi's map $h_0$, we provide the direct computation here for clarity.
We have to show
\begin{align*}
    gf+H\cdot\id 
    &= \begin{bmatrix}
    D+H\cdot\id&\gz&\gz&\gz\\
    \gz&H\cdot\id&\gz&\gz\\
    \gz&\gz&H\cdot\id&\gz\\
    \iota S&\gz&\gz&\iota S \alpha+H\cdot\id
\end{bmatrix}\\
&=
\begin{bmatrix}
    S_l^2&\gz&\gz&\gz\\
    \gz&J_r^{-1}\epsilon D_lS_r+S_l^2&J_r^{-1}\epsilon D_lS_l&\gz\\
    \gz&\gz&S_l^2&\gz\\
    \iota J_lS_r&\gz&\gz&S_rJ_r^{-1}\epsilon D_l+S_l^2
\end{bmatrix}
=dh+hd.
\end{align*}
Let's check it entrywise.
\begin{itemize}
    \item $(1,1)$: Holds by the (NC) relation.
    \item $(2,2)$: By isotopy, $J_r^{-1}\epsilon D_l S_r = D_l$. Thus, $J_r^{-1}\epsilon D_l S_r + S_l^2 = D_l + (D_l + H\cdot \id) = H\cdot \id$.
    \item $(2,3)$: Holds since $D_l S_l = 0$.
    \item $(3,3)$: $S_l^2 = D_l + H\cdot \id$. In $K_{01}$, $D_l$ is applied to a single connected strand, so $D_l = 0$.
    \item $(4,1)$: Holds by isotopy.
    \item $(4,4)$: We must show $\iota S^2 \epsilon D_l + H\cdot \id = S_r J_r^{-1} \epsilon D_l + S_l^2$. From $S_l^2 = D_l + H\cdot \id$, this reduces to $(\iota S^2 \epsilon + S_r J_r^{-1} \epsilon + \id)D_l = 0$. By the (4-Tu) relation (see \Cref{fig:h44}), the term in parentheses is $\iota J_l S_l$, and $(\iota J_l S_l)D_l = \iota J_l (S_l D_l) = 0$.
\end{itemize}
\begin{figure}
    \centering
    \input{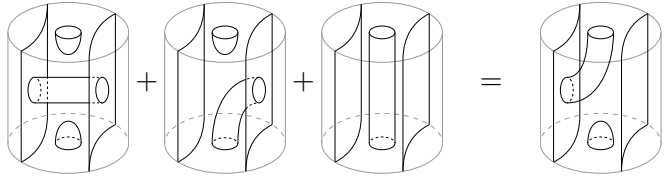}
    \caption{$\iota S^2\epsilon+S_rJ_r^{-1}\epsilon+\id=\iota J_lS_l$.}
    \label{fig:h44}
\end{figure}

Finally, we find a homotopy $k$ such that $I_\tau h + h I_\tau = dk + kd$. We define $k$ as
\[ 
k=
\begin{bmatrix}
    \gz&\gz&\gz&S\epsilon I_\tau\\
    \gz&\gz&\gz&\gz\\
    \gz&\gz&\gz&\gz\\
    \gz&\gz&\gz&\iota\epsilon I_\tau
\end{bmatrix},
\]
and it is represented as green dashed arrows in \Cref{fig:typeC_k}.
Let's verify the computation:
\begin{align*}
I_\tau h+hI_\tau &=
\begin{bmatrix}
    \gz&I_\tau S_l&S_lI_\tau&\gz\\
    \gz&\gz&\gz&I_\tau S_l+J_r^{-1}\epsilon D_lI_\tau\\
    \gz&\gz&\gz&I_\tau J_r^{-1}\epsilon D_l+S_lI_\tau\\
    \gz&\iota J_l I_\tau&I_\tau\iota J_l&\gz
\end{bmatrix}
\\&=
\begin{bmatrix}
    \gz&S\epsilon I_\tau S_r&S\epsilon I_\tau S_l&\gz\\
    \gz&\gz&\gz&S_lS\epsilon I_\tau\\
    \gz&\gz&\gz&S_rS\epsilon I_\tau\\
    \gz&\iota\epsilon I_\tau S_r&\iota\epsilon I_\tau S_l&\gz
\end{bmatrix}=dk+kd.
\end{align*}
For entry $(2,4)$, $I_\tau S_l + J_r^{-1} \epsilon D_l I_\tau = (S_r + J_r^{-1} \epsilon D_l) I_\tau$.
As shown in \Cref{fig:h24}, $S_r + J_r^{-1} \epsilon D_l = S_l S \epsilon$.
The entry $(3,4)$ is similar.
Other entries follow from $I_\tau S_l = S_r I_\tau$ and isotopy.
\begin{figure}
    \centering
    \input{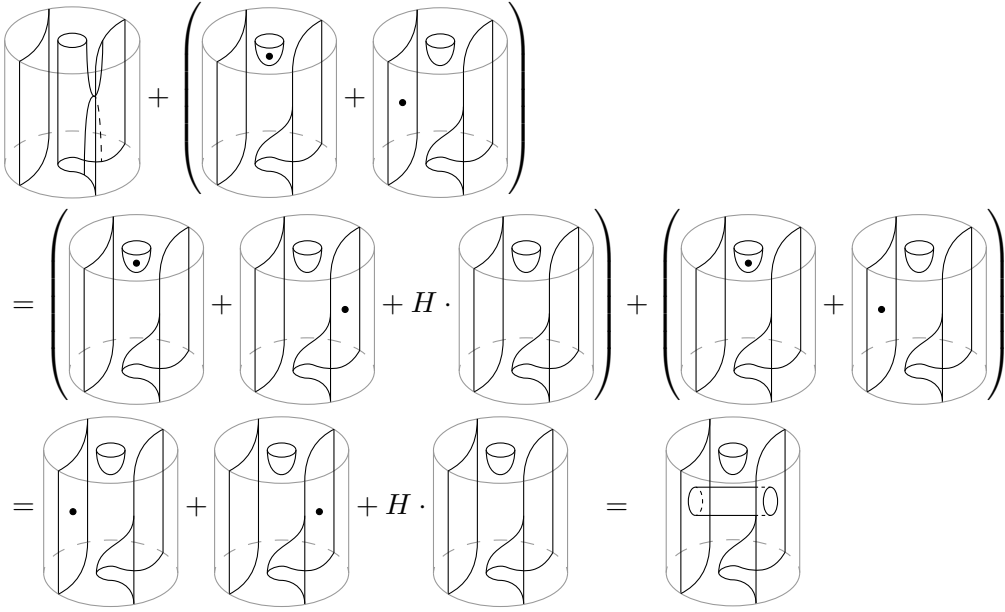}
    \caption{$S_r+J_r^{-1}\epsilon D_l = S_lS\epsilon$.}
    \label{fig:h24}
\end{figure}

\begin{rmk}
    In this paper, we focus on using transvergent diagrams to construct involutive Bar-Natan homology. Alternatively, one could define involutive Bar-Natan homology using intravergent diagrams (i.e., projections onto the $xz$-plane where $\mathrm{Fix}(\tau)$ is represented by a single point).
    Since it was shown in \cite{Chen-Yang:2026} that these two models of involutive Bar-Natan homology are essentially equivalent, one could provide an alternative proof for the Type C case by shifting to the intravergent setting. In that context, the Type C equivariant crossing change is represented by a single crossing change, allowing for an alternative proof analogous to the Type B case.
\end{rmk}

    \section{Examples}\label{sec:ex}

We computed the equivariant torsion order $\eqord(K)$ for all prime knots $K$ with ordinary crossing number up to $9$ that appear in Sakuma's table of strongly invertible knots \cite[Appendix]{Sakuma:1986}.
Among these, we identified five strongly invertible knots where $u(K)=1$ but $\eqord(K)=2$, and therefore, in particular, $u(K) < \equ(K)$.
We summarize our findings in the following propositions.
Note that a given knot may admit more than one strong involution $\tau \colon S^3 \to S^3$.
The additional subscripts $a$ and $b$ indicate specific symmetries.

\begin{prop}\label{prop:mainex}
    For the strongly invertible knots $K \in \{7_{7b}, 8_{21a}\}$, we have $u(K) = 1$ and $\equ(K) = 2$. For $K \in \{9_{28a}, 9_{34}, 9_{39}\}$, we find $u(K) = 1$ and $2 \leq \equ(K) \leq 3$.
\end{prop}

\begin{prop}\label{prop:diff_equ}
    We have $\equ(7_{7a}) = \equ(8_{21b}) = 1$, whereas $\equ(7_{7b})=\equ(8_{21a})=2$. In particular, the same underlying knot $K\in\{7_7, 8_{21}\}$ can yield distinct equivariant unknotting numbers depending on the choice of involution.
\end{prop}

These five knots are illustrated in \Cref{fig:equ_examples}.

\begin{rmk}
    \Cref{prop:mainex} and \Cref{prop:diff_equ} can also be proved without our method. For a strongly invertible knot $(K,\tau)$, consider the images $q(K)$ and $q(\mathrm{Fix}(\tau))$ under the quotient map $q\colon S^3 \to S^3/\tau \cong S^3$. Here $q(\mathrm{Fix}(\tau))$ is the unknot and $q(K)$ is an arc with endpoints on $q(\mathrm{Fix}(\tau))$. By taking the union $q(K)$ with either one of connected component of ${q(\mathrm{Fix}(\tau)) \setminus q(\mathrm{Fix}(\tau)\cap K)}$, one obtains two \emph{quotient knots} $q_1(K)$, $q_2(K)$ in $S^3$.
    
    In \cite{Boyle-Chen:2026}, the authors defined $u_X(K)$ as the minimal number of crossing changes required to transform a given strongly invertible knot $K$ into the unknot, using only equivariant crossing changes of Type $X\in\{A,B,C\}$. They provide lower bounds for each $u_X(K)$ in terms of the quotient knots:
    \begin{itemize}
        \item \cite[Theorem 1.9]{Boyle-Chen:2026} $u_A(K) \geq \max(u(q_1(K)), u(q_2(K)))$,
        \item \cite[Theorem 1.10]{Boyle-Chen:2026} $u_B(K) \geq u_4(q_1(K)) + u_4(q_2(K))$,
        \item \cite[Theorem 1.11]{Boyle-Chen:2026} $u_C(K) \geq u_{nb}(q_1(K)) + u_{nb}(q_2(K))$.
    \end{itemize}
    Here $u_4(K)$ denotes the $4$-move unknotting number of a knot $K$ and $u_{nb}(K)$ is the non-orientable band unknotting number of a knot $K$ (see \cite{Boyle-Chen:2026} for precise definitions). For each $K \in \{7_{7b}, 8_{21a}, 9_{28a}, 9_{34}, 9_{39}\}$, we verified that both quotient knots $q_1(K)$ and $q_2(K)$ are nontrivial. Consequently, $u_B(K)$ and $u_C(K)$ are both at least $2$, and therefore $\equ(K)\geq2$.
\end{rmk}

\begin{rmk}
    These are not the first known examples of strongly invertible knots for which $u(K) < \equ(K)$. Indeed, in \cite{Boyle-Chen:2026} it was shown that $\equ(T_3 \# T_3) > u(T_3 \# T_3) = 2$, where $T_3$ is the $3$-twist knot.
\end{rmk}

\begin{rmk}
    The knot $K=8_{21a}$ has $\ls(K) = -4$ and $\us(K) = -2$. As noted in \Cref{subsec:equ}, the inequality $\equ(K) \geq 2$ for this specific knot can also be deduced from these equivariant Rasmussen invariants.
\end{rmk}

\begin{figure}
    \centering
    \includegraphics[width=\textwidth]{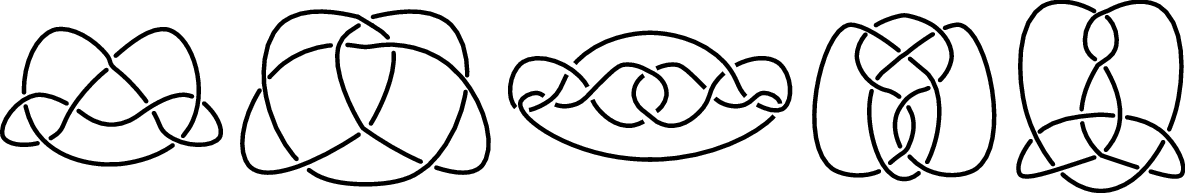}
    \caption{Examples of prime knots $K$ with $u(K) = 1 < 2 \leq \equ(K)$. These knots are $7_{7b}$, $8_{21a}$, $9_{28a}$, $9_{34}$, and $9_{39}$, in sequence.}
    \label{fig:equ_examples}
\end{figure}

    \section*{Appendix: table of invariants}
We have tabulated the following invariants for prime, strongly invertible knots with crossing numbers up to $9$: 
(1) the maximal $H$-torsion order $\ord(K)$, 
(2) the Rasmussen invariant $s(K)$, 
(3) the equivariant Rasmussen invariants $\us(K)$ and $\ls(K)$, and 
(4) the equivariant maximal $H$-torsion order $\eqord(K)$. 
Additionally, we include the unknotting number $u(K)$ sourced from KnotInfo \cite{knotinfo}.

The lower bound for the equivariant unknotting number $\equ(K)$ obtained in this appendix is 
\[\max\left\{u(K), |\us(K)|/2, |\ls(K)|/2, \eqord(K)\right\}.\]
We also establish an upper bound for $\equ(K)$ by explicitly constructing equivariant unknotting sequences using transvergent knot diagrams.

To compute $\us(K)$, $\ls(K)$, $\ord(K)$, and $\eqord(K)$, we utilized Taketo Sano's program \cite{Sano:YUI}.
For instance, the equivariant Bar-Natan homology $\eqBN{7_{7b}}$ can be obtained with the following command:
\begin{center}
    \texttt{ykh khi [[2,9,3,10], [4,2,5,1], [6,13,7,14], [8,3,9,4], [10,6,11,5], [12,7,13,8], [14,12,1,11]] -t F2 -c H}
\end{center}
where the input list represents a planar diagram code for a transvergent diagram of $7_{7b}$.

Note that we follow Sakuma's naming convention for strongly invertible knots; see \cite[Appendix]{Sakuma:1986} for further details.
\newgeometry{left=1in,right=1in}
\twocolumn
\begin{center}

\newcolumntype{u}{>{\columncolor[HTML]{D3D3D3}}c}

\tablefirsthead{%
\hline
Name&$s$&$\ord$&$u$&$\ls$&$\us$&$\eqord$&$\equ$ \\ \hline
}
\tablehead{%
\hline
Name&$s$&$\ord$&$u$&$\ls$&$\us$&$\eqord$&$\equ$ \\ \hline
}

\tabletail{%
\hline
}
\tablelasttail{%
\hline
}
\begin{supertabular}{|c||c|c|u|c|c|c|u|}
$3_1$&-2&1&1&-2&-2&1&1 \\ \hline
$4_1$&0&1&1&0&0&1&1 \\ \hline
$5_1$&-4&1&2&-4&-4&1&2 \\ \hline
$5_{2a}$&-2&1&1&-2&-2&1&1 \\ \hline
$5_{2b}$&-2&1&1&-2&-2&1&1 \\ \hline
$6_{1a}$&0&1&1&0&0&1&1 \\ \hline
$6_{1b}$&0&1&1&0&0&1&1 \\ \hline
$6_{2a}$&-2&1&1&-2&-2&1&1 \\ \hline
$6_{2b}$&-2&1&1&-2&-2&1&$[1,2]$ \\ \hline
$6_3$&0&1&1&0&0&1&1 \\ \hline
$7_1$&-6&1&3&-6&-6&1&3 \\ \hline
$7_{2a}$&-2&1&1&-2&-2&1&1 \\ \hline
$7_{2b}$&-2&1&1&-2&-2&1&1 \\ \hline
$7_{3a}$&4&1&2&4&4&1&2 \\ \hline
$7_{3b}$&4&1&2&4&4&1&2 \\ \hline
$7_{4a}$&2&1&2&2&2&1&2 \\ \hline
$7_{4b}$&2&1&2&2&2&1&2 \\ \hline
$7_{5a}$&-4&1&2&-4&-4&1&2 \\ \hline
$7_{5b}$&-4&1&2&-4&-4&1&2 \\ \hline
$7_{6a}$&-2&1&1&-2&-2&1&1 \\ \hline
$7_{6b}$&-2&1&1&-2&-2&1&$[1,2]$ \\ \hline
$7_{7a}$&0&1&1&0&0&1&1 \\ \hline
\rowcolor{yellow}$7_{7b}$&0&1&1&0&0&2&2 \\ \hline
$8_{1a}$&0&1&1&0&0&1&1 \\ \hline
$8_{1b}$&0&1&1&0&0&1&1 \\ \hline
$8_{2a}$&-4&1&2&-4&-4&1&2 \\ \hline
$8_{2b}$&-4&1&2&-4&-4&1&$[2,3]$ \\ \hline
$8_3$&0&1&2&0&0&1&2 \\ \hline
$8_{4a}$&-2&1&2&-2&-2&1&2 \\ \hline
$8_{4b}$&-2&1&2&-2&-2&1&2 \\ \hline
$8_{5a}$&4&1&2&4&4&2&$[2,3]$ \\ \hline
$8_{5b}$&4&1&2&4&4&1&$[2,3]$ \\ \hline
$8_{6a}$&-2&1&2&-2&-2&1&2 \\ \hline
$8_{6b}$&-2&1&2&-2&-2&1&2 \\ \hline
$8_{7a}$&2&1&1&2&2&1&1 \\ \hline
$8_{7b}$&2&1&1&2&2&1&$[1,2]$ \\ \hline
$8_{8a}$&0&1&2&0&0&1&2 \\ \hline
$8_{8b}$&0&1&2&0&0&1&2 \\ \hline
$8_9$&0&1&1&0&0&1&1 \\ \hline
$8_{10}$&2&1&2&2&2&1&2 \\ \hline
$8_{11a}$&-2&1&1&-2&-2&1&1 \\ \hline
$8_{11b}$&-2&1&1&-2&-2&1&$[1,2]$ \\ \hline
$8_{12}$&0&1&2&0&0&1&2 \\ \hline
$8_{13a}$&0&1&1&0&0&1&1 \\ \hline
$8_{13b}$&0&1&1&0&0&1&$[1,2]$ \\ \hline
$8_{14a}$&2&1&1&2&2&1&$[1,2]$ \\ \hline
$8_{14b}$&2&1&1&2&2&1&$[1,2]$ \\ \hline
$8_{15a}$&-4&1&2&-4&-4&2&2 \\ \hline
$8_{15b}$&-4&1&2&-4&-4&1&2 \\ \hline
$8_{16}$&-2&1&2&-2&-2&1&2 \\ \hline
$8_{18a}$&0&1&2&0&0&2&2 \\ \hline
$8_{18b}$&0&1&2&0&0&2&2 \\ \hline
$8_{19}$&6&2&3&6&6&2&3 \\ \hline
$8_{20}$&0&1&1&0&0&1&1 \\ \hline
\rowcolor{yellow}$8_{21a}$&-2&1&1&-4&-2&2&2 \\ \hline
$8_{21b}$&-2&1&1&-2&-2&1&1 \\ \hline
$9_1$&-8&1&4&-8&-8&1&4 \\ \hline
$9_{2a}$&-2&1&1&-2&-2&1&1 \\ \hline
$9_{2b}$&-2&1&1&-2&-2&1&1 \\ \hline
$9_{3a}$&6&1&3&6&6&1&3 \\ \hline
$9_{3b}$&6&1&3&6&6&1&3 \\ \hline
$9_{4a}$&-4&1&2&-4&-4&1&2 \\ \hline
$9_{4b}$&-4&1&2&-4&-4&1&2 \\ \hline
$9_{5a}$&2&1&2&2&2&1&2 \\ \hline
$9_{5b}$&2&1&2&2&2&1&2 \\ \hline
$9_{6a}$&-6&1&3&-6&-6&1&3 \\ \hline
$9_{6b}$&-6&1&3&-6&-6&1&3 \\ \hline
$9_{7a}$&-4&1&2&-4&-4&1&2 \\ \hline
$9_{7b}$&-4&1&2&-4&-4&1&2 \\ \hline
$9_{8a}$&-2&1&2&-2&-2&1&2 \\ \hline
$9_{8b}$&-2&1&2&-2&-2&1&2 \\ \hline
$9_{9a}$&-6&1&3&-6&-6&1&3 \\ \hline
$9_{9b}$&-6&1&3&-6&-6&1&3 \\ \hline
$9_{10a}$&4&1&3&4&4&1&3 \\ \hline
$9_{10b}$&4&1&3&4&4&1&3 \\ \hline
$9_{11a}$&4&1&2&4&4&1&$[2,3]$ \\ \hline
$9_{11b}$&4&1&2&4&4&1&2 \\ \hline
$9_{12a}$&-2&1&1&-2&-2&1&1 \\ \hline
$9_{12b}$&-2&1&1&-2&-2&1&$[1,2]$ \\ \hline
$9_{13a}$&4&1&3&4&4&1&3 \\ \hline
$9_{13b}$&4&1&3&4&4&1&3 \\ \hline
$9_{14a}$&0&1&1&0&0&1&1 \\ \hline
$9_{14b}$&0&1&1&0&0&1&1 \\ \hline
$9_{15a}$&2&1&2&2&2&1&2 \\ \hline
$9_{15b}$&2&1&2&2&2&1&2 \\ \hline
$9_{16a}$&6&1&3&6&6&2&3 \\ \hline
$9_{16b}$&6&1&3&6&6&1&3 \\ \hline
$9_{17a}$&-2&1&2&-2&-2&1&2 \\ \hline
$9_{17b}$&-2&1&2&-2&-2&2&2 \\ \hline
$9_{18a}$&-4&1&2&-4&-4&1&2 \\ \hline
$9_{18b}$&-4&1&2&-4&-4&1&2 \\ \hline
$9_{19a}$&0&1&1&0&0&1&$[1,2]$ \\ \hline
$9_{19b}$&0&1&1&0&0&1&$[1,2]$ \\ \hline
$9_{20a}$&-4&1&2&-4&-4&1&2 \\ \hline
$9_{20b}$&-4&1&2&-4&-4&1&$[2,3]$ \\ \hline
$9_{21a}$&2&1&1&2&2&1&$[1,2]$ \\ \hline
$9_{21b}$&2&1&1&2&2&1&$[1,2]$ \\ \hline
$9_{22}$&2&1&1&2&2&1&$[1,3]$ \\ \hline
$9_{23a}$&-4&1&2&-4&-4&1&2 \\ \hline
$9_{23b}$&-4&1&2&-4&-4&2&2 \\ \hline
$9_{24}$&0&1&1&0&0&1&$[1,2]$ \\ \hline
$9_{25}$&-2&1&2&-2&-2&1&2 \\ \hline
$9_{26a}$&2&1&1&2&2&1&1 \\ \hline
$9_{26b}$&2&1&1&2&2&1&$[1,2]$ \\ \hline
$9_{27a}$&0&1&1&0&0&1&$[1,2]$ \\ \hline
$9_{27b}$&0&1&1&0&0&1&1 \\ \hline
\rowcolor{yellow}$9_{28a}$&-2&1&1&-2&-2&2&$[2,3]$ \\ \hline
$9_{28b}$&-2&1&1&-2&-2&1&$[1,2]$ \\ \hline
$9_{29}$&-2&1&2&-2&-2&2&$[2,3]$ \\ \hline
$9_{30}$&0&1&1&0&0&1&$[1,3]$ \\ \hline
$9_{31a}$&2&1&2&2&2&1&2 \\ \hline
$9_{31b}$&2&1&2&2&2&2&2 \\ \hline
\rowcolor{yellow}$9_{34}$&0&1&1&0&0&2&2 \\ \hline
$9_{35a}$&-2&1&3&-2&-2&2&3 \\ \hline
$9_{35b}$&-2&1&3&-2&-2&2&3 \\ \hline
$9_{36}$&4&1&2&4&4&1&2 \\ \hline
$9_{37a}$&0&1&2&0&0&2&2 \\ \hline
$9_{37b}$&0&1&2&0&0&2&2 \\ \hline
$9_{38}$&-4&1&3&-4&-4&2&3 \\ \hline
\rowcolor{yellow}$9_{39}$&2&1&1&2&2&2&$[2,3]$ \\ \hline
$9_{40a}$&-2&1&2&-2&-2&2&2 \\ \hline
$9_{40b}$&-2&1&2&-2&-2&2&2 \\ \hline
$9_{41}$&0&1&2&0&0&2&2 \\ \hline
$9_{42}$&0&1&1&0&0&1&1 \\ \hline
$9_{43}$&-4&1&2&-4&-4&1&$[2,3]$ \\ \hline
$9_{44}$&0&1&1&0&0&1&1 \\ \hline
$9_{45}$&-2&1&1&-2&-2&1&1 \\ \hline
$9_{46a}$&0&1&2&-2&0&1&2 \\ \hline
$9_{46b}$&0&1&2&-2&0&1&2 \\ \hline
$9_{47}$&2&1&2&2&2&1&2 \\ \hline
$9_{48a}$&-2&1&2&-2&-2&1&2 \\ \hline
$9_{48b}$&-2&1&2&-2&-2&1&2 \\ \hline
$9_{49}$&4&1&3&4&4&1&3 \\

\end{supertabular}
\end{center}

\onecolumn
\restoregeometry

\printbibliography
\end{document}